 \documentclass[draft]{article}

\usepackage{amsmath,amsfonts,amsthm,amssymb,amscd,cancel,color,mathabx}
\usepackage{enumitem}
\usepackage{verbatim}
\usepackage[dvipdfmx]{graphicx}
\usepackage{ulem,color}

\usepackage{cleveref}

\renewcommand{\Re}{\mathop{\rm Re}\nolimits}

\def\S{\mathhexbox278}

\theoremstyle{plain}
\newtheorem{theorem}{Theorem}[section]
\newtheorem{lemma}[theorem]{Lemma}
\newtheorem{proposition}[theorem]{Proposition}
\newtheorem{corollary}[theorem]{Corollary}
\theoremstyle{definition}

\theoremstyle{remark}
\newtheorem{remark}[theorem]{Remark}
\newtheorem{notation}[theorem]{Notation}

\newcommand{\R}{{\mathbb R}}

\newcommand{\N}{{\mathbb N}}

\def\im{{\rm i}}

\newcommand{\C}{\mathbb{C}}

\def\({\left(}
\def\){\right)}
\def\<{\left\langle}
\def\>{\right\rangle}
\newcommand{\sech}{{\mathrm{sech}}}

\newcommand{\supp}{{\mathrm{supp}\ }}
    
\newcommand{\Span}{{\mathrm{Span}}}

\makeatletter

\numberwithin{equation}{section}

\begin{document}

\title{On stabilization at a soliton for generalized Korteweg--De Vries pure power equation for any power  $p\in (1,5)$ }

\author{Scipio Cuccagna, Masaya Maeda  }
\maketitle

\begin{abstract}   We apply our idea, which we previously used in the analysis of the pure power NLS, consisting of splitting the virial inequality method into a large energy inequality combined with Kato smoothing, to the case of generalized Korteweg--De Vries pure power equations.  We assume that a solution remains for all positive times very close to a soliton and then we prove an asymptotic stability result for $t\to +\infty$.
\end{abstract}

\section{Introduction}

We consider on the line $\R $  the   generalized Korteweg--de Vries  (gKdV) equation
\begin{align}\label{eq:gKdV1}&
   \partial _t  u  =- \partial_x \(  \partial_x ^2  u +  f(u) \)
\end{align}
where $f(u)=|u|^{p-1}u$ or $|u|^p$ for $1<p<5$.
Formally, \eqref{eq:gKdV1} is a Hamiltonian system with the Hamiltonian/Energy $\mathbf{E}$ and   mass $\mathbf{Q}$, where
  \begin{align}\label{eq:energy}
& \mathbf{E}( {u})=\frac{1}{2}   \| \partial_x u \| ^2 _{L^2\( \R \) } -\int_{\R} F(u)    \,dx,\     \text{ where }  F(u)=\int_0^u f(s)\,ds    \text {  and} \\&   \label{eq:mass} \mathbf{Q}( {u})=\frac{1}{2}   \| u  \| ^2 _{L^2\( \R \) }.
\end{align}
By  solutions of \eqref{eq:gKdV1}  we mean functions $u\in C^0([0,T^*),H^1(\R))$ satisfying the integral equation
    \begin{align}\label{gKdVintegral}
            u(t)=e^{-t\partial_x^3}u_0 - \int_0^t e^{-(t-s)\partial_x^3} \partial_x (f(u(s)))\,ds  \text{  for all $t\in[0,T^*)$,}
    \end{align}
for some $0<T^*\leq \infty$.

\begin{remark}\label{rem:regularity}
Consider the operator
$A=-\partial_x^3$ on $H^{-2}(\R)$ with domain
$D(A)=H^1(\R)$. The graph norm of $A$ is equivalent to the
$H^1$-norm. Moreover, since $u\in C^0([0,T^*),H^1(\R))$,
the Sobolev chain rule yields
    $\partial_x( f(u))
    \in
    C^0([0,T^*),L^2(\R))
    \subset
    C^0([0,T^*),H^{-2}(\R))$.
Consequently, \eqref{gKdVintegral} is precisely the
variation-of-constants formula in $H^{-2}(\R)$ for
\begin{align*}
    \partial_tu
    =
    Au-\partial_xf(u).
\end{align*}
By \cite[Corollary 4.1.8]{CazHarbook}, it follows that
    $u
    \in
    C^1([0,T^*),H^{-2}(\R))$
and that \eqref{eq:gKdV1} holds in $H^{-2}(\R)$ for every
$t\in[0,T^*)$. In particular, it holds in the distributional
sense on $(0,T^*)\times\R$.

If, in addition,
    $u\in C^0([0,T^*),H^3(\R))$,
the same argument, now $A$ on $L^2(\R)$ and
$D(A)=H^3(\R)$, gives
    $u\in C^1([0,T^*),L^2(\R))$,
and \eqref{eq:gKdV1} holds in $L^2(\R)$ for every
$t\in[0,T^*)$.
\end{remark}

\begin{remark} In Appendix \ref{app:A}
we show   conservation of $\mathbf{Q}$ (but not of $\mathbf{E}$) for the $u$ in \eqref{gKdVintegral}.
\end{remark}

If $f\in C^3$, that is, if $p \in [3,5)$ when $f(u)=|u|^{p-1}u$, or $p\in (3,5)\cup\{2\}$ when $f(u)=|u|^p$, one can show that \eqref{eq:gKdV1} is globally well-posed in $H^1$ with both $\mathbf{E}$ and $\mathbf{Q}$ conserved, by arguments similar to those in \cite{KPV93CPAM}.
However, our focus in this paper is on small values of $p$, particularly $p\in (1,2)$,  and we will analyze the properties of the solution under the assumption of its existence.
For example, when $p=\frac{3}{2}$, \eqref{eq:gKdV1} is called the Schamel equation which appears in the study of ion-acoustic waves \cite{Schamel73JPP}.

For   local wellposedness results in a subset of $H^1$, see \cite{LiMiPo2019,FRRSY23Nonlinearity}.
Also, by following the argument of \cite{CP14Nonlinearity}, one can construct a weak solution of \eqref{eq:gKdV1} in the class $L^\infty(\R,H^1)$ with $\mathbf{E}$ and $\mathbf{Q}$ non-increasing.


The most important solutions of the gKdV equation \eqref{eq:gKdV1}, which are also central to this paper, are the soliton solutions $u(t,x)=\phi_c(x-ct-\gamma )$, where $c>0$, $\gamma\in \R$ and
\begin{align*}
\phi_c  (x):=c ^{\frac 1{p-1}} \phi (\sqrt{c }x) \text{ , } \phi(x) =  {\(\frac {p+1}2 \)^{\frac 1{p-1}}}{
\sech   ^{\frac 2{p-1}}\(\frac{p-1}2 x\)}  .
\end{align*}
It is well known, see e.g.\cite{benjamin1972,bona1975,bonasougstrauss,GSS1,weinLyap1986}, that solitons are trapped by $\mathbf{E}$ and $\mathbf{Q}$.
\begin{proposition}\label{prop:orbstab}
    For any $c>0$, there exist $\delta_0>0$ and $C>0$ such that for any $u\in C([0,T^*),H^1)$, satisfying $\mathbf{E}(u(t))=\mathbf{E}(u(0))$, $\mathbf{Q}(u(t))=\mathbf{Q}(u(0))$ for all $t\in (0,T^*)$ and
    \begin{align}\label{eq:ini}
        \delta:=\inf_{\gamma\in \R}\|u(0)-\phi_c(\cdot-\gamma)\|_{H^1}\leq \delta_0,
    \end{align}
    will satisfy
    \begin{align}\label{eq:orbstab}
    \sup _{t\in[0,T^*)} \inf_{\gamma\in \R}  \|  u(t, \cdot   ) -\phi _c (\cdot -\gamma ) \| _{H^1}  \le C\delta,
  \end{align}
\end{proposition}


\begin{remark}
Proposition \ref{prop:orbstab} is a purely variational trapping
statement and does not use the evolution equation. This observation is
relevant below because the proof of Theorem \ref{thm:main} uses energy
conservation only to obtain the orbital bound \eqref{eq:orbstab}; the
latter may instead be assumed directly.
\end{remark}

From Proposition \ref{prop:orbstab}, if we consider a solution of \eqref{gKdVintegral} satisfying the conservation of $\mathbf{E}$,
which all solutions of \eqref{gKdVintegral} satisfy if $f\in C^3$, with initial data near a soliton, then the solution will always stay near the soliton as long as the solution exists.
Here, recall $\mathbf{Q}$ is always conserved, see Proposition \ref{prop:consQ}.
In this paper, we will study the behavior of such solutions.

\begin{theorem}
  \label{thm:main} Let $ a>0$. Then for any   $\epsilon >0$   there exists a  $\delta_0  >0$   such that for any $0<T^*\leq \infty$ and any
  solution $u\in C^0([0,T^*),H^1(\R))$ of  \eqref{gKdVintegral} which satisfies $\mathbf{E}(u(t))=\mathbf{E}(u_0)$ for all $t\in (0,T^*)$ and \eqref{eq:ini} with $c=1$,
  there exist  functions
  $ c \in C^1 ([0,T^*), (1-\epsilon,1+\epsilon) ),$ $\gamma \in C^1 ([0,T^*), \R ) $ and
  $v\in C^0 ([0,T^*),H^1(\R))$
     such that
  \begin{align} \label{eq:asstab1}
    & u(t)=    \phi _{c (t)} (\cdot -\gamma(t) ) + v(t, \cdot -\gamma(t) ) \text{   with}
 \\&  \label{eq:asstab2}   \int _0^{T^*} \|  e^{- a|x|}   v (t ) \| _{H^1(\R )}^2 dt <  \epsilon.  
\end{align}
Moreover, if $T^*=\infty$, there exists $c_+>0$ s.t.
\begin{align}
    \label{eq:asstab20}
   &\lim _{t\to +\infty  }  \|  e^{- a| x|}   v (t ) \| _{L^2(\R )}     =0\text{ and} \\
&\lim _{t\to +\infty}c (t)= c _+ . \label{eq:asstab3}
\end{align}
\end{theorem}

\begin{remark}\label{rem:scaling-reduction}
The restriction to $c=1$ in Theorem \ref{thm:main} is only a normalization.
More precisely, for every $c_0>0$, the conclusion of
Theorem \ref{thm:main} remains valid with \eqref{eq:ini} imposed at $c=c_0$
and with
    $c(t)\in \bigl(c_0(1-\epsilon),c_0(1+\epsilon)\bigr)$.
Indeed, given a solution $u$ of \eqref{eq:gKdV1}, define
\begin{align*}
    \widetilde u(s,y)
    :=
    c_0^{-\frac{1}{p-1}}
    u\left(
        c_0^{-\frac32}s,
        c_0^{-\frac12}y
    \right).
\end{align*}
Then, $\widetilde u$ also solves \eqref{eq:gKdV1}. Moreover,
    $c_0^{-\frac{1}{p-1}}
    \phi_{c_0}\left(c_0^{-\frac12}y\right)
    =
    \phi(y)$.
Suppose that the decomposition given by Theorem \ref{thm:main} for
$\widetilde u$ is
\begin{align*}
    \widetilde u(s,y)
    =
    \phi_{\widetilde c(s)}
    \bigl(y-\widetilde\gamma(s)\bigr)
    +
    \widetilde v
    \bigl(s,y-\widetilde\gamma(s)\bigr).
\end{align*}
Then the corresponding decomposition of $u$ is obtained by setting
\begin{align*}
    c(t)
    :=
    c_0\widetilde c\left(c_0^{\frac32}t\right),\quad
    \gamma(t)
    :=
    c_0^{-\frac12}
    \widetilde\gamma\left(c_0^{\frac32}t\right),\quad
    v(t,x)
    :=
    c_0^{\frac{1}{p-1}}
    \widetilde v\left(
        c_0^{\frac32}t,
        c_0^{\frac12}x
    \right).
\end{align*}
In addition, with
    $\widetilde T^*
    :=
    c_0^{\frac32}T^*$,
we have
\begin{align*}
    &\int_0^{T^*}
    \left\|
        e^{-a|x|}v(t)
    \right\|_{H^1_x}^2
    \,dt\leq
    C_{p,c_0}
    \int_0^{\widetilde T^*}
    \left\|
        e^{-\frac{a}{\sqrt{c_0}}|y|}
        \widetilde v(s)
    \right\|_{H^1_y}^2
    \,ds,
\end{align*}
where
    $C_{p,c_0}
    :=
    \max\left\{
        c_0^{\frac{2}{p-1}-2},
        c_0^{\frac{2}{p-1}-1}
    \right\}$.
Thus, applying Theorem \ref{thm:main} to $\widetilde u$ with the weight
$a/\sqrt{c_0}$ and with a sufficiently smaller value of the accuracy
parameter yields the assertion for the soliton $\phi_{c_0}$.
\end{remark}


\begin{remark}
\label{rem:udermain}
In the proof of Theorem~\ref{thm:main}, the conservation of
$\mathbf E$ is used only through Proposition~\ref{prop:orbstab}.
Thus, the assumption
    $\mathbf E(u(t))=\mathbf E(u_0)$
may be replaced by the orbital bound \eqref{eq:orbstab}.

We also comment on the energy-conservation assumption when
$p\in(1,2)$. If a solution satisfies
    $u\in C^0([0,T^*),H^3(\R))$,
then the standard Hamiltonian calculation is justified and
$\mathbf E(u(t))$ is conserved. The same conclusion holds if, on every
compact time interval, the solution is the strong
$C^0_tH^1_x$-limit of smooth energy-conserving solutions, since
$\mathbf E$ is continuous on $H^1(\R)$.

On the other hand, local well-posedness in $H^1(\R)$ is not presently
known for $p\in(1,2)$. To the best of our knowledge, neither a general
energy-conservation result for all $H^1$-solutions satisfying
\eqref{gKdVintegral}, nor a counterexample exhibiting loss of energy,
is known. Local solutions have been constructed in a smaller class of
weighted and more regular functions; see \cite{FRRSY23Nonlinearity,LiMiPo2019}.

We prove Theorem~\ref{thm:main} for $p\in(1,2)$. For $p\geq2$, the
same proof applies and can be simplified because the nonlinearity is
more regular.
\end{remark}

Our main result is inequality \eqref{eq:asstab2}, which is a surrogate of the virial inequality in  Lemma 2 in Martel and Merle
\cite{MaMeNonlinearity2005}.
  Thanks to the observation  in Remark \ref {rem:regularity}, if $u\in C^0([0,\infty),H^3(\R))$  it is possible to   follow basically verbatim the discussion
   in \S 3.1 in Martel and Merle
\cite{MaMeNonlinearity2005} which starts  from Step 2, and thanks to Theorem \ref{thm:main} prove  the following result.

\begin{corollary}
  \label{cor:main}  For the solutions in Theorem \ref{thm:main} which additionally satisfy  $u\in C^0([0,\infty),H^3(\R))$  the following limit holds,
   \begin{align*}
    \lim _{t\to +\infty }   \|  u(t, \cdot   ) -\phi _{c(t)} (\cdot -\gamma (t)) \| _{H^1( \{ x: x\ge \frac{t}{10}    \} )}         =0 .
  \end{align*}

\end{corollary}

In the proof of Theorem \ref{thm:main} we will assume $p\in (1,2)$ and $T^*=\infty$.  Since $H^1(\R)$ is an algebra, for $p\ge 2$ essentially the same arguments taken  verbatim   of the case $p<2$ hold. However,   for    $p\ge 2$ the proofs  could be made   simpler  thanks to the fact that the nonlinearity $u\to f(u)$ is more regular.  Even for some  powers  $p\ge 2$ our technique covers cases which appear not treated    by the very general result in Martel and Merle \cite{MaMegeneral2008}, which requires well posedness of the Cauchy problem.

\noindent There is an important literature about asymptotic convergence to solitons for gKdV equations starting from Pego and Weinstein \cite{PegoWei2} who proved asymptotic stability of solitons
in exponentially weighted spaces for the nonlinearity $f(u)=u^p$   with   $p=2,3,4$ and for $4<p<5$ in the case of $u$ positive. Pego and Weinstein \cite{PegoWei2}, using   their analysis of the  Evans functions in
  \cite{PegoWei1},  showed  also that the linearization, see \eqref{eq:Lin}, has no  eigenvalues different from 0 for almost all exponents $p$, see {Proposition} \ref{prop:PEGOWEI2} below.
 Mizumachi
\cite{MizuKdV2001} introduced  weights different from Pego and Weinstein \cite{PegoWei1},   part  exponentials and   part polynomials,  useful for the analysis of multisolitons.
The case  with $f(u)=u^2$ is the integrable KdV for  which soliton resolution, that is a quite complete classification of all solutions which live  in a more restricted space than $H^1(\R )$,     can be obtained using the Inverse Scattering Transform, see for example   Grunert and Teschl \cite{GrunertTeschl}.

\noindent  Other notable  results are due to Martel and Merle
\cite{MaMeARMA2001,MaMeNonlinearity2005,MaMegeneral2008,MaMeRefined2008}. In particular  Martel and Merle \cite{MaMeNonlinearity2005} contains a very direct   proof of asymptotic stability   for $p=2,3,4$   based on the positivity, see  \cite[Proposition 3]{MaMeNonlinearity2005} with the proof in \cite{MaMeARMA2001},
 of a quadratic forms which appears in the course of the virial inequality argument used to prove Lemma 2 in \cite{MaMeNonlinearity2005}. Later  Martel and Merle
\cite{MaMegeneral2008}   generalized considerably their result focusing on a dual problem introduced by Martel \cite{MartelLinear}.
Further, for $p=2$, Mizumachi and Tzvetkov \cite{MT14RIMS} proved the asymptotic stability for $L^2$-initial data by combining Pego-Weinstein and  Martel-Merle (virial) methods, see also \cite{Mizumachi09CMP}.
However, while  Martel and Merle
\cite{MaMegeneral2008}, as well as Mizumachi and Tzvetkov, holds  in contexts where  local well posedness is true in $H^1(\R )$ (resp.\ $L^2$), such local well posedness is not yet known for all our equations, especially when $p\in (1,2)$, as discussed by Linares et al.      \cite{LiMiPo2019}. Following  Martel and Merle, the  series  Kowalczyk et al. \cite{KMM2017,KMM2022,KMMV2021,KM22} splits  the virial inequality argument   in two parts, one involving first a high energy estimate and then, after a change of variables, a second, low energy, virial inequality.
An approach along the lines  of Kowalczyk et al.   has been used  for the asymptotic stability of the small amplitude  ground states of cubic quintic NLS's by Martel \cite{Martelcubquint1,Martelcubquint2} and  by Rialland   \cite{Riallandcuqui1,Riallandcuqui2} for more general small perturbations of the cubic NLS. In the series of papers   \cite{CM24D1,CM243,CM24D4}  devoted to the asymptotic stability of the ground states of the  pure NLS in dimension 1,  we  replaced  the  second  virial inequality with smoothing estimates.  In the present paper we apply this idea to gKdV \eqref{eq:gKdV1}.
We refer to    \cite{CM24D1,CM243,CM24D4} for a general discussion in the context of the NLS   of the advantages of our method, which subsist also
for gKdV  \eqref{eq:gKdV1},  as we  briefly explain now.   Proving dispersion by means of Duhamel's formula, e.g. \eqref{gKdVintegral} or  \eqref{eq:expv1}, is not easy or possible, especially  for $p$ is close to 1, when the nonlinearity becomes a long range perturbation of the linear equation. Hence we use virial inequalities, which are   akin to the  positive commutator method   in \cite[\S B p.157]{ReedSimon4} and are motivated by simple classical mechanics models, e.g.  \cite[\S 2.3 p. 65]{MR1459161}.
But then the key issue, here as well as in work by Martel and Merle, is the $ \|    v\| _{\widetilde{\Sigma}} ^2 $ term   in \eqref{eq:lem:1stV1}. By a painstaking discussion on quadratic forms,  Martel and Merle
\cite{MaMeARMA2001}  insure that that term is not there, but only for some values of $p$. Alternatively, by the  dual  problem,  in  \cite{MaMegeneral2008} they  generalize considerably their result, but they require $f\in C^3$. Our approach is different. 
We  first  bound, with rather simple routine computations,    high energies in terms of low energies, obtaining Lemma \ref{lem:1stV1}. 
To  bound of  $ \|    v\| _{\widetilde{\Sigma}}   =  \|   \sech  (    \kappa    x  )  v\|   _{L^2(\R )}  $  for some $\kappa>0$ we use Duhamel's formula, but only after  we multiply the equation of $v$ by a cutoff, which tames the nonlinearity, at the cost of a commutator term.  This is similar  to the  second virial inequality method, see for example in \cite{KMM2022}.  For related work on good Boussinesq equation  see \cite{maulen2025asymptoticstabilitystablegood} and therein for further references; for the one dimensional Gross--Pitaevskii equation see \cite{kowalczyk2025newproofliouvilletheorem}. By using a  Kato smoothing for the linearization, we do not need to devise the changes of coordinates of   Martel and collaborators, which in some cases are rather subtle. Our method    appears very robust and general. 

To develop   Kato smoothing for the  linearized operator \eqref{eq:Lin} we exploit the work by Pego and Weinstein \cite{PegoWei1,PegoWei2}   and by Mizumachi
\cite{MizuKdV2001}, with the latter inspired also by Miller \cite{Miller1997} and by Murata \cite{murata1982}.  Kato smoothing is via Fourier transform in time  reduced to estimates on the resolvent which are obtained    expressing the latter in terms of appropriate Jost functions. The theory of these  Jost functions is very well known in the case of Schr\"odinger operators in the line, see \cite{DT1979}, and is even simpler here. A generalization to higher order operators  is in the  monograph by Beals et al. \cite{Bealsbook1988}.  However here we develop from the start the theory and the estimates on the Jost functions which we need.    The integrability of the equation does not   modify the structure of the problem, except for  explicit formulas for the Jost functions in the integrable case, see \cite[p. 6620] {MR1779665}. Differently from the case of the cubic NLS, in the integrable KdV   there is no additional resonance compared to the generic non integrable case. Indeed   
in \cite{PegoWei1,PegoWei2} it is proved that the zeros $\lambda $  of the Evans functions are for $\Re \lambda \ge 0$ always eigenvalues of the linearization. 

Finally we remark that, since the  work by Martel and Merle has been used as blueprint for a large variety of other problems also in higher dimension,   we think that   variations and somewhat different viewpoints might also  prove useful, especially considering that the   notion of  smoothing  introduced by Kato \cite{katoMA1966} has been    extremely  successful.

\section{Notation }\label{sec:notation}

\begin{notation}\label{not:notation} We will use the following miscellanea of  notations and definitions.
\begin{enumerate}

\item In the course of the paper we will sometimes use the notation $\dot u=  \partial _t  u$ and $  u'=  \partial _x  u$.

\item  We will use the Japanese bracket $\<x\> :=\sqrt{1+x^2}$.


\item In analogy to      Kowalczyk et al.   \cite{KMM2022},      we   consider constants  $A, B,  A_1, \epsilon , \delta  >0$ satisfying
 \begin{align}\label{eq:relABg}
\log(\delta ^{-1})\gg\log(\epsilon ^{-1}) \gg     A  \gg     B^2\gg B    \gg 1  \text{  and  $ A_1:=B^{3/5} $.}
 \end{align}

 \item The notation    $o_{\varepsilon}(1)$  means a constant   with a parameter $\varepsilon$ such that
 $o_{\varepsilon}(1) \xrightarrow {\varepsilon  \to 0^+   }0.$

\item
Following Pego and Weinstein  \cite{PegoWei2} for $a\in \R $ we set
\begin{align}\label{eq:pwweigh0}&
   L^2_{a}=\{ v: \ e^{ax} v\in L^2(\R )   \}  \text{ with }\|  v  \| _{L^2_{a}}=\| e^{ax} v  \| _{L^2 (\R )}.
\end{align}

\item We also define (notice that $\widetilde{\Sigma} $ has a natural structure as Hilbert space)    \begin{align}\label{eq:Lebwight}&
   L ^{p,s}=\{ v: \ \< x \> ^{s}   v\in L^p(\R )   \}  \text{ with }\|  v  \| _{ L ^{p,s}}:=\| \< x \> ^{s} v  \| _{L^p(\R ) }. \text{ and }\\&  \widetilde{\Sigma}=\{ v: \ \sech \(      \kappa  x\)   v \in L^2(\R )   \}  \text{ with }\|  v \|_{ \widetilde{\Sigma} } :=\left \| \sech \(      \kappa  x\)   v\right \|_{L^2(\R )} .
\end{align}

\item Given two Banach spaces $X$ and $Y$ we denote by $\mathcal{L}(X,Y)$ the space of continuous linear operators from $X$ to $Y$. We write $\mathcal{L}(X ):=\mathcal{L}(X,X)$.

\item Following  the framework in Kowalczyk et al. \cite{KMM2022} we   fix an even function $\chi\in C_c^\infty(\R , [0,1])$ satisfying
\begin{align}  \label{eq:chi} \text{$1_{[-1,1]}\leq \chi \leq 1_{[-2,2]}$ and $x\chi'(x)\leq 0$ and set $\chi_C:=\chi(\cdot/C)$  for a  $C>0$}.
\end{align}

\item We consider a decreasing function $\vartheta _1\in C^\infty ( \R , [0,1])$ with $\vartheta _1(x) =\left\{
                                                                                                     \begin{array}{ll}
                                                                                                       1, & \hbox{ for $x\le 1$ } \\
                                                                                                       0, & \hbox{for $x\ge 2$  }
                                                                                                     \end{array}                                                                                         \right. $,   we set $ \vartheta _2=1-\vartheta _1 $ and consider the partition of unity $1= \vartheta _{1A_1} + \vartheta _{2A_1} $  where $ \vartheta _{iA_1}(x):=  \vartheta _{i }\(\frac{x}{A_1}\)$.

\item Following  the framework in Kowalczyk et al. \cite{KMM2022}  and
using the function $\chi$  in \eqref{eq:chi}
we  consider for   $C>0$  the   function
\begin{align}\label{def:zetaC}
\zeta_C(x):=\exp\(-\frac{|x|}{C}(1-\chi(x))\)   \text {  and   }  \varphi_C (x):=\int_0^x \zeta_C^2(y)\,dy.
\end{align}
We similarly define
\begin{align}\label{def:phiAi}   \varphi _{iAA_1}  (x):=\int_0^x \zeta_A^2(y) \vartheta _{iA_1} ^2 (y)\,dy  .
\end{align}

\item For $x\in \R$ we set $x^\pm  :=\max (0, \pm x) $.

\item  We will consider the Banach space \begin{align}\label{eq:L11+}
    L^{1,1}_+:=\{ v:\ \<x^+\>v\in L^1\}\ \text{with norm}\ \|v\|_{L^{1,1}_+}:=\|\<x^+\>v\|_{L^1}.
\end{align}

\item   We consider  semi--norms
\begin{align}
&
\| v  \|_{  { \Sigma }_{iA A_1}} :=\left \|  \( \vartheta _{iA_1} \zeta_A v \) ' \right \|_{L^2(\R )} + \left \|   \vartheta _{iA_1} \zeta_A v      \right\|_{L^2(\R )}  \text{for $i=1,2$ .} \label{eq:normA} 
\end{align}
For $v\in C^0(I, H^1(\R ))$   with $I= [0, T) $, we set  
\begin{align}\label{eq:normL2A}
   \| v  \|_{ L^2\( I  , { \Sigma }_{iA A_1} \) } =\lim _{t\to T^-}   \|   \|   v  \| _{{ \Sigma }_{iA A_1}}  \|_{ L^2 (0, t)  }.
\end{align}

\item By a simple computation we have \begin{align}
      \partial _c   \phi _c(x)& = \partial _c \( c ^{\frac 1{p-1}} \phi (\sqrt{c }x)  \)  = \frac{c ^{\frac 1{p-1} -1}}{p-1} \phi (\sqrt{c }x) +\frac{1}{2}  c ^{\frac 1{p-1}-\frac{1}{2}} x \phi '(\sqrt{c }x)\nonumber  \\& =  \frac{c ^{  -1}}{p-1}  \phi _c(x) +\frac{c ^{  -1}}{2} x \phi _c'(x). \label{eq:partcphic}
    \end{align}
\item We set
\begin{align}
  \label{eq:C+} {\C} _+:=\{ z\in \C : \Re z > 0     \}.
\end{align}

\item  Give an operator $T$ for us the resolvent is  $R_{T }(z)=(T -z)^{-1}$. 

 \end{enumerate}

\end{notation} 

\section{Linearization and  modulation }\label{sec:mod}

We set \begin{align}\label{eq:lin1}&
  L _{+c}:=- \partial _x^2   +c -p \phi _c ^{p-1}  
  \\&
  \mathcal{L}_{[c]}
  =  \partial _x       L _{+c} \label{eq:Lin}
\end{align}
When $c=1$, we will write  simply $\mathcal{L}$ and $L_+$ and we will omit the $[c]$ also in the following functions which we take from Pego and Weinstein \cite[Eq. (2.9)]{PegoWei2},
\begin{align}& \label{eq:ker1}
  \xi _1 [c](x )   := \phi _c '(x )  , \quad   \xi _2 [c](x ):= \partial _c \phi _c (x )  \\& \eta _1 [c](x ):=\theta _1 \int _{-\infty} ^{x } \partial _c \phi _c (y) dy   +\theta _2 \phi _c(x ) , \quad   \eta _2 [c](x ):=\theta _3 \phi _c(x )  \label{eq:ker2}
\end{align}
where, 
\begin{align}  \theta _1=-\theta _3=\frac{1}{\mathbf{q}'(c)    } \text{   and } \theta _2= \frac{\| \partial _c \phi _c\| _{L^1}^2}{2( \mathbf{q}'(c) )^2   }, \text{  with } \mathbf{q}(c):=\mathbf{Q}( \phi _{c})= \frac{1}{2} c^{\frac{5-p}{2(p-1)}} \| \phi \| _{L^2(\R )}^2.\label{eq:deftheta}
\end{align}
Notice here that $\mathbf{q}'(c)=\< \partial _c \phi _c,\phi _c \> $ and that
\begin{align}
  \label{eq:lambdap} \mathbf{q}'(1)=\< \Lambda_p  \phi  ,\phi  \>  \text{ where }\Lambda_p=\frac{1}{2}x\partial_x + \frac{1}{p-1}.
\end{align}
An elementary computation shows
\begin{align}
  \label{eq:dcphi}  \partial_{c} \(   c^{\frac{1}{p-1}}u(\sqrt{c}x)\)  = c^{\frac{2-p}{p-1}} (\Lambda_p u)(\sqrt{c}x).
\end{align}
For later use we record that
\begin{align}\nonumber
   \int _\R \partial _c \phi _c (y) dy & =  c^{\frac{2-p}{p-1} -\frac{1}{2}} \int _\R   (\Lambda_p \phi)( y) dy  =  c^{\frac{5-3p}{2(p-1)}  }   \int _\R  \( \frac{1}{2}y\phi '(y)  + \frac{1}{p-1} \phi  (y)  \) dy\\& =  c^{\frac{5-3p}{2(p-1)}  } \frac{3-p}{2(p-1)}\| \phi \| _{L^1(\R )} .\label{eq:asseta1}
\end{align}
The functions in \eqref{eq:ker1}--\eqref{eq:ker2} satisfy, see \cite{PegoWei1},
\begin{align}\label{eq:M2.3}&
   \mathcal{L}_{[c]}\xi _1 [c]=0 , \quad  \mathcal{L}_{[c]}\xi _2 [c]=-\xi _1 [c] ,\\& \label{eq:M2.4}
  \mathcal{L}_{[c]}^*\eta _2 [c]=0 , \quad  \mathcal{L}_{[c]}^*\eta _1 [c]=- \eta _2 [c],\\& \label{eq:M2.5}
  \<  \xi _i [c],\eta _j [c] \> = \delta_{ij}. 
\end{align}
Let $ \displaystyle N_g( \mathcal{L}_{[c]}):=\bigcup _{n=1}^{+\infty}\ker   \mathcal{L}_{[c]}^n$. Pego and Weinstein \cite[Section 2]{PegoWei2} proved the following.

\begin{proposition} \label{prop:PEGOWEI2} Let here $a\in (0, \sqrt{c/3})$ and $1<p<  5$. The following hold.
  \begin{description}
  	  \item[i] As an operator defined in $L^2(\R)$ with domain $H^3( \R )$ we have $\sigma  \(    \mathcal{L}_{[c]}  ) \) =\im \R$.
     \item[ii]  Both in $L^2(\R)$ and $L^2_{a}(\R)$
\begin{align*} &
  \ker   \mathcal{L}_{[c]} =\Span \{     \xi _1 [c]\} \, , \quad N_g( \mathcal{L}_{[c]})=\Span \{     \xi _1 [c],   \xi _2 [c]   \} .
\end{align*}

  \item[iii]  For $\mathcal{L}_{[c]}^*:= -       L _{+c}\partial _x       $ in $L^2_{-a}(\R)$ we have
\begin{align*} & \ker   \mathcal{L}_{[c]} ^* =\Span \{     \eta _2 [c]\} \, , \quad N_g( \mathcal{L}_{[c]}^*)=\Span \{     \eta _1 [c],   \eta _2 [c]   \} .
\end{align*}

  \item[iv] The eigenvalues of $ \mathcal{L}_{[c]}$  in  $L^2(\R)$ and   in  $L^2_{a}(\R)$ are the same.

     \item[v]  The subset  $\mathbf{E}\subseteq (1,5)$ of $p$'s  where $ \mathcal{L}_{[c]}$ has a nonzero eigenvalue $\lambda $ with $\Re \lambda \ge 0$     is a discrete set which does not contain $2$ and 3.
   \end{description}
\end{proposition}
\qed

In Sect. \ref{sec:noeigh}  we will prove the following.
\begin{proposition}\label{prop:noeigh}
 The set $\mathbf{E} $, given above in Proposition \ref{prop:PEGOWEI2}, is empty.
\end{proposition}
 \proof For $p\in (1,5)$ the absence of eigenvalues $\lambda $ with $\Re \lambda > 0$  is well known. See, e.g.\ \cite{KS14SAM,LZ22Memoiers,Pelinovsky13NPS}.
 For the absence of eigenvalues $ \lambda \in \im \R$ with $\lambda \neq 0$, see \S \ref{sec:noeigh} below.   \qed

For  $c  _0 ,r \in \R_+:=(0,\infty)$ and for   $D _{H ^1  (\R )}(u,r):=\{v\in H ^1  (\R )\ |\ \|u-v\| _{H ^1  (\R )}<r \}$,  we set     $$\mathcal{U} (c _0 ,r  ) := \bigcup _{  x_0\in \R }  e^{-x_0 \partial _x}  D_{H ^1  (\R )}({\phi}_{c _0},r ). $$

\begin{lemma}[Modulation]\label{lem:mod1}  For any $c_0>0$ there exist     $\delta _0 >0$  and  $ B_0>0$   with $\delta _0 \ll B_0 ^{-3/2}$
such that for $  \delta _0 ^{-1 }\gg    B\ge B_0$ and for the function  $\zeta _{B}$ introduced in \eqref{def:zetaC} there are
$c,\gamma      \in C^1\(  \mathcal{U} (c_0 ,\delta _0 ), \R _+  \times \R   \) $
 such that for  any $u \in \mathcal{U} (c_0,\delta _0   )$
\small \begin{align}\label{61} &  v (u):=e^{   \gamma (u) \partial _x } u -  \phi  _{c (u)}     \text{ satisfies }   \\& \nonumber  \<    v (u), \zeta _{B} \eta _1 [c (u)]   \> =  \<    v (u),   \eta _2 [c (u)]   \> =0. \end{align} \normalsize
 \normalsize
 For any $c_1>0$   and $\gamma_0\in \R$         we have the identities
\begin{align}& c  (  \phi   _{c_1}  )  =c_1  \text{, }   \gamma  (  \phi   _{c_1}  )=0
  \text{, }   \gamma  (  e^{  - \gamma _0 \partial _x }  u)    = \gamma  (    u)    + \gamma _0 \text{ and } c  (  e^{  - \gamma _0 \partial _x }  u)    =  c  (    u).\label{eq:equivariance}
\end{align}
Furthermore  there exists a constant $K >0$ such that
\begin{align}
  \label{eq:mod2}  B ^{-\frac{1}2}|\gamma(u) -\gamma_0| + |c(u)- 1|\le K       \delta _0 \text{ for all  } u\in  e^{-\gamma_0 \partial _x}  D_{H ^1  (\R )}({\phi} ,\delta  _0    ).
\end{align}
\end{lemma}
\begin{proof}
     Set   $F =  \( F_1  ,  F_2  \) ^\intercal$        with               \small
\begin{align*}
  F_i(u, \gamma, c) =  \left\{
                    \begin{array}{ll}
                       \<   u-e^{  - \gamma   \partial _x } \phi  _c ,e^{  - \gamma   \partial _x } \( \zeta _{B}   \eta _1  [c]  \)  \> , & \hbox{ for $i=1$;} \\
                        \<   u-e^{  - \gamma   \partial _x } \phi  _c ,e^{  - \gamma   \partial _x }     \eta _2  [c]    \>  & \hbox{for $i=2$.}
                    \end{array}
                  \right.
\end{align*} \normalsize
We now introduce the variable
\begin{equation}\label{eq:tildeD}
 \widetilde{\gamma} := B ^{-\frac{1}{2}}(\gamma-\gamma_0)
\end{equation}
 and we set
\begin{align*}
  G_1(u, \widetilde{\gamma}, c) = B ^{-\frac{1}{2}} F_1(u, B ^{ \frac{1}{2}}\widetilde{\gamma} +\gamma_0, c) \text{   and }  G_2(u, \widetilde{\gamma}, c) =   F_2(u, B ^{ \frac{1}{2}}\widetilde{\gamma} +\gamma_0, c)
\end{align*}
Then, from $\< \xi _1 [c],  \zeta _{B}   \eta _1  [c] \> =   \delta _{1i} +  O\( B^{-1}    \)$, we have
\begin{align*}
    \partial _{\widetilde{\gamma}} G_1(u, \widetilde{\gamma}, c)  &= \< \xi _1 [c],  \zeta _{B}   \eta _1  [c] \> -  \<   u - e^{  - \gamma   \partial _x } \phi  _c ,e^{  - \gamma   \partial _x } \(  \zeta _{B}'    \eta _1  [c]   +  \zeta _{B}     \eta _1 ' [c] \)  \> \\& =   1 +  O\( B^{-1}    \) +      O\( \|  u - e^{  - \gamma   \partial _x } \phi  _c \| _{L^2(\R )}    \)
\end{align*}
 and, using  \eqref{eq:M2.5},
\begin{align*}
    \partial _{\widetilde{\gamma}} G_2(u, \widetilde{\gamma}, c)  &=   - B ^{ \frac{1}{2}}  \<   u - e^{  - \gamma   \partial _x } \phi  _c ,e^{  - \gamma   \partial _x }       \eta _2 ' [c]   \> \\& =       O\( B^{1/2}   \|  u - e^{  - \gamma   \partial _x } \phi  _c \| _{L^2(\R )}    \) .
\end{align*}
 From $\< \xi _2 [c],  \zeta _{B}   \eta _1  [c] \> =      O\( B^{-1}    \)$, we have
 \begin{align*}
    \partial _{c} G_1(u, \widetilde{\gamma}, c)  &= -B ^{- \frac{1}{2}}\< \xi _2 [c],  \zeta _{B}   \eta _1  [c] \>  + B ^{- \frac{1}{2}} \<   u - e^{  - \gamma   \partial _x } \phi  _c ,e^{  - \gamma   \partial _x } \(  \zeta _{B}    \partial _c \eta _1  [c]  \)  \> \\& = O\( B^{- \frac{3}{2}}    \) + O\(   \|  u - e^{  - D   \partial _x } \phi  _c \| _{L^2(\R )}    \)
\end{align*}
 and, using  \eqref{eq:M2.5},
 \begin{align*}
    \partial _{c} G_2(u, \widetilde{\gamma}, c)  &= - \< \xi _2 [c],     \eta _2  [c] \>  +   \<   u - e^{  - \gamma   \partial _x } \phi  _c ,e^{  - \gamma   \partial _x } \(  \zeta _{B}    \partial _c \eta _2  [c]  \)  \> \\& = -1+ O\( B^{- 1}    \) + O\(   \|  u - e^{  - \gamma   \partial _x } \phi  _c \| _{L^2(\R )}    \) .
\end{align*}
 So,  for \begin{align}
    \|  u - e^{  - \gamma   \partial _x } \phi  _c \| _{L^2(\R )} \le B ^{-1}  \label{eq:implf3}
\end{align}
  we conclude
 \begin{align} \label{eq:implf1}
   \frac{\partial  G(u, \widetilde{\gamma}, c)}{\partial (\widetilde{\gamma}, c)}
&=  \sigma _3 +O(B ^{-1/2})  \text{ where }\sigma _3 :=\left(
                 \begin{array}{cc}
                                                                                  1 & 0 \\
                                                                                 0 & -1 \\
                                                                                \end{array}
                                                                              \right) .   \end{align}
 Using also \eqref{def:zetaC}, we have
\begin{align*}&
    \partial _{u}G_i(u, \widetilde{\gamma}, c)=\left\{
                                            \begin{array}{ll}
                                              B^{- \frac{1}{2}}  e^{  - \gamma   \partial _x }\( \zeta _{B}   \eta _1  [c]\), & \hbox{ if $i=1$ and } \\
                                               e^{  - \gamma   \partial _x }    \eta _2  [c]  , & \hbox{if $i=2$.}
                                            \end{array}
                                          \right.
\end{align*}
Then $\|  \partial _{u}G_i(u, \widetilde{\gamma}, c) \| _{H^1} \lesssim 1. $
The higher order derivatives of $G(u, \widetilde{\gamma}, c)$ are similarly bounded   with bounds which depend, continuously,  only  on $c$.
Then for any  $c_0$    there are  $\delta =\delta (c_0  ) >0$   with $\delta \ll B ^{-3/2}$     and $K= K(c_0)  $ such that for any $\gamma_0$ there is       an implicit function
\begin{align*}
   (c(\cdot ), \widetilde{\gamma}(\cdot ) )\in  C^1\(     e^{-\gamma_0 \partial _x}  D_{H ^1  (\R )}({\phi}_{c _0},\delta     )  ,   \(  c_0-  K    \delta   ,  c_0+  K   \delta     \) \times \(   -  K     \delta   ,    K     \delta    \)  \)
\end{align*}
with $c(e^{-\gamma_0 \partial _x} {\phi}_{c _0})=c_0$ and $\widetilde{\gamma}(e^{-\gamma_0 \partial _x} {\phi}_{c _0})=0$.
  Using \eqref{eq:tildeD}  and back to the variable $\gamma$,  we have with the same constants \small
\begin{align}\label{eq:implf2}
  (c(\cdot ), \gamma(\cdot ) )\in  C^1\(     e^{-\gamma_0 \partial _x}  D_{H ^1  (\R )}({\phi}_{c _0},\delta     )  ,   \(  c_0-  K    \delta   ,  c_0+  K   \delta     \) \times \(  \gamma_0-  K    B ^{\frac{1}2} \delta   ,  \gamma_0+  K   B ^{\frac{1}2}  \delta    \)  \)
\end{align}
\normalsize
with $c(e^{-\gamma_0 \partial _x} {\phi}_{c _0})=c_0$ and $\gamma (e^{-\gamma_0 \partial _x} {\phi}_{c _0})=\gamma_0$.
  Notice that
\begin{align}\label{eq:orbest}
 \|  u - e^{  - \gamma   \partial _x } \phi  _c \| _{L^2(\R )} \le \|  u -  e^{-\gamma_0 \partial _x} {\phi}_{c _0} \| _{L^2(\R )} +\|   e^{-\gamma_0 \partial _x} {\phi}_{c _0} - e^{  - \gamma   \partial _x } \phi  _c \| _{L^2(\R )}\le C(K) B ^{\frac{1}2}  \delta
\end{align}
so that the choice  $\delta \le  \frac{1}{ C(K) B ^{3/2} }$  yields \eqref{eq:implf3}. Here we have $F(u,\gamma (u),c(u))=0$ identically.

\noindent For any $  y\in \R $  we have identically
\begin{align*}
  F( e^{-y \partial _x} u, \gamma +y, c)= F(u, \gamma , c)
\end{align*}
which differentiating in $y$ at $y=0$ implies
\begin{align*}
   \partial _\gamma F(u, \gamma, c)=   \partial _u F(   u, \gamma  , c)  u'.
\end{align*}
Then by the chain rule, differentiating in $y$ at $y=0$  at a $u\in  e^{-\gamma_0 \partial _x}  D_{H ^1  (\R )}({\phi}_{c_0},\delta (c_0) )$     we obtain \small
\begin{align*}
  0=\left . \partial _y  F( e^{-y \partial _x} u, \gamma(e^{-y \partial _x} u) , c(e^{-y \partial _x} u))\right | _{y=0}=  - \partial _\gamma F(u, \gamma , c)  \(   1+ d\delta(u) u' \) - \partial _c F(u, D, c)    dc(u) u'
\end{align*}\normalsize
where $dc(u) , d\gamma (u) \in \mathcal{L}\( H^1 \( \R \) \)$ are the Frech\'et differentials of the functions in \eqref{eq:implf2}.
Then  by \eqref{eq:implf1}, by taking $\delta (c_0) >0$ small enough, we see that the last equalities can be true only for
$dc(u) u'\equiv 0$ and $ d\gamma(u) u'\equiv  -1$   that is $ \left . \frac{d}{dy}\gamma( e^{-y \partial _x}   u) \right | _{y=0}=1$ for all $u$, i.e. we obtain \eqref{eq:equivariance}. Then we can extend
$c,\gamma     \in C^1\(  \mathcal{U} (c_0 ,\delta (c_0)   ), \R _+  \times \R   \) $ by this equivariance property.
  Finally notice that the estimate in \eqref{eq:mod2} is implicit in \eqref{eq:implf2}.
\end{proof}

For $  u\in C^0\(   [0,+\infty ), H^1(\R )\)$  the solution of \eqref{eq:gKdV1} we have $  u\in C^1\(   [0,+\infty ), H ^{-2}(\R )\)$    and the function
$(c(t),\gamma(t)) = (c(u(t)),\gamma (u(t))) \in C^1([0,+\infty ), \R _+ \times \R )     $.
Using  Lemma \ref{lem:mod1} we can insert in \eqref{eq:gKdV1} the ansatz $u=\phi _c(\cdot -\gamma ) + v(\cdot -\gamma )$  obtaining
\begin{align}     \nonumber &
    \dot v        -   (\dot \gamma -c )  \  v' -   ( \dot \gamma -c ) \  \xi _1[c] + \dot c \ \xi _2[c]       \\& =   \mathcal{L}_{[c]}   v   -  \partial _x    \( f( \phi  _{c} + v  )  -  f( \phi  _{c} )-  f'( \phi  _{c} ) v\)     .   \label{eq:gKdV3}
\end{align}
Notice that setting $\widetilde{v}(t):=e^{t\partial _x ^3}{v}(t)$ we have
\begin{align}
  \label{eq:dertildev}    e^{-t\partial _x ^3}\dot {\widetilde{v}}  = c v'   -  \partial _x    \( f( \phi  _{c} + v  )  -  f( \phi  _{c} ) \)  + (\dot \gamma -c )  \  v' +   ( \dot \gamma -c ) \  \xi _1[c] - \dot c \ \xi _2[c]
\end{align}
which yields
\begin{align}
  \label{eq:regtildev}   \widetilde{v}\in  C^0\(   [0,+\infty ), H ^{ 1}(\R )\)  \bigcap  C^1\(   [0,+\infty ), L ^{ 2}(\R )\).
\end{align}
The proof of Theorem \ref{thm:main} is mainly  based on the following continuation argument, where the norms are defined in Notation \ref{not:notation}.

\begin{proposition}\label{prop:continuation}
For any small $ \epsilon >0$ and for large constants $A$, $B$  and  $A_1:= B ^{3/5}$ as in formula \eqref{eq:relABg}
there exists  a    $\delta _0= \delta _0(\epsilon )   $ s.t.\  if
\begin{align}
  \label{eq:main2}
\|v \| _{L^2(I,  { \Sigma }_{1A A_1} )} +
\|v \| _{L^2(I,  { \Sigma }_{2A A_1} )} +  \|v \| _{L^2(I, \widetilde{\Sigma}  )} + \| \dot \gamma - c \| _{L^2(I  )} + \| \dot c \| _{L^2(I  )}\le \epsilon
\end{align}
holds  for $I=[0,T]$ for some $T>0$ and if  the constant $C\delta$ in \eqref{eq:orbstab} satisfies  $\delta  \in (0, \delta _0)$
then in fact for $I=[0,T]$    inequality   \eqref{eq:main2} holds   for   $\epsilon$ replaced by $   o _{\varepsilon }(1) \epsilon $.
\end{proposition}
 We will split the proof of Proposition \ref{prop:continuation} in a number of partial results obtained assuming the hypotheses of Proposition  \ref{prop:continuation}. The first is the following, proved in \S  \ref{sec:discr}.
\begin{lemma}\label{lem:lemdscrt} We have the estimates
 \begin{align} \label{eq:discrest1}
   &\|\dot c   \| _{L^2(I)} +    \|\dot \gamma -c \| _{L^2(I)}    \lesssim  B ^{-1}    \(  \|v \| _{L^2(I,  { \Sigma }_{1A A_1} )} +
\|v \| _{L^2(I,  { \Sigma }_{2A A_1} )} \)   .
\end{align}

\end{lemma}

In \S  \ref{sec:virial} we will prove the following.
\begin{lemma}[Virial Inequality]\label{prop:1virial}
 We have
 \begin{align}\label{eq:sec:1virial1}
   \| v \| _{L^2(I,  { \Sigma }_{1A A_1}   )}& \lesssim \sqrt{A B} \delta+   \| v\| _{L^2(I, \widetilde{\Sigma}   )}+ A_1 ^{-1}   \| v \| _{L^2(I,  { \Sigma }_{2A A_1}   )}   \text{ and} \\   \label{eq:sec:1virial2}
   \| v \| _{L^2(I,  { \Sigma }_{2A A_1}   )}& \lesssim \sqrt{A B} \delta+  A _1 ^{-1 }  \| v \| _{L^2(I,  { \Sigma }_{1A A_1}   )}   .
 \end{align}
 \end{lemma}
Notice that $\sqrt{A B} \delta  \ll  A ^{-1}\epsilon ^2 =o _{B^{-1}} (1)  \epsilon ^2 $  in \eqref{eq:sec:1virial1}
and similarly, from $A_1=B^{3/5}$ below we have $ A _1  ^{3/2} B ^{-1}  =   B ^{-\frac{1}{10}} $ and  $ B ^{ 1/2}    A _1 ^{-1 }   = B ^{-\frac{1}{10}} $
where the following will be proved    from  \S \ref{sec:disp} to \S\ref{sec:proofsmooth2}.

\begin{lemma}[Smoothing Inequality]\label{prop:smooth11}
 We have
 \begin{align}\label{eq:sec:smooth11}
   \| v \| _{L^2(I, \widetilde{\Sigma}   )} \lesssim       A _1  ^{3/2} B ^{-1}   \| v \| _{L^2(I,  { \Sigma }_{1A A_1}   )} +     B ^{ 1/2}   \| v \| _{L^2(I,  { \Sigma }_{2A A_1}   )} +  o _{B^{-1}} (1)  \epsilon     .
 \end{align}
 \end{lemma}

\textit{Proof of Theorem \ref{thm:main}.}
From Lemmas \ref{prop:1virial}  and \ref{prop:smooth11}  and from \eqref{eq:relABg} we have  \small
\begin{align*}
  & \| v \| _{L^2(I, \widetilde{\Sigma}   )}  \lesssim A _1  ^{3/2} B ^{-1}   \| v \| _{L^2(I,  { \Sigma }_{1A A_1}   )} +     B ^{ 1/2}   \| v \| _{L^2(I,  { \Sigma }_{2A A_1}   )} +  o _{B^{-1}} (1)  \epsilon   \\& \lesssim    A _1  ^{3/2} B ^{-1}  \| v\| _{L^2(I, \widetilde{\Sigma}   )} +  B ^{ 1/2}    A _1 ^{-1 }   \| v \| _{L^2(I,  { \Sigma }_{1A A_1}   )}   +   A _1  ^{1/2} B ^{-1}  \| v \| _{L^2(I,  { \Sigma }_{2A A_1}   )}  +  o _{B^{-1}} (1)  \epsilon
\end{align*}
\normalsize
which by $ A_1=B^{3/5}$ and   bootstrap   yields
\begin{align*}
   \| v \| _{L^2(I, \widetilde{\Sigma}   )} &  \lesssim  B ^{ - \frac{1}{10}} \(  \| v \| _{L^2(I,  { \Sigma }_{1A A_1}   )}   +     \| v \| _{L^2(I,  { \Sigma }_{2A A_1}   )}   \)  +  o _{B^{-1}} (1)  \epsilon
\end{align*}
 so that  by \eqref{eq:main2}  we obtain $ \| v \| _{L^2(I, \widetilde{\Sigma}   )}= o _{B^{-1}} (1)  \epsilon$ which inserted in
\eqref{eq:sec:1virial1}   yields   $  \| v \| _{L^2(I,  { \Sigma }_{1A A_1}   )}= o _{B^{-1}} (1)  \epsilon$ and
in
\eqref{eq:sec:1virial2}   yields   $  \| v \| _{L^2(I,  { \Sigma }_{2A A_1}   )}= o _{B^{-1}} (1)  \epsilon$. In turn these inserted in    \eqref{eq:discrest1}  yield $  \|\dot c   \| _{L^2(I)} +    \|\dot \gamma -c \| _{L^2(I)} = o _{B^{-1}} (1)  \epsilon$. This proves   Proposition \ref{prop:continuation}. By $u\in C^0([0,+\infty),H^1(\R))$ it is elementary that  this implies that   we can take $I=\R _+$ in all the above inequalities.  
By 
\begin{align*}
    |  (e^{- a|x|}   v  )'| + |   e^{- a|x|}   v   | \lesssim   |  ( \zeta  _A^{-1} e^{- a|x|}  \zeta  _A v  )'| +|     \zeta  _A v   |\lesssim |    (\zeta  _A v  )'| + |     \zeta  _A v   |
\end{align*}
for $a>1/A$, and inserting   $1= \vartheta _{1A_1} + \vartheta _{2A_1} $, we obtain 
$\|  e^{- a|x|}   v   \| _{H^1(\R )} \lesssim \|      v   \| _{\Sigma _{1A  A_1}}  + \|      v   \| _{\Sigma _{2A  A_1}}$ and so 
the inequalities in   \eqref{eq:sec:1virial1} for $I=\R _+$   implies \eqref{eq:asstab2}.
Then we can proceed like in \cite[Proof of Theorem 1.2]{CM24D1} considering
  \begin{equation*}
  \begin{aligned}
   \mathbf{a}(t) &:= 2 ^{-1}\|  e^{- a\< x\>}   v (t)  \| _{L^2(\R )} ^2 = 2 ^{-1}\< e ^{t\partial _x ^3} e^{- 2a\< x\>} e ^{-t\partial _x ^3}\widetilde{v} (t) , \widetilde{v} (t) \> .
  \end{aligned}
  \end{equation*}
Taking time derivative  we claim   that $t\to \mathbf{a}(t)$ is an absolutely continuous function with
\begin{align*}
   \dot {\mathbf{a}}(t)&=   \frac{1}{2} \< \left [      \partial _x ^3,  e^{- 2a\< x\>}    \right ] v , v \>       + \<  e^{- 2a\< x\>}   {v} (t) , e ^{-t\partial _x ^3}\dot {\widetilde{v}} (t) \>  .
\end{align*}
This indeed is the case if $\widetilde{v}$ is regular and can be obtained for general $\widetilde{v} $ like in  \eqref{eq:regtildev}  taking a sequence of regular functions  $\{   \widetilde{v}_n\} _{n\in \N}$ with $\widetilde{v}_n \xrightarrow{n\to +\infty} \widetilde{v}$ locally uniformly in the space in  \eqref{eq:regtildev}.
Then we obtain  \small
\begin{align*}
  \dot {\mathbf{a}}   & =  \frac{1}{2} \< \left [      \partial _x ^3,  e^{- 2a\< x\>}    \right ] v , v \> +   \frac{1}{2}   \dot \gamma        \< \left [   e^{- 2a\< x\>} ,   \partial _x   \right ] v , v \>
   +    (\dot \gamma-c)   \<  \xi _1[c]  , e^{- 2a\< x\>} v    \>   \\& -\dot c   \<  \xi _2[c]  , e^{- 2a\< x\>} v    \>  +   \<
          f( \phi  _{c} + v  )  -  f( \phi  _{c}  )     , \left [      \partial _x  ,  e^{- 2a\< x\>}    \right ] v  + e^{- 2a\< x\>}  v'   \>  .
\end{align*}\normalsize
We exploit the fact that when $ u\in   e^{-\gamma _0 \partial _x}  D_{H ^1  (\R )}({\phi} ,\delta     )$
\begin{align}\label{eq:estvH1}
   \|   v\|  _{H^1(\R ) } & =  \|  e^{   \gamma  (u) \partial _x } u -  \phi  _{c (u) }    \|  _{H^1(\R ) }  \le   \|  u  -  e^{-\gamma _0 \partial _x}  \phi  _{c (u) }    \|  _{H^1(\R ) }
\\& + \|  e^{  \( \gamma (u) -\gamma _0\) \partial _x } \phi -  \phi  _{c (u) }    \|  _{H^1(\R ) }\lesssim \delta + |\gamma  (u) -\gamma_0| + |c(u)-1|\lesssim B ^{\frac{1}{2}} \delta .\nonumber
\end{align}
Then    by Lemma \ref{lem:lemdscrt}
\begin{align*}
  | \dot {\mathbf{a}}  |& \lesssim |\dot c| ^2 + |\dot \gamma-c|^2 + \| v \| _{H^1}^2\lesssim B \delta^2
\end{align*}
and since we already know from \eqref{eq:asstab2} that $\mathbf{a}\in L^1(\R )$, we conclude that
$ \mathbf{a}(t) \xrightarrow{t\to +\infty} 0$.

\noindent
Next we prove  \eqref{eq:asstab3} by proceeding like in \cite {CM24D1}.
We can  take   $ 0<a\le \kappa /2$ such that we have the  following, which will be used below,  
\begin{equation}\label{eq:uniformity1}
  e^{-2a\<x\>}\gtrsim    \max \{   \phi _{c(t)} ,  \phi _{c(t)} ^{p-1} \}  \text {  for all } t\ge 0  .
\end{equation}
Since $\mathbf{Q}(\phi_{c})=\mathbf{q}(c)$ is strictly  monotonic in $c$, it suffices to show that  the  $\mathbf{Q}(\phi_{c (t)})$ converges as $t\to\infty$.
From the conservation of $\mathbf{Q}$, see Proposition \ref{prop:consQ}, the exponential decay of $\phi _c$, \eqref{eq:asstab2} and   \eqref{eq:asstab20}, we have
\begin{align}\label{eq:equipart1}
\lim_{t\to \infty}\(\mathbf{Q}(u_0)-\mathbf{Q}(\phi_{c (t)})-\mathbf{Q}(v (t))\)=0.
\end{align}
Thus, our task is now to prove $\frac{d}{dt}\mathbf{Q}(v) \in L^1(\R _+) $, which is sufficient to show the convergence of $\mathbf{Q}(v)$.
Like  $\mathbf{a}(t)$ also $\mathbf{Q}(v(t))$ is absolutely continuous in $t$ with derivative
\begin{align*}
\frac{d}{dt}\mathbf{Q}(v)&= \frac{d}{dt}\mathbf{Q}(\widetilde{v})  =    \<v,  e^{-t\partial _x ^3} \dot{\widetilde{v}}\>
=  \dot \gamma   \cancel{\<v, {v}'\> } + \left [ (\dot \gamma -c ) \<v, \xi _1[c]\> -  \dot c \<v, \xi _2[c]\>  \right ]\\&   +\<v',f( \phi  _{c} + v  )  -  f( \phi  _{c} )\>
=:I+II
\end{align*}
where the canceled term is null and are hence neglected.  Then  \eqref{eq:main2} implies that  $I\in L^1(\R _+)$.   So the key term is $II$.  We partition  for $s\in (0,1)$ the line    as\begin{align}&
   \Omega_{1,t ,s} :=\{x\in \R\ |\   | s \ v(t,x)| \le 2 \phi   _{c (t)}(x) \}  \text{ and }\label{eq:partition}\\&  \Omega_{2,t,s } :=\R\setminus \Omega_{1,t,s }  =\{x\in \R\ |\  | s\  v(t,x)|> 2 \phi   _{c (t)}(x)\} .\nonumber
\end{align}
Then, we have
\begin{align*}
II(t)&=\sum_{j=1,2}   \int_0^1ds \int_{\Omega_{j,t ,s }} \(   f '( \phi_{c (t)}+ s\ v (t)) - f '(   s\ v (t))\)    v(t)  v' (t)        \,dx\\& =:II_1(t)+II_2(t).
\end{align*}
We have
\begin{align*}
|II_1(t)|&\lesssim    \int_{\R}  \phi_{c (t)} ^{p-1} |v (t)v' (t)  | \,dx    \lesssim \|  e^{- a\<x\>}   v (t)\| _{H^1 (\R )}  ^2 \in L^1(\R _+ ).
\end{align*}
We have
\begin{align*}
   II_2(t) &=   - \int _{[0,1] ^2 }  d\tau ds  \int_{\Omega_{2,t,s }}   f''(\tau \ \phi_{c (t)}+s\ v (t))  \phi_{c (t)}   v(t)  v' (t)        dx .
\end{align*}
Then by  $| sv(t,x)|> 2\phi_{c (t)} (x)$,     we have the following, which yields \eqref{eq:asstab3},
\begin{align*}
   |II_2(t)| &\lesssim    \int _{[0,1] ^2 }  d\tau  ds \int_{\Omega_{2,t,s }}       | s\ v (t)  | ^{p-2}  \phi_{c (t)}  | v(t)  v' (t)   |    dx \\& \lesssim    \int_{\R }        \phi_{c (t)}  ^{p-1}   | v(t)  v' (t)   |      dx \lesssim    \|  e^{- a\<x\>}    v (t)\| _{H^1 (\R )}  ^2 \in L^1(\R _+ ).
\end{align*}
\qed

\section{Bounds on the discrete modes: proof of Lemma \ref{lem:lemdscrt} }\label{sec:discr}

We will prove the following, which implies \eqref{eq:discrest1},
\begin{align} \label{eq:discrest1stat}
   & |\dot c    |  +     |\dot\gamma -c  |     \lesssim  B ^{-1}    \(  \|v \| _{   { \Sigma }_{1A A_1} } +
\|v \| _{  { \Sigma }_{2A A_1} } \)   .
\end{align}
We apply  to equation \eqref{eq:gKdV3} the inner product $\< \cdot , \zeta_{B} \eta _1[c] \>$ obtaining, thanks to cancellations due to \eqref{eq:M2.4} and \eqref{61},
\begin{align}& \nonumber
  - ( \dot \gamma -c ) \left [ \< \xi _1[c] , \zeta_{B} \eta _1[c] \>  +\< v' , \zeta_{B} \eta _1[c] \>  \right ]    + \dot c   \left [ \< \xi _2[c] , \zeta_{B} \eta _1[c] \> - \<  v , \zeta_{B} \partial _c\eta _1[c] \>  \right ] \label{eq:seceq}\\& = -  \<  v , \zeta_{B} \eta _2[c] \> +   \<  v ,[ \mathcal{L}_{[c]}^* , \zeta_{B}] \eta _1[c] \> + \<   f( \phi  _{c} + v  )  -  f( \phi  _{c} )-  f'( \phi  _{c} ) v   ,\( \zeta_{B} \eta _1[c]\) ' \> \\&=: A _{0 }+ A _{1 }+A _{2 }\nonumber
\end{align}
where we also used  the following, due to \eqref{61},
\begin{align*}
      -\<\dot v , \zeta_{B} \eta _1[c] \> = \dot c    \<  v , \zeta_{B} \partial _c\eta _1[c] \> .
\end{align*}
By \eqref{61} we have
\begin{align*}
  A_0 = \<  v , \eta _2[c] \> -  \<  v , (1- \zeta_{B}) \eta _2[c] \> =-  \<  v , (1- \zeta_{B}) \eta _2[c] \>
\end{align*}
so that  by    the definition of   $ \eta _2[c]=\theta _3 \phi _c $ in \eqref {eq:ker2}   and using the partition of unity $1=\vartheta _{1A_1} + \vartheta _{2A_1}$   we have \small
\begin{align*}
  |A_0 |&\le  \< \frac{\zeta _B}{\zeta _A}\zeta _A (\vartheta _{1A_1} + \vartheta _{2A_1})|v|,        (\zeta _B ^{-1}- 1)  \eta _2[c] \>   \le   (  \| v  \|_{  { \Sigma }_{1A A_1}}+ \| v  \|_{  { \Sigma }_{2A A_1}})  \left   \|     \( e ^{ \frac{|x|}{B}(1-\chi )  }  -1\) \phi _c \right \| _{L^2}\\&  \le    (  \| v  \|_{  { \Sigma }_{1A A_1}}+ \| v  \|_{  { \Sigma }_{2A A_1}})   \left  \|   \frac{|x|}{B}(1-\chi )     \zeta _B ^{-1}  \phi _c \right \| _{L^2}\lesssim B ^{-1}  (  \| v  \|_{  { \Sigma }_{1A A_1}}+ \| v  \|_{  { \Sigma }_{2A A_1}})   .
\end{align*}\normalsize
With an elementary computation, we have
\begin{align*}
  [ \mathcal{L}_{[c]}^* , \zeta_{B}] &= -  L_{+c}\zeta_{B}'  -[ L_{+c}   , \zeta_{B}] \partial _x =
  \zeta_{B}'(\partial _x^2  -c + f'( \phi _c)) )
 +[ \partial _x^2    , \zeta_{B} '] +[ \partial _x^2     , \zeta_{B}  ]  \partial _x \\& =  \zeta_{B}'\( \partial ^2_x + f' ( \phi  _{c} ) \) -c\zeta_{B}' +  2  \zeta_{B} '' \partial _x + \zeta_{B} '''+( 2  \zeta_{B}'   \partial _x + \zeta_{B} '' ) \partial _x .
\end{align*}
Then we write
\begin{align*}
  A _{1 }&= \<  v ,  \zeta_{B} ' \( \partial ^2_x + f' ( \phi  _{c} ) \)   \eta _1[c] +\( 2  \zeta_{B}'   \partial _x   + 3 \zeta_{B} ''    \)     \eta _1'[c]\>
\\& +  \<  v ,   \zeta_{B} '''      \eta _1 [c]\>-c \<  v ,  \zeta_{B} '    \eta _1[c] \>  =:  A _{1 1}+ A _{1 2} +  A _{1 3}.
\end{align*}
Then,  using again   $1=\vartheta _{1A_1} + \vartheta _{2A_1}$,    we have \small
\begin{align*}
  | A _{1 1}|&\lesssim  B ^{-1}  \sum _{i=1,2}\|    \zeta _A  \vartheta _{iA_1} v \| _{L^2}  \(  \|   \zeta _A ^{-1} \( \partial ^2_x + f' ( \phi  _{c} ) \)   \eta _j[c] \| _{L^2}    + B \|   \zeta _A ^{-1}  \( 2  \zeta_{B}'   \partial _x   + 3 \zeta_{B} ''    \)     \eta _j'[c] \| _{L^2}  \) \\&  \lesssim        B ^{-1}                (  \| v  \|_{  { \Sigma }_{1A A_1}}+ \| v  \|_{  { \Sigma }_{2A A_1}}) .
\end{align*}\normalsize
By $ |\zeta_{B} '''|\lesssim B ^{-3}\zeta_{B}\lesssim B ^{-3}\zeta_{2B}^{2}$
  we have
\begin{align*}
  | A _{1 2}|\lesssim  B ^{-3} \|  \zeta_{2B}  \| _{L^2}  \sum _{i=1,2} \left  \| \frac{ \zeta_{2B} }{\zeta _A } \zeta _A \vartheta _{iA_1} v \right \| _{L^2}   \lesssim     B ^{- \frac{5}{2}}  (  \| v  \|_{  { \Sigma }_{1A A_1}}+ \| v  \|_{  { \Sigma }_{2A A_1}})     .
\end{align*}
Next we focus on  $A _{13}$.
From \eqref{def:zetaC} we have
\begin{align*}
  \<  v ,  \zeta_{B} '    \eta _1[c] \> &= -\frac{1}{B}   \<  v ,  \zeta_{B}   \frac{x}{|x|}    \eta _1[c] \>  +\frac{1}{B}  \<  v ,  \zeta_{B} \(  \frac{x}{|x|}  \chi + |x| \chi '   \)   \eta _1[c] \> .
   \end{align*}
  Then by $\supp \chi \subseteq [-2,2] $ we have like above
 \begin{align*}
   \left |B ^{-1}  \<  v ,  \zeta_{B} \(  \frac{x}{|x|}  \chi + |x| \chi '   \)   \eta _1[c] \> \right |  \lesssim        B ^{-1}               \| v  \|_{  { \Sigma }_{1A A_1}}
 \end{align*}
   while we write
 \begin{align*}
   \frac{1}{B}   \<  v ,  \zeta_{B}   \frac{x}{|x|}    \eta _1[c] \>  = -    \frac{2}{B}   \<  v ,  \zeta_{B}        \eta _1[c] \> _{L^2(\R _-)} + \cancel {B ^{-1}  \<  v ,  \zeta_{B}      \eta _1[c] \> }
 \end{align*}
where the canceled term is null by \eqref{61}. Since $\eta _1[c]\xrightarrow{x\to -\infty}0$ exponentially fast, then
\begin{align*}
   |B ^{-1} \<  v ,  \zeta_{B}        \eta _1[c] \> _{L^2(\R _-)} |  &\lesssim      B ^{-1}  \( \| v  \|_{  { \Sigma }_{1A A_1}} +\| v  \|_{  { \Sigma }_{2A A_1}} \)    \|   \zeta _A ^{-1}    \zeta_{B}       \eta _1[c]  \| _{L^2(\R _-)}   \\&\lesssim    B ^{-1}                \( \| v  \|_{  { \Sigma }_{1A A_1}} +\| v  \|_{  { \Sigma }_{2A A_1}} \).
 \end{align*}
Hence we have proved
\begin{align*}
  | A _{13}|\lesssim   B ^{-1}               \( \| v  \|_{  { \Sigma }_{1A A_1}} +\| v  \|_{  { \Sigma }_{2A A_1}} \) .
\end{align*}
Turning to $A _{ 2}$, for  $s\in (0,1)$  we set
\begin{align}&
   \Omega_{1,t ,s } :=\{x\in \R\ |\  2| s v(t,x)| \le  \phi   _{c (t)}(x) \}  \text{ and }\label{eq:partition+}\\&  \Omega_{2,t,s } :=\R\setminus \Omega_{1,t ,s }  =\{x\in \R\ |\ 2| s v(t,x)|> \phi   _{c (t)}(x)\} .\nonumber
\end{align}
Then we have
\begin{align*}
   A _{ 2 }& = \sum _{i=1,2}\int _{0}^{1} ds \int _{\Omega_{i,t,s }} \(    f'( \phi  _{c} + sv  ) -  f'( \phi  _{c} )   \) v\( \zeta_{B} \eta _1[c]\) ' dx  =:  \sum _{i=1,2}A _{2 i}  .
\end{align*}
We have
 \begin{align*}
   |A _{2 1}|\lesssim \int _0 ^1 ds   \int _{\Omega_{1,t ,s }}\phi _c ^{p-1} |v |   \left |\zeta_{B}' \eta _1[c]\right | dx  +  \int _0 ^1 ds\int _{\Omega_{1,t ,s }}\phi _c ^{p-2}s| v|^2   \left |\zeta_{B} \eta _1'[c]\right | dx  =:  \sum _{i=1,2}A _{2 1i}
 \end{align*}
Then from $|\zeta_{B}'|\lesssim B ^{-1}\zeta_{B}$,  $\|\eta _1[c]\| _{L^\infty (\R )}\lesssim 1$  and $\|\phi _c ^{p-1}     \zeta_{B} \zeta_{A}   ^{-1}\| _{L^2 (\R )}\lesssim 1$, we have
\begin{align*}
  A _{211}\lesssim B ^{-1} \int _{\R}\phi _c ^{p-1}    |\eta _1[c]| \zeta_{B} \zeta_{A}   ^{-1}       \zeta_{A}     \(  \vartheta _{1A_1}  + \vartheta _{2A_1}   \)  |v|   dx \lesssim  B ^{-1} (  \| v  \|_{  { \Sigma }_{1A A_1}} + \| v  \|_{  { \Sigma }_{2A A_1}}  )
\end{align*}
and  from $\|\eta _1'[c]\| _{L^2 (\R )}\lesssim 1$  and, in $\Omega_{1,t ,s }$, from $\phi _c ^{p-2}|s v|=\phi _c ^{p-2}| sv|^{\frac{3-p}{2}+ \frac{p-1}{2}}\le \phi _c ^{\frac{p-1}{2}}|s v|^{  \frac{p-1}{2}}$, we have
\begin{align*}
  A _{212}&\lesssim \int _{\R}  \phi _c ^{\frac{p-1}{2}}  |v | ^{ \frac{p-1}{2} +1}  \zeta_{B} \( |\partial _c \phi _c| +   |  \phi _c'| \) dx  \\ &\lesssim  \| v \| _{L^\infty (\R )} ^{\frac{p-1}{2}}  \int _{\R}     \phi _c ^{\frac{p-1}{2}}  \( |\partial _c \phi _c| +   |  \phi _c'| \)    \zeta_{B} \zeta_{A}   ^{-1}       \zeta_{A}     \(  \vartheta _{1A_1}  + \vartheta _{2A_1}   \)  |v|  dx \\&  \lesssim B ^{\frac{p-1}{4}} \delta ^{\frac{p-1}{2}} (  \| v  \|_{  { \Sigma }_{1A A_1}} + \| v  \|_{  { \Sigma }_{2A A_1}}  ) \end{align*}
  where we used    $\|v\| _{L^\infty (\R )}\lesssim B ^{\frac{1}2}  \delta$   from
  \eqref{eq:estvH1}.
Summing up, we conclude that
\begin{align*}
   |A _{21}|\lesssim       B ^{-1} (  \| v  \|_{  { \Sigma }_{1A A_1}} + \| v  \|_{  { \Sigma }_{2A A_1}}  )   .
 \end{align*}
 Next we turn to
\begin{align*}
   |A _{22}|&\lesssim  \int _{\R}|v| ^{p }   |\zeta_{B}' \eta _1[c] |   dx + \int _{\R}|v| ^{p }  |\zeta_{B}  \eta _1' [c] | dx
  =: A _{221}+A _{222}.
 \end{align*}
Then, using again $\| v  \|_{ L^\infty (\R) } \lesssim B ^{1/2}\delta $, we have
\begin{align*}
   A _{222}\lesssim  B ^{\frac{p-1}{2}}\delta ^{p-1} (  \| v  \|_{  { \Sigma }_{1A A_1}} + \| v  \|_{  { \Sigma }_{2A A_1}}  )
\end{align*}
and
\begin{align*}
 A _{221}&\lesssim   \delta ^{p-1} B ^{\frac{p-1}{2}-1}    \|    \zeta_{A}^{-1}  \zeta_{B}\| _{L^2}   \|     \zeta_{A} \(  \vartheta _{1A_1}  + \vartheta _{2A_1}   \)v\| _{L^2}
 \\&
 \lesssim \delta ^{p-1} B ^{\frac{p-1}{2}-\frac{1}{2}}(  \| v  \|_{  { \Sigma }_{1A A_1}} + \| v  \|_{  { \Sigma }_{2A A_1}}  ).
\end{align*}
Collecting the estimates, taking $\delta$ sufficiently small so that $B^{\frac{p-1}{4}}\delta^{\frac{p-1}{2}}\leq B^{-1}$, we have
\begin{align*}
    |A_0+A_1+A_2|\lesssim B^{-1}\left(\|v\|_{\Sigma_{1AA_1}}+\|v\|_{\Sigma_{1AA_1}}\right).
\end{align*}
Therefore, since
\begin{align*}
    |\< \xi _1[c] , \zeta_{B} \eta _1[c] \>-1|+|\<\xi_2[c],\zeta_B\eta_1\>|\lesssim B^{-1}
\end{align*}
and
\begin{equation}
  \label{eq:eqmaincoff1}
  \begin{aligned}
    \< v' , \zeta_{B} \eta _1[c] \> &=-  \< v  , \zeta_{B}'  \eta _1[c] \>-   \< v  , \zeta_{B}   \( \eta _1[c] \)' \> =O\(  \| v \| _{L^2}    \)=O\(  B ^{\frac{1}{2}} \delta    \) \text{ and}
\\
   \<  v , \zeta_{B} \partial _c\eta _1[c] \> &= O\( B ^{\frac{1}{2}} \| v \| _{L^2}    \) =  O\(  B   \delta    \) ,
  \end{aligned}
\end{equation}
we have
\begin{align}\label{est:dotD-c}
    |\dot \gamma-c|\lesssim B^{-1}|\dot c| + B^{-1}\left(\|v\|_{\Sigma_{1AA_1}}+\|v\|_{\Sigma_{1AA_1}}\right).
\end{align}
 Next, we  apply  to equation \eqref{eq:gKdV3} the inner product $\< \cdot ,   \eta _2[c] \>$, obtaining
 \begin{align}&
    ( \dot \gamma -c ) \< v  ,   \eta _2[c]'  \>      + \dot c   \left [ 1 - \<  v ,   \partial _c\eta _2[c] \>  \right ] \label{eq:seceq2}\\& =   \<   f( \phi  _{c} + v  )  -  f( \phi  _{c} )-  f'( \phi  _{c} ) v   ,  \(\eta _2[c] \) ' \>  =:  \widetilde{A} _{2 }.\nonumber
\end{align}
Proceeding as above,  using \eqref{eq:estvH1}, we have $|\widetilde{A} _{2 }|\lesssim  B^{-1}\left(\| v  \|_{  { \Sigma }_{1A A_1}} + \| v  \|_{  { \Sigma }_{2A A_1}}\right)$, where we have used $B^{1/2}\delta\leq B^{-1}$.
Thus, we have
\begin{align}\label{est:dotc}
    |\dot c|\lesssim B^{1/2}\delta |\dot \gamma-c|+B^{-1}\left(\| v  \|_{  { \Sigma }_{1A A_1}} + \| v  \|_{  { \Sigma }_{2A A_1}}\right).
\end{align}
Combining this with \eqref{est:dotD-c} and by bootstrap we obtain \eqref{eq:discrest1stat} .
\qed

\section{The virial inequality: proof of Lemma \ref{prop:1virial} }\label{sec:virial}

Using the notation in \eqref{def:phiAi}  we set
\begin{align}\label{def:funct:vir}
\mathcal{I}_{i} := 2^{-1}\<  v  , \varphi _{iAA_1} v  \>    \text{  for $i=1,2$ }.
\end{align}

\begin{lemma}\label{lem:1stV1}
There exists a fixed constant $C>0$ s.t.
\begin{align}
 \| v \| _{  { \Sigma }_{1A A_1}    }^2& \lesssim      {C}\left [  \dot{\mathcal{I}}_1   + |\dot \gamma -c|^2+    |\dot  c|^2 + \|    v\| _{\widetilde{\Sigma}} ^2  + A_1 ^{-2 }  \| v \| _{  { \Sigma }_{2A A_1}    }^2 \right ] \text{ and}    \label{eq:lem:1stV1} \\  \nonumber \| v \| _{  { \Sigma }_{2A A_1}    }^2& \lesssim      {C}\left [  \dot{\mathcal{I}}_2  +     A_1 ^{-2 }  \| v \| _{  { \Sigma }_{1A A_1}    }^2 \right ]   .
\end{align}

\end{lemma}
\begin{proof}
Like in \S \ref{sec:mod}    we have $t\to \mathcal{I}_{i}$ absolutely continuous with \small
\begin{align} \nonumber
  \dot {\mathcal{I}}_{i} &= \frac{d}{dt} 2^{-1}\<  e^{-t\partial _x ^3}\widetilde{{v}}    , \varphi _{iAA_1} e^{-t\partial _x ^3}\widetilde{{v}}   \> = 2^{-1}\<   {{v}}    ,[ \partial _x ^3 , \varphi _{iAA_1}]  v   \> +  \<  e^{-t\partial _x ^3}
\dot {\widetilde{{v}}}    , \varphi _{iAA_1} e^{-t\partial _x ^3}\widetilde{{v}}   \> \\&  \nonumber= \< v , \(  2^{-1} [ \partial _x ^3 , \varphi _{iAA_1}]  +c  \varphi _{iAA_1}  \partial _x    \)  v   \> - \< \( f( \phi  _{c} + v  )  -  f( \phi  _{c} ) \) ' , \varphi _{iAA_1} v  \>
  \\&   +(\dot \gamma -c )\< v' ,  \varphi _{iAA_1} v  \>+(\dot \gamma -c )\< \xi _1[c] ,  \varphi _{iAA_1} v  \> -\dot c \< \xi _2[c] , \varphi _{iAA_1} v  \>
=: \sum _{k=1}^{5}
B_k^{(i)} .\label{eq:vir2}
\end{align} \normalsize
Only for this proof, we   set $\omega ^{(i)}:= \zeta_A  \vartheta _{iA_1}v$.
We have
\begin{align}\label{eq:B1_decomposition}
   B_1^{(i)} &=  \< v'' , \zeta_A ^2\vartheta _{iAA_1}^2 v \>  -\frac{1}{2} \< \zeta_A\vartheta _{i A_1} v'  , \zeta_A \vartheta _{i A_1} v '\> -\frac{c}{2}  \| \omega ^{(i)}\| _{L^2}^{2}  ,
\end{align}
with \small
\begin{equation*}
    \< v'' , \zeta_A^2\vartheta _{i A_1}^2 v \>  -\frac{1}{2} \< \zeta_A  \vartheta _{i A_1}v'  , \zeta_A \vartheta _{i A_1} v '\> = \frac{1}{2}   \<  (\zeta_A^2\vartheta _{i A_1}^2)'',        v ^2 \> - \frac{3}{2} \< \zeta_A \vartheta _{i A_1} v'  , \zeta_A \vartheta _{i A_1} v '\>
\end{equation*}\normalsize
where  \small
\begin{align*}&
  \< \zeta_A  \vartheta _{i A_1} v'  , \zeta_A  \vartheta _{i A_1} v '\>  =  \<  (\zeta_A  \vartheta _{i A_1} v)'   - (\zeta_A \vartheta _{i A_1}) ' v  ,  (\zeta_A \vartheta _{i A_1}  v)'   -(\zeta_A \vartheta _{i A_1}) '  v\>\\&  =   \|  (\omega ^{(i)} )'\| _{L^2}^{2}- 2 \<     (\zeta_A \vartheta _{i A_1}   )' v  ,     (\zeta_A \vartheta _{i A_1}  )' v+    \zeta_A \vartheta _{i A_1}  v ' \> + \<  (\zeta_A \vartheta _{i A_1}  )'^2,v^2\>
\\& =  \|(\omega ^{(i)} )'\| _{L^2}^{2}-  \< ( (\zeta_A \vartheta _{i A_1}  ) ')^2,v^2\> +\< \left ((\zeta_A \vartheta _{i A_1}  ) '  \zeta_A \vartheta _{i A_1} \right ) ' ,v^2\>  \\& = \| (\omega ^{(i)} )'\| _{L^2}^{2}+ \<   (\zeta_A \vartheta _{i A_1}  ) ''\zeta_A \vartheta _{i A_1}  ,v^2\> .
\end{align*}\normalsize
Then we conclude the following, \small
\begin{align*}
   B_1^{(i)} &= -  \frac{3}{2} \| (\omega ^{(i)} )'\| _{L^2}^{2}  -\frac{c}{2}  \|\omega ^{(i)} \| _{L^2}^{2}+   B _{11}^{(i)} \text{ with }
B _{11}^{(i)}:=
\< ((\zeta_A \vartheta _{i A_1}  ) ')^2-\frac{1}2 (\zeta_A \vartheta _{i A_1}  ) '' \zeta_A \vartheta _{i A_1} ,v^2 \> .
\end{align*} \normalsize
We have
\begin{align*}
   B _{11}^{(i)}&= \<  \(  ( \zeta_A     ')^2-\frac{1}2  \zeta_A   \zeta_A ''\) \vartheta _{i A_1}     ^2 ,   v^2 \> +  \<    \zeta_A   '  \vartheta _{i A_1}  '  v,\zeta_A \vartheta _{i A_1}   v  \>
\\& +
 \<   ( \vartheta _{i A_1}    ')^2-\frac{1}2   \vartheta _{i A_1}    '' \vartheta _{i A_1}  ,  \zeta_A ^2v^2 \> =:\sum _{k=1}^{3} B _{11k}^{(i)}.
\end{align*}
From $|\zeta_A ^{(k)}|\lesssim A ^{-k}\zeta_A $ we have
\begin{align*}
  | B _{111}^{(i)}|\lesssim A ^{-2}  \|\omega ^{(i)} \| _{L^2}^{2}
\end{align*}
while from $A\gg A_1$
\begin{align}\nonumber
  | B _{112}^{(i)}   +  B _{113}^{(i)}  |& \lesssim A ^{-1}  \(  \|\omega ^{(1)} \| _{L^2}^{2} +\|\omega ^{(2)} \| _{L^2}^{2} \)      + A _1 ^{-2}\| v \| ^{2}_{L^2(A_1,2A_1)}\\& \lesssim A _1 ^{-2} \(  \|\omega ^{(1)} \| _{L^2}^{2} +\|\omega ^{(2)} \| _{L^2}^{2} \). \label{eq:B1-1}
\end{align}
Moving to  $B_2 ^{(i)}$, we claim that
for any fixed small constant $\varepsilon _1>0$
\begin{align}
  \label{eq:B2i1}& \left  | B_2^{(1)}\right | \lesssim  \varepsilon _1  \| v  \|_{  { \Sigma }_{1A A_1}} ^2+A_1 ^{-100} \| v  \|_{  { \Sigma }_{2A A_1}}^2 + \varepsilon _1^{-1} \| v\| _{\widetilde{\Sigma}}^2 \text{  and  }\\&  \label{eq:B2i2}  \left  | B_2^{(2)}\right | \lesssim   A _1 ^{-100}\(   \| v  \|_{  { \Sigma }_{1A A_1}} ^2+ \| v  \|_{  { \Sigma }_{2A A_1}} ^2 \)  .
\end{align}
To prove \eqref{eq:B2i1} and \eqref{eq:B2i2}   we    write
\begin{align*}
 B_2 ^{(i)}=-\< \( f( \phi  _{c} + v  )  -  f( \phi  _{c} )-f(v) \) ' ,\varphi _{iAA_1} v  \> - \< \(  f(v) \) ' ,\varphi _{iAA_1} v  \>=: B_{21}^{(i)}+B_{22}^{(i)} .
\end{align*}
From
\begin{align*}
 B_{22} ^{(i)} &= -\< \(  f(v) \) ' , \varphi _{iAA_1}v  \>  = \<    f(v)v ' , \varphi _{iAA_1}    \> +  \<    f(v)v   ,\zeta_A ^2 \vartheta _{i A_1}^2  \> = - \<  F(v)-  f(v)v , \zeta_A ^2 \vartheta _{i A_1}^2    \>
\end{align*}
 using \eqref{eq:estvH1}  we obtain
\begin{align}
  \label{eq:nonlinv}
  |B_{22} ^{(i)}|\lesssim   B ^{\frac{p-1}{2}} \delta ^{p-1}  \| v  \|_{  { \Sigma }_{iA A_1}}^2.
\end{align}
 Setting
\begin{align}&
   \Omega_{1,t  } :=\{x\in \R\ |\  2|  v(t,x)| \le  \phi   _{c (t)}(x) \}  \text{ and }\label{eq:partition2}\\&  \Omega_{2,t  } :=\R\setminus \Omega_{1,t  }  =\{x\in \R\ |\ 2|   v(t,x)|> \phi   _{c (t)}(x)\}  \nonumber
\end{align}
we split
\begin{align*}
  B_{21}^{(i)}= \sum _{k=1,2}  \<   f( \phi  _{c} + v  )  -  f( \phi  _{c} )-f(v)   , ( \varphi _{iAA_1} v)'  \>_{L^2(\Omega_{k,t  })}=: \sum _{k=1,2} B_{21k}^{(i)} .
\end{align*}
 We  further split
\begin{align*}
  B_{212}^{(i)}&=- \<    f( \phi  _{c} )    , ( \varphi _{iAA_1} v)'  \>_{L^2(\Omega_{2,t  })}+ \int _0^1 \<    f'( v+s\phi  _{c} ) \phi  _{c}   , ( \varphi _{iAA_1} v)'  \>_{L^2(\Omega_{2,t  })}   ds=: B_{2121}^{(i)}+B_{2122}^{(i)}.
\end{align*}
By  $\phi  _{c}^{p}\lesssim \phi  _{c}^{\frac{p-1}{2} } |v|^{\frac{p-1}{2}} |v|$  in $\Omega_{2,t  }$  and by \eqref{eq:estvH1},
 we have
\begin{align*}
   |B_{2121}^{(i)}|&\lesssim  \int_{\R}\phi  _{c}^{\frac{p-1}{2} }|v|^{\frac{p-1 }{2}+1}\zeta_{A} \zeta_{A}^{-1}  \left |   \varphi _{iAA_1} \zeta_{A}^{-1} (\zeta_{A}v) '  +\zeta_{A}v(\varphi _{iAA_1} \zeta_{A}^{-1})'   \right | dx \\&\lesssim B ^{\frac{p-1}{4}}\delta ^{\frac{p-1 }{2} }  \| \zeta_{A}v\| _{H^1}^2 \lesssim B ^{\frac{p-1}{4}}\delta ^{\frac{p-1 }{2} }(  \| v  \|_{  { \Sigma }_{1A A_1}} ^2+ \| v  \|_{  { \Sigma }_{2A A_1}}^2  ) .
\end{align*}
By  $\phi  _{c}\lesssim  |v|^{1 -\frac{p-1}{2} }  \phi  _{c}^{\frac{p-1}{2}} $    in $\Omega_{2,t  }$  and by $1= \vartheta _{1A_1} + \vartheta _{2A_1} $ ,  we get the following,
\begin{align*}
   |B_{2122}^{(i)}|&\lesssim  \int_{\R} |v| ^{p-1}\phi  _{c}\left |  \varphi _{iAA_1} \zeta_{A}^{-1} (\zeta_{A}v) '  +\zeta_{A}v(\varphi _{iAA_1} \zeta_{A}^{-1})'   \right | dx\\& \lesssim B ^{\frac{p-1}{4}} \delta ^{\frac{p-1 }{2} }  \int_{\R} |v| \zeta_{A} \zeta_{A}^{-1}  \phi  _{c}^{\frac{p-1}{2}} \left |   \varphi _{iAA_1} \zeta_{A}^{-1}  (\zeta_{A}v) '  +\zeta_{A}v(\varphi _{iAA_1} \zeta_{A}^{-1})'   \right | dx
  \\&  \lesssim  B ^{\frac{p-1}{4}} \delta ^{\frac{p-1 }{2} }  \| ( \vartheta _{1A_1} + \vartheta _{2A_1} )\zeta_{A}v\| _{H^1}^2\lesssim B ^{\frac{p-1}{4}} \delta ^{\frac{p-1 }{2} }  (  \| v  \|_{  { \Sigma }_{1A A_1}} ^2+ \| v  \|_{  { \Sigma }_{2A A_1}}^2  ),
\end{align*}
where we used   \eqref{eq:estvH1}.
 Next we split
\begin{align*}
   B_{211}^{(i)}&=\<    f'( \phi  _{c}   )    v   , (\varphi _{iAA_1} v)'  \>_{L^2(\Omega_{1,t  })}     +\int _{0}^{1} \<  \left [ f'( \phi  _{c} + sv  ) - f'( \phi  _{c}   ) \right ]     v   , (\varphi _{iAA_1} v)')  \>_{L^2(\Omega_{1,t  })} ds
  \\& - \int _{0}^{1} \< f'(sv)v   , (\varphi _{iAA_1} v)'  \>_{L^2(\Omega_{1,t  })} ds =: \sum _{k=1}^{3}B_{211k}^{(i)}.
\end{align*}
We have  \small
\begin{align}\nonumber
  |B_{2113}^{(i)}|&\lesssim \<  |v|^p   , |\varphi _{iAA_1} v '  + \zeta _A ^2 \vartheta _{i  A_1 } ^2 v| \>_{L^2(\Omega_{1,t  })}
   \lesssim  \< \phi _c ^{\frac{p-1}{2}} |v| ^{\frac{p-1}{2}}    |v|   , |  \zeta_A ^{-1}\varphi _{iAA_1}  \zeta_A v '  + \zeta _A ^2 \vartheta _{i  A_1} ^2 v| \>_{L^2(\Omega_{1,t  })} \\& \lesssim  \|v\| _{L^\infty (\R )}^{\frac{p-1}{2}} \|  \zeta_A v\| _{H^1} ^2 \lesssim  B ^{\frac{p-1}{2}}\delta  ^{\frac{p-1}{2}}   (  \| v  \|_{  { \Sigma }_{1A A_1}} ^2+ \| v  \|_{  { \Sigma }_{2A A1}}^2  ).\label{eq:B2113}
\end{align}
\normalsize
By concavity  $|f'( \phi  _{c} + sv  ) - f'( \phi  _{c}   )|\le |f'(   sv  )|$  (in the proof    $1<p\le 2$, Remark \ref{rem:udermain}) and as in \eqref{eq:B2113}
\begin{align*}
  |B_{2112}^{(i)}|&\lesssim  B ^{\frac{p-1}{2}}\delta  ^{\frac{p-1}{2}}   (  \| v  \|_{  { \Sigma }_{1A A_1}} ^2+ \| v  \|_{  { \Sigma }_{2A A_1}}^2  ).
\end{align*}
With $B_{2111}^{(i)}$  we make a clear distinction between the two cases $i=1,2$. For $i=1$, for a preassigned  fixed small constant $\varepsilon _1>0$, we have the following, which completes the proof of \eqref{eq:B2i1},
\begin{align}\label{eq:B2for1}
  | B_{2111}^{(1)}|&= \left |\<    f'( \phi  _{c}   )    v   ,  (\vartheta  _{1 A_1}   + \vartheta  _{2 A_1}   )    \varphi _{1AA_1} v '   +  \zeta_A^2\vartheta _{1A_1}^2 v \>_{L^2(\Omega_{1,t  })}\right |\\& \lesssim   \varepsilon _1 \| v  \|_{  { \Sigma }_{1A A_1}}^2  +A_1 ^{-100} \| v  \|_{  { \Sigma }_{2A A_1}}^2 + \varepsilon _1^{-1} \| v\| _{\widetilde{\Sigma}}^2 \nonumber
\end{align}
    where we exploit  the exponential decay at infinity of $f'( \phi  _{c}   )$ and
\begin{align}
  \label{eq:suppvarphi1}\supp  \varphi _{2AA_1}\subseteq \supp  \vartheta _{2A_1}\subseteq [A_1,+\infty ).
\end{align}
 In the $i=2$ case once again  by \eqref{eq:suppvarphi1}  we obtain the following, which yields \eqref{eq:B2i2},
\begin{align}\label{eq:B2for2}
  | B_{2111}^{(2)}|&=\left  |\<    f'( \phi  _{c}   )    v   ,  (\vartheta  _{1 A_1}   + \vartheta  _{2 A_1}   ) \varphi _{2AA_1} v '   +  \zeta_A^2\vartheta _{2A_1}^2 v \>_{L^2(\Omega_{1,t  } \cap \{ x\ge A_1    \}     )}\right | \\& \lesssim  A _1 ^{-100} \( \| v\| _{\widetilde{\Sigma}}^2 +  \| v  \|_{  { \Sigma }_{1A A_1}}^2  +\| v  \|_{  { \Sigma }_{2A A_1}}^2\) \lesssim  A _1 ^{-100} \(   \| v  \|_{  { \Sigma }_{1A A_1}}^2  +\| v  \|_{  { \Sigma }_{2A A_1}}^2\) .\nonumber
\end{align}
  We examine  the remaining terms in \eqref{eq:vir2}. By   \eqref{eq:estvH1}  and \eqref{eq:discrest1stat}   we have
\begin{align*}
  |B_3^{(i)}|&=  2 ^{-1}|(\dot \gamma -c ) |   \| v  \|_{  { \Sigma }_{iA A_1}}^2 \lesssim \delta   \| v  \|_{  { \Sigma }_{iA A_1}}^2 .
\end{align*}
Finally, exploiting \eqref{eq:suppvarphi1} and the exponential decay  at infinity of $ \xi _k [c]$  for $k=1,2$,  we have
\begin{align*}&
   |B_4^{(i)}|+ |B_5^{(i)}|\lesssim     A _1 ^{ -100 (i-1)}\( |\dot \gamma -c|^2+    |\dot  c|^2+  \| v \| _{ \widetilde{{\Sigma}} } ^2 \).
\end{align*}
 Bootstrapping
  gives the following and  with  \eqref{eq:discrest1stat} yields  \eqref{eq:lem:1stV1}   proving Lemma \ref{lem:1stV1},
\begin{align*}
   \| v \| _{  { \Sigma }_{1A A_1}    }^2& \lesssim     \widetilde{C}\left [  \dot{\mathcal{I}}_1   +  \|    v\| _{\widetilde{\Sigma}} ^2  +|\dot \gamma   -c| ^2+    |\dot     c| ^2 +A_1 ^{-2}  \| v \| _{  { \Sigma }_{2A A_1}    }^2 \right ] \text{ and}\\ \| v \| _{  { \Sigma }_{2A A_1}    }^2& \lesssim     \widetilde{C}\left [  \dot{\mathcal{I}}_2  + A_1 ^{-2} \(  |\dot \gamma   -c| ^2+    |\dot     c| ^2\) +A_1 ^{-2} \| v \| _{  { \Sigma }_{1A A_1}    }^2 \right ] .
 \end{align*}

\end{proof}

\textit{Proof of Lemma \ref {prop:1virial}.} We integrate  in the time  interval $I$ the inequalities in \eqref{eq:lem:1stV1} and we use
\begin{align*}&
  |\mathcal{I}_{i} (t)|\le \| \varphi _{iAA_1}\| _{L^\infty (\R ) }\| v(t) \| _{L^2 (\R ) }^2 \lesssim A  B \delta ^2 \text{  for $i=1,2$.   }
\end{align*}
By Lemmas  \ref{lem:lemdscrt}     and \ref{lem:1stV1},    for $A_1 ^{10}=B^2\gg B$    the following ensues, proving Lemma \ref {prop:1virial},
\begin{align*}&
   \| v \| _{L^2(I, { \Sigma }_{1A A_1} )} \lesssim  \sqrt{A B} \delta  +      \| v \| _{L^2(I, \widetilde{\Sigma}   )}   +  A_1 ^{-1}    \| v \| _{L^2(I, { \Sigma }_{2A A_1} )}   \text{ and} \\& \| v \| _{L^2(I, { \Sigma }_{2A A_1} )} \lesssim  \sqrt{A B} \delta     + A_1 ^{-1}   \| v \| _{L^2(I, { \Sigma }_{1A A_1} )} .
\end{align*}
 \qed

To complete the proof of {Proposition} \ref{prop:continuation} we are left
with  the proof of Lemma
\ref{prop:smooth11}, which will take sections  \ref{sec:disp}--\ref{sec:proofsmooth2}. For this we will first state  some linear  dispersive and smoothing  estimates for the linearization operator in \ref{eq:Lin}. We will then use this linear theory to  prove  Lemma
\ref{prop:smooth11}. Finally, we will prove the linear estimates.

\section{Dispersion and Kato smoothing  for the linearized equation }\label{sec:disp}

We will need  the    dispersive estimates for the linearized  operator in \eqref{eq:Lin}  due to Pego and Weinstein \cite{PegoWei2}.
  We consider the operators
\begin{align}\label{eq:projections}
  P_{[c]} := \sum _{i=1,2}   \xi _i [c]    \< \eta _i [c],  \cdot \>
 \text{ and } Q_{[c]}=1- P_{[c]} .
\end{align} Here $c=1$ and we will  omit the $[c]$
These are the operators associated to the spectral decomposition
\begin{align}
  \label{eq:specdec} L^2 _a = N_g\( \mathcal{L} \) \oplus N_g^\perp \( \mathcal{L}^* \)
\end{align}
where the perpendicularity is  with $ N_g  \( \mathcal{L}^*  \) \subseteq  L^2 _{-a}$   in terms
of the $L^2$ inner product in \eqref{eq:mass}. Then we quote Pego and Weinstein \cite[Theorem 4.2]{PegoWei2} to which we refer for the proof, where here $c=1$. Notice that the following follows for all  $ p\in (1,5) $ by Proposition \ref{prop:noeigh}.

\begin{theorem}\label{thm:pwdisp}For
   $a\in (0, \sqrt{1/3})$    and $ p\in (1,5)$  there are   constants $C( a,p)>0$  and $ b(a,p)>0$ s.t.
   \begin{align}\label{eq:pwdisp1}
      \| e^{  t \mathcal{L} } Q  v_0  \| _{L^2_a}\le C( a,p)  \<t\>^{-1/2} e^{-b(a,p) t}\|   Q v_0\| _{L^2_a}\text{ for all }v_0\in L^2_a.
   \end{align}

 \end{theorem}
\qed

Besides the above dispersive estimate, we   will need the following version of the Kato smoothing, which we will prove in this paper. Notice that in \eqref{eq:smooth1} we use in the left hand side  the weighted spaces in \eqref{eq:Lebwight} while in \eqref{eq:smooth11} we have the standard $ L^1 ( \R  )$ space. Later we will need to examine in detail why in
\eqref{eq:smooth1} and \eqref{eq:smooth21}  instead of $ L^1\( \R\)$ we need  the space $  L^{1,1}_+$ defined in \eqref{eq:L11+}.

\begin{proposition}\label{prop:smooth} For  any $p\in (1,5)$   there is a constant $C(  p)>0$ such that
  \begin{align}\label{eq:smooth1} &
      \left \| \int _0 ^t    e^{  (t-s) \mathcal{L}  } Q  f(s) ds  \right \| _{L^2 \( \R _+ ,    \widetilde{\Sigma}\) }\le C(  p)  \|     f \| _{ L^{ 1, 1}_+ \( \R   ,    L^{ 2}\( \R _+\)\) }\text{ for all }f\in L^{ 1, 1}_+ \( \R   ,    L^{ 2}\( \R _+\)\)   \text{ and }  \\&   \label{eq:smooth11}    \left \| \int _0 ^t    e^{  (t-s) \mathcal{L}  } Q \partial _x f(s) ds  \right \| _{L^2 \( \R _+ ,    \widetilde{\Sigma}\) }\le C(  p)  \|     f \| _{L^{ 1 } \( \R   ,    L^{ 2}\( \R _+\)\) }\text{ for all }f\in L^{ 1 } \( \R   ,    L^{ 2}\( \R _+\)\) .
   \end{align}
\end{proposition}
\begin{proof} The two estimates are proved similarly.  
  Pego and Weinstein \cite[p. 328]{PegoWei2} show  that     $e^{  t \mathcal{L}  }$  for $t\ge 0$  is a $C_0$ semigroup  in $L^2(\R )$ and also in the $L^2_a$ in {Theorem} \ref{thm:pwdisp}. Then, see \cite[p. 240]{MR617913}, 
  \begin{equation}\label{eq:forresol}
    R_{\mathcal{L} }(z) Q w_0 =-\int _{0}^{+\infty} e^{-z t} e^{  t \mathcal{L} } Q  w_0 dt \text{ for any }\Re z > - b(a,p), 
  \end{equation}  
  gives the resolvent in  $L^2_a$.
   By Theorem    \ref{thm:pwdisp} and by Plancherel, see  \cite[XIII.7, Lemma 1]{ReedSimon4}, we have 
  \begin{align*}
  \frac{1}{2\pi}  \|  R_{\mathcal{L} }(\im \lambda ) Q w_0  \| _{L^2( \R  , L^2_a) }   =   \| e^{  t \mathcal{L} } Q  w_0  \| _{L^2( \R _+, L^2_a) }\le C(a,p) \|  w_0 \| _{ L^2_a} 
  \end{align*}
  for some $C(a,p)>0$.
   Since  for $a \in (0, \kappa )$
we have $L ^{ 2}   _{a}  \subset \widetilde{\Sigma}$ 
we conclude also
\begin{align*}
  \frac{1}{2\pi}  \|  R_{\mathcal{L} }(\im \lambda ) Q w_0  \| _{L^2( \R  , \widetilde{\Sigma} ) }  & =   \| e^{  t \mathcal{L} } Q  w_0  \| _{L^2( \R _+, \widetilde{\Sigma}) } \le    \| e^{  t \mathcal{L} } Q  w_0  \| _{L^2( \R _+, L^2_a) }\le C(a,p) \|  w_0 \| _{ L^2_a}  \ .
  \end{align*}
For   $f\in C_c^{ \infty}([0, +\infty )) $  we have 
 \begin{align*} &
   \int _\R  dt e^{-\im \lambda t} \int _0 ^t    e^{  (t-s) \mathcal{L} } Q  f(s) ds =   \int _{0}^{+\infty} ds  e^{-\im \lambda s}   \int _{0}^{+\infty}  e^{-\im \lambda t}   e^{  t \mathcal{L} } Q  f(s) dt\\& =-  \int _{0}^{+\infty}    e^{-\im \lambda s}  R_{\mathcal{L} }(\im \lambda) Q  f(s) ds =-  R_{\mathcal{L} }     ( \im \lambda   )  Q  \widehat{f}(\lambda) 
 \end{align*}
where, in the first equality, the exchange of integrals is allowed by the absolute integrability in  $(t,s)$, thanks to    \eqref{eq:pwdisp1} and the compact support of $f$.
  Then by Plancherel 
  \begin{equation}\label{eq:fourtran}
     \begin{aligned}
        &
     2\pi  \left \| \int _0 ^t    e^{  (t-s) \mathcal{L} } Q  f(s) ds  \right \| _{L^2 \( \R _+ ,    \widetilde{\Sigma}\) } =  \left \|     R_{\mathcal{L} }     ( \im \lambda   )  Q  \widehat{f}(\lambda)   \right \| _{L^2 \( \R  ,    \widetilde{\Sigma}\) }\\&\le   \left \|    R_{\mathcal{L} }     ( \im \lambda   )  Q    \right \| _{L^\infty \( \R  ,    \mathcal{L} \( L^{ 1, 1}_+\( \R\),   \widetilde{\Sigma} \)  \) }   \left \|         {f}    \right \| _{L^{1,1}_+ \( \R  ,    L^2\( \R_+\)\) } .
     \end{aligned}
  \end{equation}
   The proof  of \eqref{eq:smooth1} is completed by means of Proposition \ref{prop:smooth2}, see below. The proof of \eqref{eq:smooth11} is similar.

\end{proof}

The following is proved in \S \ref{sec:proofsmooth2} using the information on the Jost functions derived in \S \ref{sec:jost}. We will examine in considerable detail the reason why  in \eqref{eq:smooth21}  we need to use the space  $L^{ 1, 1}_+\( \R\)$  instead of the larger space  $L^{ 1 }\( \R\)$.
\begin{proposition}\label{prop:smooth2} For   any $p\in (1,5)$   there is a constant $C(  p)>0$  such that
  \begin{align}\label{eq:smooth21}&
     \sup  _{\lambda \in \R \backslash \{ 0  \}} \left \|    R_{\mathcal{L} }     ( \im \lambda   )  Q    \right \| _{ \mathcal{L}\( L^{ 1, 1}_+\( \R\),   \widetilde{\Sigma} \)   }\le C(  p) \text{ and } \\&    \label{eq:smooth221}
     \sup  _{\lambda \in \R \backslash \{ 0  \}} \left \|    R_{\mathcal{L} }     ( \im \lambda    )  Q  \partial _x  \right \| _{ \mathcal{L}\( L^1\( \R\),   \widetilde{\Sigma} \)   }\le C(  p).
   \end{align}
\end{proposition}


\section{Bounding $\| v \| _{\widetilde{\Sigma}}$  by Kato smoothing}\label{sec:smooth}

We prove Lemma \ref{prop:smooth11} assuming Proposition \ref{prop:smooth2}.
Multiplying equation \eqref{eq:gKdV3}  by $\zeta_{B}$
we obtain
\begin{align}     \nonumber
    \partial _t   \(  \zeta_{B}v    \)  &= \mathcal{L}    \(  \zeta_{B}v    \)  + [\zeta_{B}, \mathcal{L} ]v
\\& + (\dot \gamma -c )  \zeta_{B} v'+
( \dot \gamma -c ) \ \zeta_{B}  \xi _1[c] + \dot c \  \zeta_{B} \xi _2[c] \nonumber
\\&
+    \zeta_{B} \(   \mathcal{L}_{[c ]} - \mathcal{L}   \) v     -   \zeta_{B} \partial _x    \( f( \phi  _{c} + v  )  -  f( \phi  _{c} )-  f'( \phi  _{c} ) v\)     .   \label{eq:gKdV14}
\end{align}
Next we consider the spectral decomposition
 \begin{align}\label{eq:defw}
    \zeta_{B}v=P \zeta_{B}v+ w \text { where }  w:= Q \zeta_{B}v .
 \end{align}
The first and preliminary observation is the following.

\begin{lemma}
  \label{lem:disczbv} We have
\begin{align}
   \|  P \zeta_{B}v \| _{L^2}\lesssim B ^{-1}  (  \| v  \|_{  { \Sigma }_{1A A_1}} + \| v  \|_{  { \Sigma }_{2A A_1}}  ) .\label{eq:disczbv}
\end{align}
\end{lemma}
\begin{proof}

 We have the following,
\begin{align*}
  P \zeta_{B}v =   P_{[c]}\zeta_{B}v + \(  P -  P_{[c ]}  \) \zeta_{B}v .
\end{align*}
Then, since  $\< \eta _1 [c],  \zeta_{B}v \>=0$  by  \eqref{61}, we obtain  \small
\begin{align*}
  P_{[c]}\zeta_{B}v =   \xi _2 [c]    \< \eta _2 [c],  \zeta_{B}v \> =   \xi _2 [c]    \< \eta _2 [c],  ( \zeta_{B} -1)(\vartheta _{1A_1} + \vartheta _{2A_1}) v \> =O\(  B ^{-1} (  \| v  \|_{  { \Sigma }_{1A A_1}} + \| v  \|_{  { \Sigma }_{2A A_1}}  ) \) .
\end{align*}\normalsize
 Next we focus on
\begin{align*}
   \(  P_{[c ]} -  P  \) \zeta_{B}v  = \sum _{i=1,2}\left [  {(\xi _i [c]  -\xi _i )  \< \eta _i [c], \zeta_{B}v\>}
+\xi _i  \< \eta _i [c]-\eta _i   , \zeta_{B}v\>\right ] .
\end{align*}
Here    \eqref{61} implies   $\< \eta _1 [c], \zeta_{B}v\> =0$ and
\begin{align}\label{eq:mod111}
   |\< \eta _2 [c], \zeta_{B}v\>|\lesssim   B ^{-1}  \|   \zeta _A v \| _{ L^2}  =  B ^{-1}  \|   \zeta _A (\vartheta _{1A_1} + \vartheta _{2A_1}) v \| _{ L^2} \le  B ^{-1}  (  \| v  \|_{  { \Sigma }_{1A A_1}} + \| v  \|_{  { \Sigma }_{2A A_1}}  ) .
\end{align}
It is elementary that
\begin{align}\label {eq:disczbv1}
   \| \xi _2   \< \eta _2 [c]-\eta _2  , \zeta_{B}v\> \| _{L^2}\lesssim |\< \eta _2 [c]-\eta _2  , \zeta_{B}v\>|
\lesssim |c-1| \| \zeta_{B} v \| _{L^2}\lesssim \delta  (  \| v  \|_{  { \Sigma }_{1A A_1}} + \| v  \|_{  { \Sigma }_{2A A_1}}  ) .
\end{align}
We have the following, where we exploit  \eqref{61},
\begin{align}\label {eq:disczbv2}
   \| \xi _1  \< \eta _1 [c]-\eta _1   , \zeta_{B}v\> \| _{L^2}\lesssim |\< \eta _1 [c]-\eta _1  , \zeta_{B}v\>|  =|\<  \eta _1 , \zeta_{B}v\>| .
\end{align}
Recall by \eqref{eq:ker2} that
\begin{align*}
  \<  \eta _1 , \zeta_{B}v\> = \theta _1(1) \<  \int _{-\infty} ^{x } \left . \partial _c \phi _c (y)\right |_{c=1} dy , \zeta_{B}v\>  +\frac{\theta _2 (1)}{\theta _3 (1)} \< \eta _2 ,   \zeta_{B}v\> =: I+II.
\end{align*}
Then like in \eqref{eq:mod111}  and  \eqref{eq:disczbv1}
\begin{align}\label {eq:disczbv3}
  |II|\lesssim  |\< \eta _2 [c], \zeta_{B}v\>|+|\< \eta _2   -  \eta _2 [c] ,   \zeta_{B}v\>| \lesssim B ^{-1}  (  \| v  \|_{  { \Sigma }_{1A A_1}} + \| v  \|_{  { \Sigma }_{2A A_1}}  ) .
\end{align}
We have
\begin{align*}
  I&=  \<  \theta _1(1)\int _{-\infty} ^{x } \Lambda _p \phi   (y)  dy     -     \frac{ \theta _1(1)\int _{\R } \Lambda _p \phi   (y)  dy}{   \cancel{\theta _1(c)}\int _{\R } \partial _c \phi _c (y)   dy }  \cancel{\theta _1(c)}  \int _{-\infty} ^{x }   \partial _c \phi _c (y)  dy     , \zeta_{B}v\> \\& =\theta _1(1)\<  \int _{-\infty} ^{x } \Lambda _p \phi   (y)  dy     -
  \frac{\int _{\R } \Lambda _p \phi   (y)  dy}{c^{\frac{2-p}{p-1} -\frac{1}{2}}\int _{\R } \Lambda _p \phi   (y)  dy}
  c^{\frac{2-p}{p-1}  }
  \int _{-\infty} ^{x }   \Lambda _p \phi   (\sqrt{c}y)  dy      , \zeta_{B}v\>
 \\& =  \theta _1(1)\<  \int _{\sqrt{c}x} ^{x } \Lambda _p \phi   (y)  dy        , \zeta_{B}v\> .
\end{align*}
Since now $ | \Lambda _p \phi   (x)| \lesssim e^{-\frac{|x|}{2}}$ then for $x>0$ and for $c >1$
\begin{align*}
  \left |   \int _{\sqrt{c}x} ^{x } \Lambda _p \phi   (y)  dy  \right | \lesssim 2 e^{-\frac{x}{2}}\( 1-e^{- \frac{\sqrt{c}-1}{2} x}\) \le 2 e^{-\frac{x}{2}}  \frac{\sqrt{c}-1}{2} x
\end{align*}
and proceeding similarly there is a $K>0$  such that for $c$ close enough to 1 we have
\begin{align}\label {eq:disczbv4}
  \left |   \int _{\sqrt{c}x} ^{x } \Lambda _p \phi   (y)  dy  \right | \le K e^{-\frac{|x|}{4}} |c-1|  \text{      for all $  x\in \R$}
\end{align}
where we exploit also the fact that the left hand side is even in $x$.
From \eqref{eq:disczbv4}  we conclude
\begin{align*}
  |I |\lesssim    \delta (  \| v  \|_{  { \Sigma }_{1A A_1}} + \| v  \|_{  { \Sigma }_{2A A_1}}  )
\end{align*}
which, together with \eqref{eq:disczbv3}, yields the following which completes the proof of Lemma \ref{lem:disczbv},
\begin{align*}
 |\<  \eta _1  , \zeta_{B}v\>|& \lesssim \delta  (  \| v  \|_{  { \Sigma }_{1A A_1}} + \| v  \|_{  { \Sigma }_{2A A_1}}  ) .
\end{align*}

\end{proof}

We now focus on the   $w$   in \eqref{eq:defw}. Applying the projection $Q $ in \S \ref{sec:disp} (case $c=1$) to   \eqref{eq:gKdV14}
gives
\begin{align}
   \dot w  &= \mathcal{L}   w + Q [\zeta_{B}, \mathcal{L} ]v \nonumber
\\& + (\dot \gamma -c ) Q  \zeta_{B} v'+
( \dot \gamma -c ) Q \zeta_{B}  \xi _1[c] + \dot c \ Q  \zeta_{B} \xi _2[c]  \label{eq:gKdV22}
\\&
+   Q  \zeta_{B} \(   \mathcal{L}_{[c ]} - \mathcal{L}   \) v   \label{eq:gKdV23}  \\&  - Q  \zeta_{B} \partial _x    \( f( \phi  _{c} + v  )  -  f( \phi  _{c} )-  f'( \phi  _{c} ) v\)     .   \label{eq:gKdV24}
\end{align}
We next consider
\begin{align}\label{eq:expv1}
  w&=e^{  t \mathcal{L} }  w(0) +   \int _{0}^t e^{ (t-s) \mathcal{L} }Q [\zeta_{B}, \mathcal{L} ]v ds  +     \int _{0}^t e^{  (t-s)\mathcal{L} } \text{lines \eqref{eq:gKdV22}--\eqref{eq:gKdV24}} ds.
\end{align}

\begin{lemma}\label{lem:w0} We have
   \begin{align}\label{eq:w0}
      \| e^{  t \mathcal{L} }  w(0)\| _{L^2 (\R _+, \widetilde{\Sigma})}\le  \sqrt{\delta}.
   \end{align}
\end{lemma}

\begin{proof} By Theorem \ref{thm:pwdisp} we have
\begin{align*}
  \| e^{  t \mathcal{L} }  w(0)\| _{\widetilde{\Sigma}}& \le \| e^{  t \mathcal{L} }  w(0)\| _{L^{2}_{\frac{1}{B}}}
  \lesssim  C( 1/B)\<t\>^{-1/2} e^{-b\(1/B ,p\)  t}\|   w(0)\| _{L^{2}_{\frac{1}{B}}}\\& \le  C( 1/B)\<t\>^{-1/2} e^{-b\(1/B ,p\)  t}  \(  \| e^{\frac{x}{B} } \zeta_{B}v(0)\| _{L^2} + \| e^{\frac{x}{B} } P \zeta_{B}v(0)\| _{L^2} \)
\end{align*}
From \eqref{eq:estvH1}, we have
\begin{align*}
   \| e^{\frac{x}{B} } \zeta_{B}v(0)\| _{L^2} \lesssim  \|  v(0)\| _{L^2} \lesssim B ^{1/2}\delta
\end{align*}
and  by \eqref{eq:disczbv} similarly
\begin{align*}
   \| e^{\frac{x}{B} } P \zeta_{B}v(0)\| _{L^2} \lesssim  \sum _{i=1}^{2} |  \< \eta _i , \zeta_{B}v(0) \>   | \lesssim B ^{1/2}   \| v(0) \| _{\widetilde{\Sigma}}  \lesssim B   \delta  .
\end{align*}
 Hence the following whose $L^2(\R )$ norm in time,  for $\delta $ very small in terms of $B$, yields   \eqref{eq:w0}
 \begin{align*}
  \| e^{  t \mathcal{L} }  w(0)\| _{\widetilde{\Sigma}}& \lesssim  C( 1/B)\<t\>^{-1/2} e^{-b\(1/B ,p\)  t}  B ^{1/2}\delta  .
\end{align*}

\end{proof}

In Lemmas \ref{lem:gKdV22}--\ref{lem:gKdV21easy}, we will use \eqref{eq:smooth1} with $L^{1,1}$ instead of $L^{1,1}_+$, which obviously holds from $\|\cdot\|_{L^{1,1}_+}\leq \|\cdot\|_{L^{1,1}}$.
The only place we need $L^{1,1}_+$ instead of $L^{1,1}$ will be Lemma \ref{lem:gKdV21hard}.

\begin{lemma}\label{lem:gKdV22} We have
   \begin{align}\label{eq:gKdV221}
     \left  \|  \int _{0}^t e^{  (t-s)\mathcal{L} } \text{line  \eqref{eq:gKdV22}} ds\right \| _{L^2 (I, \widetilde{\Sigma})}\lesssim   \|  \dot \gamma -c \| _{L^2 \( I  \) }+ \|  \dot  c \| _{L^2 \( I  \) }     .
   \end{align}
\end{lemma}
\begin{proof}
  Applying Proposition \ref{prop:smooth}   we  conclude immediately
 \begin{align*} &
      \left \| \int _0 ^t    e^{  (t-s) \mathcal{L} }   ( \dot \gamma -c ) Q  \zeta_{B}  \xi _1[c] ds  \right \| _{L^2 \( I ,    \widetilde{\Sigma}\) }\lesssim      \|   ( \dot \gamma -c )   \zeta_{B}  \xi _1[c]  \| _{ L^{ 1, 1} \( \R   ,    L^{ 2}\( I\)\) } \\& \lesssim \|   ( \dot \gamma -c )    \xi _1[c]  \| _{ L^{ 2,2} \( \R   ,    L^{ 2}\( I\)\) } =
        \|   ( \dot \gamma -c )    \xi _1[c]  \| _{ L^{ 2} \( I    ,    L^{ 2,2}\( \R \)\) }\\& \le \|   ( \dot \gamma -c )    \| _{ L^{ 2} \( I  \) } \|        \xi _1[c]  \| _{  L^{ \infty} \( I    ,    L^{ 2,2}\( \R \)\)} \lesssim  \|   ( \dot \gamma -c )    \| _{ L^{ 2} \( I  \) }
  \end{align*}
  and similarly
  \begin{align*}  \\& \left \| \int _0 ^t     e^{  (t-s) \mathcal{L} }   \  \dot c \ Q  \zeta_{B} \xi _2[c] ds  \right \| _{L^2 \( I ,    \widetilde{\Sigma}\) }\lesssim      \|  \dot  c \| _{L^2 \( I  \) }.
  \end{align*}
  Similarly applying Proposition \ref{prop:smooth} and \eqref{eq:estvH1} for the remaining term we have
  \begin{align*} &
      \left \| \int _0 ^t    e^{  (t-s) \mathcal{L} }  (\dot \gamma -c ) Q  \zeta_{B} v' ds  \right \| _{L^2 \( I ,    \widetilde{\Sigma}\) }\lesssim     \|   (\dot \gamma -c )   \zeta_{B} v'  \| _{ L^{ 1, 1} \( \R   ,    L^{ 2}\( I\)\) }
      \\& \lesssim  \| \zeta_{B}  \| _{    L ^{2,1} (\R ) }
       \|  (\dot \gamma -c ) v'\| _{ L^{ 2} \( \R   ,    L^{ 2}\( I\)\) } \lesssim  B ^{\frac{3}{2}}    \|  (\dot \gamma -c )  \| _{   L^{ 2}\( I\)  }   \|   v'\| _{ L^{ \infty} \( I  ,    L^{ 2}\( \R \)\) }\\&  \lesssim B^2 \delta  \|  (\dot \gamma -c )  \| _{ L^{ 2} \( I\)  }  .
  \end{align*}

\end{proof}

\begin{lemma}\label{lem:gKdV23} We have
   \begin{align}\label{eq:gKdV231}
     \left  \|  \int _{0}^t e^{  (t-s)\mathcal{L} } \text{line  \eqref{eq:gKdV23}} ds\right \| _{L^2 (I, \widetilde{\Sigma})}\lesssim  \sqrt{ \delta}  (\| v \| _{L^2(I,  { \Sigma }_{1A A_1}   )}+\| v \| _{L^2(I,  { \Sigma }_{2A A_1}   )}).
   \end{align}
\end{lemma}
\begin{proof}
  We use again Proposition \ref{prop:smooth}  to obtain
  \begin{align*}&
     \left  \|  \int _{0}^t e^{  (t-s)\mathcal{L} }  Q  \zeta_{B} \(   \mathcal{L}_{[c ]} - \mathcal{L}   \) v ds\right \| _{L^2 (I, \widetilde{\Sigma})}\\& \lesssim    \left  \| \zeta_{B}  ( c-1)    v '  \right \| _{ L^{ 1, 1} \( \R   ,    L^{ 2}\( I\)\) } +       \left  \| \zeta_{B} \partial _x   \( \phi ^{p-1}  -\phi ^{p-1}_{c}  \)   v \right \| _{ L^{ 1, 1} \( \R   ,    L^{ 2}\( I\)\) }.
  \end{align*}
  From \eqref{eq:mod2}, we have
  \begin{align*}
      \| c-1\|_  {L^\infty (\R  )}  \left  \| \zeta_{B}  v '\right \| _{L^{ 1, 1} \( \R   ,    L^{ 2}\( I\)\)}  &\lesssim
       \delta  \left  \| \zeta_{2B}  \right \| _{L^{ 2, 1} \( \R   \)}
      \left  \| \zeta_{2B}  v '\right \| _{L^{ 2} \( \R   ,    L^{ 2}\( I\)\)}
      \\& \lesssim   B^{3/2}\delta \left  \| \frac{\zeta_{2B}}{\zeta _A} \zeta _A  (\vartheta _{1A_1} + \vartheta _{2A_1})  v '\right \| _{L^{ 2} \( I   ,    L^{ 2}\( \R \)\)}
       \\&  \lesssim B^{3/2}\delta   (  \| v  \|_{ L^2 (I  , { \Sigma }_{1A A_1})} + \| v  \|_{ L^2 (I  , { \Sigma }_{2A A_1})}  )
  \end{align*}
  and similarly, it is possible to prove the following, which yields the desired inequality in \eqref{eq:gKdV231}
 \begin{align*}
    \left  \| \zeta_{B} \partial _x   \( \phi ^{p-1}  -\phi ^{p-1}_{c}  \)  v \right \| _{L^{ 1, 1} \( \R   ,    L^{ 2}\( I\)\)} \lesssim  \delta   (  \| v  \|_{ L^2 (I  , { \Sigma }_{1A A_1})} + \| v  \|_{ L^2 (I  , { \Sigma }_{2A A_1})}  ) .
 \end{align*}

\end{proof}

Next we have the following. Recall, see Remark \ref{rem:udermain},  that in the proof we are assuming $1<p<2$.

\begin{lemma}\label{lem:gKdV24} We have
   \begin{align}\label{eq:gKdV241}
     \left  \|  \int _{0}^t e^{  (t-s)\mathcal{L} } \text{line  \eqref{eq:gKdV24}} ds\right \| _{L^2 (I, \widetilde{\Sigma})}\lesssim   \delta ^{   \frac{p-1}{2}   } (\| v \| _{L^2(I,  { \Sigma }_{1A A_1}   )}+\| v \| _{L^2(I,  { \Sigma }_{2A A_1}   )}) .
   \end{align}
\end{lemma}
\begin{proof} By Proposition \ref{prop:smooth} the left hand side of \eqref{eq:gKdV241} is up to a fixed constant bounded  above by
\begin{align*}&
   A_1 +   A_2  \text{ where}
\\&
   A_1 =  \left  \| \zeta_{B}   \( f'( \phi  _{c} + v  ) -f' ( \phi  _{c} ) -   f''( \phi  _{c} ) v\)\phi '  _{c}    \right \| _{L^{ 1, 1} \( \R   ,    L^{ 2}\( I\)\)} \text{ and}\\& A_2 =  \left  \| \zeta_{B}   \( f'( \phi  _{c} + v  ) -f' ( \phi  _{c}  \) v'   \right \| _{L^{ 1, 1} \( \R   ,    L^{ 2}\( I\)\)} .
\end{align*}
 We have
\begin{align*}
  A_1&\le\left  \| \zeta_{B}   \( f'( \phi  _{c} + v  ) -f' ( \phi  _{c} ) -   f''( \phi  _{c} ) v\)\phi '  _{c}    \right \| _{L^{1,1}L^2 (\{ |v|\ll \phi  _{c} \})} \\& + \left  \| \zeta_{B}   \( f'( \phi  _{c} + v  ) -f' ( \phi  _{c} ) -   f''( \phi  _{c} ) v\)\phi '  _{c}    \right \| _{L^{1,1}L^2 (\{ |v|\gtrsim  \phi  _{c} \})} =: A_{11}+A_{12}.
\end{align*}
By $\zeta_{2B}=\sqrt{ \zeta_{B}} $, \eqref{eq:relABg}, \eqref{eq:normA}  and \eqref{eq:estvH1}
we have \small
\begin{align*}
       A_{11} &\le  \int _{0}^1 \left  \| \zeta_{B}   \( f''( \phi  _{c} + s v  )  -   f''( \phi  _{c} )  \)  v  \phi '  _{c}    \right \| _{L^{1,1}L^2 (\{ |v|\ll \phi  _{c} \})} ds \\&  \le \int _{[0,1]^2}  \left  \| \zeta_{B}    f'''( \phi  _{c} + \tau s v  )    v ^2  \phi '  _{c}    \right \| _{L^{1,1}L^2 (\{ |v|\ll \phi  _{c} \})} s \ ds d\tau  \\& \lesssim  \left  \| \zeta_{B}    \phi  _{c}  ^{p-3}    v ^2  \phi '  _{c}    \right \| _{L^{1,1}L^2 (\{ |v|\ll \phi  _{c} \})} \lesssim  \left  \| \zeta_{B}       | v |  ^p      \right \| _{L^{1,1} (\R, L^2(I) )}   \\& \le  \left  \| \zeta_{2B}            \right \| _{L^{2,1} (\R ) )}   \left\| \zeta_{2B}       | v |  ^p      \right \| _{L^{2} (\R, L^2(I) )}
          \lesssim B ^{3/2} \left\| \zeta_{2B}       | v |  ^p      \right \| _{L^{2} (I, L^2(\R ) )}
        \\& \le  B ^{3/2}
        \| v \| ^{p-1}_{L^{\infty} (I\times \R ) )}   \left  \|\zeta_{2B}       v      \right \| _{L^2(I, L^2 (\R ))}   \lesssim B^{\frac{p}{2}+1} \delta ^{p-1} \left  \| \frac{\zeta_{2B}}{\zeta_{A} }  \zeta_{A}   (\vartheta _{1A_1} + \vartheta _{2A_1})      v        \right \| _{L^{2} (I, L^2(\R ) )}
   \\& \le   B^{\frac{p}{2}+1} \delta ^{p-1}   \left  \| \zeta_{A}   (\vartheta _{1A_1} + \vartheta _{2A_1})     v        \right \| _{L^2 (\R )}
   \le  B^{\frac{p}{2}  +1  } \delta ^{p-1}(\| v \| _{L^2(I,  { \Sigma }_{1A A_1}   )}+\| v \| _{L^2(I,  { \Sigma }_{2A A_1}   )}) ,
\end{align*}\normalsize
where we used  $|\phi'  _{c}|\lesssim \phi  _{c}$ and here we are using $p<3$. Similarly
\begin{align*}
   A_{12} & \lesssim
   \left  \| \zeta_{B}  |v|^{p }        \right \| _{L^{1,1} (\R, L^2(I) ) }\lesssim B^{\frac{p}{2}  +1  } \delta ^{p-1}(\| v \| _{L^2(I,  { \Sigma }_{1A A_1}   )}+\| v \| _{L^2(I,  { \Sigma }_{2A A_1}   )}).
\end{align*}
We have
\begin{align*}
  A_2&\le   \left  \| \zeta_{B}   \( f'( \phi  _{c} + v  ) -f' ( \phi  _{c}  \) v'   \right \| _{L^{1,1}L^2 (\{ |v|\ll \phi  _{c} \})} \\& +\left  \| \zeta_{B}   \( f'( \phi  _{c} + v  ) -f' ( \phi  _{c}  \) v'   \right \| _{L^{1,1}L^2 (\{ |v|\gtrsim  \phi  _{c} \})} =: A_{21}+A_{22}.
\end{align*}
We have, proceeding like above in the last inequality,
\begin{align*}
   A_{21}&\lesssim \int _0 ^1 \left  \| \zeta_{B}     f''( \phi  _{c} + sv  ) v  v'   \right \| _{L^{1,1}L^2 (\{ |v|\ll \phi  _{c} \})} ds\lesssim  \left  \| \zeta_{B}      \phi  _{c} ^{p-2} v  v'   \right \| _{L^{1,1}L^2 (\{ |v|\ll \phi  _{c} \})} \\& \lesssim  \left  \| \zeta_{B}      |v| ^{  p-1   }  v'   \right \| _{L^{1,1} (\R, L^2(I) )}
   \lesssim  \left  \| \zeta_{2B}            \right \| _{L^{2,1} (\R ) )}   \left  \| \frac{\zeta_{2B}}{\zeta _A}  \zeta_{A}   (\vartheta _{1A_1} + \vartheta _{2A_1})      |v| ^{  p-1   }  v'   \right \| _{L^{2} (\R, L^2(I) )} \\& \lesssim
   B ^{3/2} \| v \| ^{p-1}_{L^{\infty} (I\times \R ) )}(\| v \| _{L^2(I,  { \Sigma }_{1A A_1}   )}+\| v \| _{L^2(I,  { \Sigma }_{2A A_1}   )})\\&
   \lesssim B^{\frac{p}{2}+1} \delta ^{  p-1   }(\| v \| _{L^2(I,  { \Sigma }_{1A A_1}   )}+\| v \| _{L^2(I,  { \Sigma }_{2A A_1}   )})
\end{align*}
and the following, which with the above inequalities and with \eqref{eq:estvH1} yields immediately \eqref{eq:gKdV241}
\begin{align*}
   A_{22}&\lesssim  \| \zeta_{B}      |v| ^{p-1}  v'     \| _{L^{1,1} (\R, L^2(I) )}\lesssim B^{\frac{p}{2}+1} \delta ^{p-1}(\| v \| _{L^2(I,  { \Sigma }_{1A A_1}   )}+\| v \| _{L^2(I,  { \Sigma }_{2A A_1}   )}).
\end{align*}

\end{proof}

The most delicate contributor in the formula \eqref{eq:gKdV22}--\eqref{eq:gKdV24} is the commutator.
We have
\begin{align}&\label{eq:comm}
   [\zeta_{B}, \mathcal{L}] v= 3 \zeta_{B}'v''
        -c \zeta_{B}'v +\zeta_{B}'  f' ( \phi  _{c}    )v + 3 \zeta_{B}''v' +  \zeta_{B}'''v .
\end{align}
The troublesome  term  is  $-c \zeta_{B}'v $, which we set aside for the moment.

\begin{lemma}\label{lem:gKdV21easy} We have
   \begin{align}\label{eq:gKdV21easy1}
     \left  \|  \int _{0}^t e^{  (t-s)\mathcal{L}} Q \( [\zeta_{B}, \mathcal{L}]  + c \zeta_{B}'v\)      ds\right \| _{L^2 (I, \widetilde{\Sigma})}\lesssim   B ^{-1/2}   (\| v \| _{L^2(I,  { \Sigma }_{1A A_1}   )}+\| v \| _{L^2(I,  { \Sigma }_{2A A_1}   )}).
   \end{align}
 \end{lemma}
\begin{proof}
We   omit initially the term $3 \zeta_{B}'v''$
in \eqref{eq:comm}  proceeding like above  by   {Proposition} \ref{prop:smooth}
\begin{align*}&
    \left  \|  \int _{0}^t e^{  (t-s)\mathcal{L}} Q \(\zeta_{B}'  f' ( \phi  _{c}    )v + 3 \zeta_{B}''v' +  \zeta_{B}'''v   \)      ds\right \| _{L^2 (I, \widetilde{\Sigma})}\\&   \lesssim
     \| \zeta_{B}'  f' ( \phi  _{c}    )v + 3 \zeta_{B}''v' +  \zeta_{B}'''v  \| _{L^{ 1, 1} \( \R   ,    L^{ 2}\( I\)\)} \lesssim B ^{-1/2}  (\| v \| _{L^2(I,  { \Sigma }_{1A A_1}   )}+\| v \| _{L^2(I,  { \Sigma }_{2A A_1}   )}) .
  \end{align*}
Writing   $ \zeta_{B}'v'' = \partial _x   ( \zeta_{B}'v')- \zeta_{B}''v' $,   using    for the first time also \eqref {eq:smooth11} we obtain
\begin{align*}&
    \left  \|  \int _{0}^t e^{  (t-s)\mathcal{L}} Q \( \partial _x   ( \zeta_{B}'v')- \zeta_{B}''v'   \)      ds\right \| _{L^2 (I, \widetilde{\Sigma})}\\&   \lesssim
     \|  \zeta_{B}'v'  \| _{L^{ 1 } \( \R   ,    L^{ 2}\( I\)\) } +     \|  \zeta_{B}''v'  \| _{L^{ 1, 1} \( \R   ,    L^{ 2}\( I\)\) }
     \lesssim B ^{-1/2} (\| v \| _{L^2(I,  { \Sigma }_{1A A_1}   )}+\| v \| _{L^2(I,  { \Sigma }_{2A A_1}   )}) .
  \end{align*}

\end{proof}


\begin{lemma}\label{lem:gKdV21hard} We have
   \begin{align}\label{eq:gKdV21hard1}
     \left  \|  \int _{0}^t e^{  (t-s)\mathcal{L}} Q   \zeta_{B}'v      ds\right \| _{L^2 (I, \widetilde{\Sigma})}\lesssim   A_ 1^{3/2}B ^{-1 }    \| v \| _{L^2(I,  { \Sigma }_{1A A_1}   )}+B ^{ 1/2 } \| v \| _{L^2(I,  { \Sigma }_{2A A_1}   )} .
   \end{align}
 \end{lemma}

\begin{proof}
    Split
$
v=\vartheta_{1A_1}v+\vartheta_{2A_1}v$.
Then \eqref{eq:smooth1} gives
\begin{align}
\left\|
\int_0^t e^{(t-s)\mathcal L}Q(\zeta_B'v)(s)\,ds
\right\|_{L^2(I,\widetilde{\Sigma})}
&\lesssim
\|\langle x^+\rangle\zeta_B'\vartheta_{1A_1}v\|_{L^1_xL^2_t}
+
\|\langle x^+\rangle\zeta_B'\vartheta_{2A_1}v\|_{L^1_xL^2_t}
\notag\\
&\lesssim
\|\langle x^+\rangle\zeta_B'1_{(-\infty,2A_1]}\zeta_A^{-1}\|_{L^2_x}
\|\vartheta_{1A_1}v\|_{L^2(I,\Sigma_A)}
\notag\\
&\quad
+
\|\langle x^+\rangle\zeta_B'1_{[A_1,\infty)}\zeta_A^{-1}\|_{L^2_x}
\|\vartheta_{2A_1}v\|_{L^2(I,\Sigma_A)}.
\label{eq:smoothing-short-Fcom-2}
\end{align}
Since 
$|\zeta_B'|\lesssim B^{-1}$,
\begin{align}
\|\langle x^+\rangle\zeta_B'1_{(-\infty,2A_1]}\zeta_A^{-1}\|_{L^2_x}
\lesssim
B^{-1/2}+A_1^{3/2}B^{-1}\lesssim A_1^{3/2}B^{-1},
\label{eq:smoothing-short-Fcom-3}
\end{align}
where we have used $A_1=B^{3/5}$.
For the second term, we have
\begin{align}
\|\langle x^+\rangle\zeta_B'1_{[A_1,\infty)}\zeta_A^{-1}\|_{L^2_x}^2
&=
\int_{A_1}^\infty
\langle x\rangle^2 |\zeta_B'(x)|^2
\zeta_A(x)^{-2}\,dx.
\label{eq:smoothing-short-Fcom-4a}
\end{align}
Moreover, for $x\ge0$,
$
\langle x\rangle |\zeta_B'(x)|
\lesssim
e^{-x/(2B)}$,
and hence
\begin{align}
\|\langle x^+ \rangle\zeta_B'\vartheta_{2A_1}\zeta_A^{-1}\|_{L^2_x}^2
\lesssim
\int_{A_1}^\infty e^{-x/B+2x/A}\,dx
\lesssim
B\,e^{-cA_1/B}
\lesssim
B.
\label{eq:smoothing-short-Fcom-4c}
\end{align}
Therefore
\begin{align}
\|\langle x^+\rangle\zeta_B'\vartheta_{2A_1}\zeta_A^{-1}\|_{L^2_x}
\lesssim
B^{1/2}.
\label{eq:smoothing-short-Fcom-4}
\end{align}
Combining \eqref{eq:smoothing-short-Fcom-2}, \eqref{eq:smoothing-short-Fcom-3} and \eqref{eq:smoothing-short-Fcom-4}, we have the conclusion.
\end{proof}

\textit{Proof of Lemma \ref{prop:smooth11}.} From   Lemmas \ref{lem:w0}--\ref{lem:gKdV21hard}  and Lemma \ref{lem:lemdscrt}
we have
\begin{align*}
  \| w \| _{L^2(I, \widetilde{\Sigma}   )} \lesssim A_ 1^{3/2}B ^{-1 }    \| v \| _{L^2(I,  { \Sigma }_{1A A_1}   )}+B ^{ 1/2 } \| v \| _{L^2(I,  { \Sigma }_{2A A_1}   )} +\sqrt{\delta} .
\end{align*}
 Using  \eqref{eq:relABg}, decomposition \eqref{eq:defw} and inequality \eqref{eq:disczbv}, we have the following
\begin{align*}
   \| v \| _{L^2(I, \widetilde{\Sigma}   )} &\le      \| (1-\zeta_{B}) v \| _{L^2(I, \widetilde{\Sigma}   )}  + \| \zeta_{B} v \| _{L^2(I, \widetilde{\Sigma}   )}   \\& \lesssim B ^{-1} \(  \| v \| _{L^2(I,  { \Sigma }_{1A A_1}   )} + \| v \| _{L^2(I,  { \Sigma }_{2A A_1}   )}    \) + \| P \zeta_{B}v  \| _{L^2(I, \widetilde{\Sigma}   )} +\|   w \| _{L^2(I, \widetilde{\Sigma}   )}  \\& \lesssim  A_ 1^{3/2}B ^{-1 }    \| v \| _{L^2(I,  { \Sigma }_{1A A_1}   )}+B ^{ 1/2 } \| v \| _{L^2(I,  { \Sigma }_{2A A_1}   )}  +\sqrt{\delta}.
\end{align*}
\qed

We
completed the proof of Theorem \ref{thm:main} up to {Proposition} \ref{prop:noeigh}, proved in  Section \ref{sec:noeigh},   {Proposition} \ref{prop:smooth2} proved in the next two Sections  \ref{sec:jost} and  \ref{sec:proofsmooth2}.

\section{The Jost functions}\label{sec:jost}

We  look for appropriate solutions of
\begin{align}
  \label{eq:lineq1}  \mathcal{L}_{[c ]}u=\lambda u \text{ for } \lambda \in \C    .
\end{align}
We will consider only case $c=1$ dropping the subscript  $[c ]$. 
The third order equation  \eqref{eq:lineq1} has characteristic equation
\begin{equation}\label{eq:mus}
    \mu^3 - \mu + \lambda =0.
\end{equation}
We summarize the properties of the solutions of \eqref{eq:mus}.

\begin{lemma}\label{lem:characteristic-roots}
There exist three analytic roots of \eqref{eq:mus}
    $\mu_j:\Omega:=\left\{
        \lambda\in\C:
        |\Re\lambda|<\frac{1}{3\sqrt{3}}
    \right\}\to\C$,
    $(j=1,2,3)$,
which are
uniquely determined by
\begin{align*}
    \mu_1(0)=-1,
    \qquad
    \mu_2(0)=0,
    \qquad
    \mu_3(0)=1.
\end{align*}
The relations between the roots and the coefficients of
\eqref{eq:mus} are
\begin{align}\label{eq:sum_of_mus}
    \mu_1+\mu_2+\mu_3
    &=
    0,\\
    \mu_1\mu_2+\mu_1\mu_3+\mu_2\mu_3
    &=
    -1,\\
    \mu_1\mu_2\mu_3
    &=
    -\lambda.
\end{align}
For $0<\Re\lambda<\frac{1}{3\sqrt{3}}$,
\begin{align}\label{eq:stres0}
    \Re\mu_1
    <
    0
    <
    \Re\mu_2
    <
    \Re\mu_3.
\end{align}
Moreover, we can continuously extend $\mu_j$ on $\overline{\Omega}$ and there exists $\kappa>0$ such that
\begin{align}\label{eq:relmus}
    \Re\mu_1(\lambda)
    &\leq
    -\kappa,\\
    \Re\bigl(\mu_2(\lambda)-\mu_1(\lambda)\bigr)
    &\geq
    \kappa,\\
    \Re\bigl(\mu_3(\lambda)-\mu_2(\lambda)\bigr)
    &\geq
    \kappa
\end{align}
for all $\lambda\in\overline{\Omega}$.

For every $k\in\R$, if
    $\lambda
    =
    \im k(k^2+1)$,
then
\begin{align}\label{eq:detmu13}
    \mu_1(\lambda)
    =
    -\sqrt{1+\frac34k^2}
    -
    \frac{\im k}{2},\quad
    \mu_2(\lambda)
    =
    \im k,\quad
    \mu_3(\lambda)
    =
    \sqrt{1+\frac34k^2}
    -
    \frac{\im k}{2}.
\end{align}

Define
\begin{align}\label{eq:def_a_j_jost}
    a_j(\lambda)
    :=
    \prod_{\substack{1\leq \ell\leq3\\ \ell\neq j}}
    \bigl(
        \mu_j(\lambda)-\mu_\ell(\lambda)
    \bigr)^{-1},
    \qquad
    j=1,2,3.
\end{align}
Then, for every $n\in\N\cup\{0\}$,
\begin{align}\label{eq:muhighen}
    \left|
        \partial_\lambda^n\mu_j(\lambda)
    \right|
    &\lesssim_n
    \langle\lambda\rangle^{\frac13-n},\quad
    \left|
        \mu_j(\lambda)-\mu_\ell(\lambda)
    \right|
    \gtrsim
    \langle\lambda\rangle^{\frac13},
    \qquad
    j\neq\ell,\\
    \left|
        \partial_\lambda^n a_j(\lambda)
    \right|
    &\lesssim_n
    \langle\lambda\rangle^{-\frac23-n},\quad
    \left|
        \partial_\lambda^n
        \bigl(
            a_j(\lambda)\mu_j(\lambda)
        \bigr)
    \right|
    \lesssim_n
    \langle\lambda\rangle^{-\frac13-n}\label{eq:mua_bound}
\end{align}
for all $\lambda\in\Omega$ and $j=1,2,3$.
\end{lemma}

\begin{proof}
The discriminant is $4-27\lambda^2$, and hence the equation has no
multiple root on $\overline{\Omega}$, which is simply connected.
    Thus, the analytic roots with the
stated normalization follow from analytic perturbation theory; see
\cite[Theorem XII.2]{ReedSimon4}.

For $\lambda=\im k(k^2+1)$, substituting $\mu_2=\im k$ into
\eqref{eq:mus} and factoring the remaining quadratic factor gives
\eqref{eq:detmu13}. The ordering \eqref{eq:stres0}, the large
$|\lambda|$ behavior of the roots, and the corresponding separation
estimates are standard; see \cite[Section 2]{PegoWei2}. Combining the
large $|\lambda|$ estimates with \eqref{eq:detmu13} and compactness on
bounded subsets of $\Omega$ yields
\eqref{eq:relmus} and the second estimate in
\eqref{eq:muhighen}. The remaining estimates follow by repeatedly
differentiating \eqref{eq:mus} and the definition of $a_j$.
\end{proof}

In the following, we fix $\varepsilon\in (0,\frac{1}{3\sqrt{3}})$.

Given three scalar functions $y_j$ with $j=1,2,3$ we consider the Wronskian
    \begin{align}\label{eq:wronskdef}
      W[y_1,y_2,y_3]:=  \det \left(
                                 \begin{array}{ccc}
                                   y_1 & y_2 & y_3 \\
                                   y_1' & y_2' & y_3' \\
                                   y_1 ''& y_2 ''& y_3 ''\\
                                 \end{array}
                               \right).
    \end{align}
 For   the functions $y_j = e^{\mu_j(\lambda)  x} $   we set
\begin{equation}\label{eq:vandermonde}
    W _0\(\lambda\) :=
    W[y_1,y_2,y_3]
      = \prod _{1\leq i < j\leq 3} (\mu_i - \mu_j )    .
\end{equation}
By the variation of parameter formulas  the solutions of the nonhomogeneous equation
\begin{align}\label{eq:nonhom}
  y'''-y'+\lambda y=F
\end{align}
are \begin{align}\label{eq:solution_nonhomogeneous_equation}
   & y =
    \sum_{ {j}=1,2,3 } a_ {j}  e^{\mu_ {j}  x}\int \ e^{-\mu_ {j}  x } F(x)   d x  ,
\end{align}
where $a_j$ is given in \eqref{eq:def_a_j_jost}.
 A simple computation, see   \cite[Lemma 4.3]{PegoWei2},  yields the identity
    \begin{equation}\label{eq:sum_aj}
        \sum_{ {j}=1,2,3} a_ {j}  \equiv 0 .
    \end{equation}
The solutions  of \eqref{eq:lineq1}  are of the form \eqref{eq:solution_nonhomogeneous_equation}  for $F=-\(   f' ( \phi  _{c}    )y   \) '$ .
Since
\begin{align*}
   -a_ {j}  e^{\mu_ {j}  x}\int \ e^{-\mu_ {j}  x } \(   f' ( \phi   )y   \) '    d x  =  -a_ {j} f' ( \phi    )y  -  a_ {j}\mu_ {j} e^{\mu_ {j}  x }\int \ e^{-\mu_ {j}  x }    f' ( \phi    )y   d x,
\end{align*}
summing on $j$ and  using   \eqref{eq:sum_aj}  we have
\begin{align}
    & y =-
    \sum_{ {j}=1,2,3 } a_ {j}\mu_ {j} e^{\mu_ {j}  x }\int \ e^{-\mu_ {j}  x }    f' ( \phi    )y   d x . \label{eq:solnonhom}
\end{align}
The first Jost function of \eqref{eq:lineq1} is provided by the following, proved  like \cite[Lemma 1 p. 130]{DT1979}.
\begin{lemma}\label{lem:estimate_f1}
    There exists a  solution  $f_1(x , \lambda ) = e^{\mu_1 x} m_1(x , \lambda )$    of  \eqref{eq:lineq1}
    such that for any $\alpha, \beta \ge 0$  there is a constant $C_{\alpha\beta}$ such that   for all $\lambda \in \C $ such that $|\Re \lambda| \le \varepsilon  $   and $x\in \R$
    \begin{equation}\label{eq:m1_estimate_partial}
        \left |     \partial_\lambda^\beta \partial_x^{\alpha} \(   m_1(x , \lambda ) -1 \)  \right |  \leq C_{\alpha, \beta}  \<\lambda \> ^{-\frac{1}{3}   -\beta  }         \int _{x}^{+\infty }f' (\phi   (x') ) dx' \text{ . }
    \end{equation}

\end{lemma}
 \begin{proof} We seek a solution  to
     \begin{equation*}
          f_1(x , \lambda )  = e^{\mu_1 x} +  \sum_{ {j}=1,2,3 } a_ {j}  \mu _ {j}   e^{\mu_ {j}  x}  \int_x^{+\infty} e^{-\mu_{ {j}}  y}  f' ( \phi    (y)   ) f_1 (y , \lambda )  d y
    \end{equation*}
    or equivalently
    \begin{equation}\label{eq:integral_eq_m1}
          m_1(x , \lambda )  = 1 +       \int_x^{+\infty} q  (  x -y , \lambda)   f' ( \phi   (y)   ) m_1 (y , \lambda )  d y
    \end{equation}
    where
\begin{align*}
         q  (  x , \lambda  ) =   \sum_{ {j}=1,2,3 } a_ {j}  \mu _ {j}      e^{(\mu_{ {j}}   -\mu_{1} ) x}    .
     \end{align*}
  The solution of \eqref{eq:integral_eq_m1} can be expressed like  in Deift and Trubowitz \cite[Lemma 1]{DT1979} as
    \begin{equation}\label{eq:m1_sum}
        m_1(x , \lambda )= \sum_{n=0}^\infty   m_1(x , \lambda )_{[n]} \ ,
    \end{equation}
    where
    \begin{equation}
    \label{eq:Volterra_integral}
    \begin{aligned}
        &m_1(x , \lambda )_{[0]}   = 1 \ ,
        \\&
         m_1(x , \lambda )_{[n+1]}
         =      \int_x^{+\infty}  q (  x -y , \lambda)  f' ( \phi   (y)   ) m_1 (y , \lambda ))_{[n]}  d y  .
    \end{aligned}
    \end{equation}
   By   \eqref{eq:muhighen}  and  \eqref{eq:solution_nonhomogeneous_equation}  there exists a constant $C_*>0$ such that \begin{equation}\label{eq:estamu}
        \sup _{|\Re \lambda| \le \varepsilon} \sum _{ {j}=1,2,3} | a_ j \mu_j |  \le   C_*   \<\lambda \> ^{-\frac{1}{3}}.
    \end{equation}
 Then  like in \cite[Lemma 1]{DT1979} we obtain
     \begin{align*}
     \left |  m_1(x , \lambda )_{[n ]} \right | &\le     C_* ^n   \<\lambda \> ^{- \frac{n}{3}}      \int_{x\leq x_1\leq \ldots \leq x_n} | f' ( \phi   (x_1)   )   |   \ldots | f' ( \phi   (x_n)   )   |  d x_n \ldots  d x_1   \\& = \frac{1}{n!}  C_* ^n   \<\lambda \> ^{- \frac{n}{3}} \(  \int_x^\infty  | f' ( \phi  (y)   )   |   dy   \) ^n
  \end{align*}
    which yields  \eqref{eq:m1_estimate_partial} for $\alpha =\beta =0$. Differentiating \eqref{eq:integral_eq_m1} and integrating by parts  we obtain
    \begin{align}\nonumber
        \partial _x   m_1(x , \lambda )&  = -  q  (  0 , \lambda)  f' ( \phi  (x)   )   m_1 (x , \lambda ) + \int_x^{+\infty} \partial _x q  (  x -y , \lambda)   f' ( \phi  (y)   )   m_1 (y , \lambda )  d y
        \\& =   m_1'(x , \lambda ) _{[0]}   +      \int_x^{+\infty} q  (  x -y , \lambda)   f' ( \phi   (y)   )\partial _y  m_1 (y , \lambda )  d y   \text{ where}\label{eq:derm11}
       \\   m_1'(x , \lambda ) _{[0]}&:=      \int_x^\infty q  (  x -y , \lambda)   f'' ( \phi    (y)   ) \phi  ' (y)  m_1 (y , \lambda )  d y.\nonumber
    \end{align}
This can be solved as \eqref{eq:m1_sum} as a sum
    \begin{align*}
    \partial _x m_1(x , \lambda )= \sum_{n=0}^\infty   m_1'(x , \lambda )_{[n]}
    \end{align*}
    where  for $n\ge 0$
     \begin{equation}
    \nonumber
    \begin{aligned}
         m_1'(x , \lambda )_{[n+1]}
         =
              \int_x^{+\infty} q  (  x -y , \lambda)   f' ( \phi   (y)   ) m_1 '(y , \lambda ))_{[n]}  d y  .
    \end{aligned}
    \end{equation}
    Then
    \begin{align*}
      \left |  m_1'(x , \lambda )_{[0 ]} \right | &\le   \| m_1 (\cdot  , \lambda )\| _{L^\infty (\R )}    C_*
      \<\lambda \> ^{- \frac{1}{3}}    \int _x ^{+\infty}| f'' ( \phi   (y)   ) \phi   ' (y)| dy
    \end{align*}
    and   for $n\ge 1$  and   using that there exists $C_0>0 $ such that $| f'' ( \phi   (y)   ) \phi   ' (y)|\le C_0| f'  ( \phi   (y)   )  |$
    we have    \small
    \begin{align*}
     \left |  m_1'(x , \lambda )_{[n ]} \right | &\le         C_*  ^{n}   \<\lambda \> ^{- \frac{n }{3}}  \|m'_1(\cdot,\lambda)_{[0]}\|_{L^\infty}    \int_{x\leq x_1\leq \ldots \leq x_n} | f' ( \phi  (x_1)   )   |     \ldots | f' ( \phi  _{c} (x_n)   )   |  d x_n \ldots  d x_1   \\& =   \frac{1}{n!}C_0     C_*  ^{n+1}   \<\lambda \> ^{- \frac{n +1}{3}}    \(  \int_x^\infty  | f' ( \phi (y)   )   |   dy   \) ^{n+1}.
  \end{align*}
    \normalsize
   Then we obtain  \eqref{eq:m1_estimate_partial} for $\alpha =1$ and $\beta  =0$. Applying $\partial _\lambda $ to  \eqref{eq:integral_eq_m1} we obtain
     \begin{align*}
       \partial _\lambda m_1(x , \lambda ) & =      \int_x^\infty        \partial _\lambda q  (  x -y , \lambda) f' ( \phi   (y)   )   m_1 (y , \lambda )  d y \\& +         \int_x^\infty q  (  x -y , \lambda) f' ( \phi   (y)   )     \partial _\lambda m_1 (y , \lambda )  d y .
     \end{align*}
Then
   \begin{align*}
    \partial _\lambda m_1(x , \lambda )= \sum_{n=0}^\infty  \dot  m_1 (x , \lambda )_{[n]}
    \end{align*}
     where
     \begin{equation}
    \nonumber
    \begin{aligned}
        &   \dot m_1  (x , \lambda ) _{[0]}   =    \int_x^\infty  \partial _\lambda q  (\lambda , x-y )   f' ( \phi   (y)   ) m_1 (y , \lambda )  d y \text{ and}
        \\&
        \dot  m_1 (x , \lambda )_{[n+1]}
         =   \int_x^\infty q  (  x -y , \lambda) f' ( \phi   (y)   ) \dot  m_1 (y , \lambda ))_{[n]}  d y  .
    \end{aligned}
    \end{equation}
    We have for some constant  $C_1>0$
    \begin{align*}  &| \dot m_1  (x , \lambda ) _{[0]}|  \le | \partial _\lambda  (a_1\mu _1)           |    \| m_1 (\cdot  , \lambda )\| _{L^\infty (\R )}   \int_x^\infty      f' ( \phi   (y)   )   d y  \\&  +   \| m_1 (\cdot  , \lambda )\| _{L^\infty (\R )}  \sum_{ {j}=  2,3 }   \int_x^\infty   \left |     a_ {j}  \mu _ {j}   \partial _\lambda\(e^{-(\mu_{ {j}}   -\mu_{1} ) (x-y)}  \)      \right  |     f' ( \phi  (y)   )   d y    \\& \le C _1 \< \lambda \> ^{- \frac{4}{3}  }  \int_x^\infty      f' ( \phi    (y)   )   d y
    \end{align*}
  where we used \eqref{eq:muhighen}, \eqref{eq:mua_bound} and
  \begin{align} &
   \nonumber \left  |   (x-y)    e^{(\mu_{ {j}}   -\mu_{1} ) (x-y)}   \partial _\lambda  (\mu_{ {j}}   -\mu_{1} )       \right |
     \lesssim
     \left  |  (\mu_{ {j}}   -\mu_{1} )  (x-y)    e^{(\mu_{ {j}}   -\mu_{1} ) (x-y)}   \frac{\partial _\lambda  (\mu_{ {j}}   -\mu_{1} ) }{\mu_{ {j}}   -\mu_{1}}     \right | \\& \lesssim    \left  |    \frac{\partial _\lambda  (\mu_{ {j}}   -\mu_{1} ) }{\mu_{ {j}}   -\mu_{1}}     \right | \lesssim   \< \lambda \> ^{-1  } \text{ and  for   for all $y>x$}.\label{eq:expdl}
  \end{align}
   For $n\ge 1$ we have   for a constant $C_2>0$
   \begin{align*}
     \left |  \dot m_1 (x , \lambda )_{[n ]} \right | &\le  C _2\< \lambda \> ^{- \frac{4}{3}  }          C_*  ^{n }   \<\lambda \> ^{- \frac{n  }{3}}      \int_{x\leq x_1\leq \ldots \leq x_n} | f' ( \phi  (x_1)   )   |     \ldots | f' ( \phi    (x_n)   )   |  d x_n \ldots  d x_1   \\& =  \frac{1}{n!}  C _2\< \lambda \> ^{- \frac{4}{3}  }        C_*  ^{n }   \<\lambda \> ^{- \frac{n  }{3}}    \(  \int_x^\infty  | f' ( \phi (y)   )   |   dy   \) ^n .
  \end{align*}
   We hence get case $\alpha =0$ and $\beta =1$.
 For the general case we proceed by induction.
   Differentiating     \eqref{eq:integral_eq_m1} and integrating by parts   we obtain
   \begin{align*}
      \partial_x^{\alpha} \partial_\lambda^\beta  m_1(x , \lambda )  =     \widetilde{m}_1  (x , \lambda ) _{[0]} +     \int_x^{+\infty} q  (  x -y , \lambda)   f' ( \phi    (y)   )  \partial_y^{\alpha} \partial_\lambda^\beta m_1 (y , \lambda )  d y
   \end{align*}
   with the first term on the right of the form
   \begin{align}\label{eq:m1dergen}
     \widetilde{m}_1  (x , \lambda ) _{[0]}  &= \sum a_{\beta _1\beta_2}^{\alpha  _1\alpha _2}      \int_x^{+\infty} \partial _{\lambda } ^{ \beta _1 }   q  (  x -y , \lambda)   \partial _{y } ^{ \alpha _1 }  \(  f' ( \phi (y)   \)   \partial_y^{\alpha _2} \partial_\lambda^{\beta _2}  m_1 (y , \lambda )  d y,
   \end{align}
   where   $\alpha  _1+\alpha  _2=\alpha  $,    $\beta _1+\beta _2=\beta $  and $\alpha  _2 +\beta _2<  \alpha    +\beta$   and where the $a_{\beta _1\beta_2}^{\alpha  _1\alpha _2}$ are fixed constants. Formula \eqref{eq:m1dergen} can be obtained in elementary fashion by induction on $\alpha$ proceeding like for
   \eqref{eq:derm11}.
    We have
   like above  by induction
   \begin{align*}
     \left |    \partial _{\lambda } ^{ \beta _1 }   q_1 (  x -y , \lambda)   \partial _{y } ^{ \alpha _1 }  \(  f' ( \phi   (y)   \)   \partial_y^{\alpha _2} \partial_\lambda^{\beta _2}  m_1 (y , \lambda ) \right | \lesssim     \<\lambda \> ^{-\frac{1}{3} -  \beta  }  | \partial _{y } ^{ \alpha _1 }  \(  f' ( \phi  (y)   \)|
   \end{align*}
   so that  we get the following   which, fed in the Volterra equation resolution mechanism, yields   \eqref{eq:m1_estimate_partial}
   \begin{align*}
     \left |   \widetilde{m}_1  (x , \lambda ) _{[0]}  \right | \le  C ' _{\alpha\beta}      \<\lambda \> ^{-\frac{1}{3} -  \beta  }  \sum _{j=0}^{\alpha} \int _x ^{+\infty}    | \partial _{y } ^{ \alpha   }  \(  f' ( \phi  (y)   \)| dy  ,
   \end{align*}
   where we remark that this last integral can be bounded in terms of the integral in     \eqref{eq:m1_estimate_partial}.
 \end{proof}

The second  Jost function   of \eqref{eq:lineq1} is    obtained similarly.
\begin{lemma}\label{lem:estimate_f3}
    There exists a  solution  $f_3(x , \lambda ) = e^{\mu_3 x} m_3(x , \lambda )$    of  \eqref{eq:lineq1}
    such that for any $\alpha, \beta \ge 0$  there is a constant $C_{\alpha\beta}$ such that   for all   $\lambda \in \C $ such that $|\Re \lambda| \le \varepsilon  $  and $x\in \R$
    \begin{equation}\label{eq:m3_estimate_partial}
        \left |     \partial_x^{\alpha}    \partial_\lambda^\beta \(   m_3(x , \lambda ) -1 \)  \right |  \leq C_{\alpha, \beta}  \<\lambda \> ^{-\frac{1}{3}   -\beta  }         \int _{-\infty }^{x }f' (\phi   (x') ) dx' \text{  .}
    \end{equation}

\end{lemma}  \begin{proof}
 The proof is similar to that of  Lemma \ref{lem:estimate_f1}. We only remark   that  here the integral equation is
 \begin{equation}\label{eq:integral_eq_m3}
          m_3(x , \lambda )  = 1 -   \sum_{ {j}=1,2,3 } a_ {j}  \mu _ {j}     \int_ {-\infty} ^xe^{ (\mu_{ {j}}   -\mu_{3} ) (x-y)}  f' ( \phi   (y)   ) m_3 (y , \lambda )  d y  .
    \end{equation}
\end{proof}

Observe that for $\lambda 	\in 	\im \R $  we have that $f_{1}(-x, \lambda  )$ solves equation \eqref{eq:lineq1}  with $\lambda$ replaced by $-\lambda$. Then, by the asymptotics, we conclude
\begin{align} \label{eq:symmetryf13}
  f_{1}(-x, \lambda )=f_{3}( x, -\lambda ) \text{ for all $x\in \R$ and $\lambda \in \im \R$.}
\end{align}
A third Jost function   of \eqref{eq:lineq1} is the following one. As we will see later, there is no as good control globally for $x\in \R$ of this function.
\begin{lemma}\label{lem:estimate_tildef2}
There exists a Jost function $ \tilde{f}_2(x , \lambda )= e^{\mu_2x} \tilde{m}_2(x , \lambda ) $  and $ x_0\in \R $ such that  for all $\lambda \in \C $ such that $|\Re \lambda| \le \varepsilon  $ we have
\begin{align}\label{eq:estimate_tildef21}&
\|  \tilde{m}_2(\cdot  , \lambda )\| _{L^\infty\( (  -\infty, x_0 )  \) } \le 2
\text{ and }\\&
\lim_{x\to-\infty} \tilde{m}_2(x , \lambda ) =1 \ .\label{eq:estimate_tildef22}
\end{align}
There are  constants $C_{\alpha\beta} = C_{\alpha\beta} (c,x_0)$ such that for $\alpha +\beta \ge 1  $
\begin{align} \label{eq:estimate_tildef232}
  \|  \partial_x^{\alpha} \partial_\lambda^\beta   \tilde{m}_2(\cdot  , \lambda )\| _{L^\infty\( (  -\infty, x_0 )  \) } \le C_{\alpha, \beta}  \<\lambda \> ^{-\frac{1}{3}   -\beta  } \text{  } .
\end{align}
Finally, there are constants  $C_{\alpha\beta} = C_{\alpha\beta} ( x_0)$ such that for $x\le x_0$ and for     $|\Re \lambda| \le \varepsilon$,
\begin{align}\label{eq:estimate_tildef233} &  |  \partial_x^{\alpha} \partial_\lambda^\beta\( \tilde{m}_2(x  , \lambda ) -1\)    |\le C_{\alpha\beta}   \<\lambda \> ^{-\frac{1}{3}  -\beta    } e^{-   (p-1)  |x|}.
\end{align}

\end{lemma}
\begin{proof}
  For  $ x_0\in \R $   consider the equation
\begin{equation}\label{eq:m2tilde}
\tilde{m}_2(x , \lambda ) = 1 + T  \tilde{m_2}(\cdot , \lambda )  \ ,
\end{equation}
where
\begin{align*}
Tg (x) &:= - a_1 \mu_1 \int_{-\infty}^x e^{ ( \mu_1 -\mu_2)   (  x-y) } f' ( \phi   (y))  g(y) d y
\\&
- a_2  \mu_2 \int_{-\infty}^x   f' ( \phi   (y))  g(y) d y
+
a_3 \mu_3 \int_x^{x_0}   e^{ ( \mu_3 -\mu_2)   (  x-y) } f' ( \phi    (y))  g(y) d y\ .
\end{align*}
 Then, from Lemma \ref{lem:characteristic-roots}, there exists a $ x_1\in \R $ such that for $ x_0\leq x_1 $ we have that $$ \| T\| _{\mathcal{L}\(L^\infty\( (  -\infty, x_0 )  \)\)}\leq 1/2 $$ for all  $|\Re \lambda| \le \varepsilon  $.  
 Then  $ \| (1-T) ^{-1}\| _{\mathcal{L}\(L^\infty\( (  -\infty, x_0 )  \)\)} \le 2 $  
 yielding \eqref{eq:estimate_tildef21}. 
 The limit \eqref{eq:estimate_tildef22} is obtained using dominated convergence.
  Differentiating \eqref{eq:m2tilde} we obtain 
 \begin{align*}&
   (1-T)  \partial_x^{\alpha} \partial_\lambda^\beta \tilde{m}_2(x , \lambda )   =  A_1+A_2  \text{ where }
   \\ & A_1 = -\sum a_{\beta _1\beta_2}^{\alpha  _1\alpha _2}      \int_{-\infty}^x \partial _{\lambda } ^{ \beta _1 }   \(  a_1 \mu_1   e^{ ( \mu_1 -\mu_2)   (  x-y) } +  a_2 \mu_2   \)   \partial _{y } ^{ \alpha _1 }  \(  f' ( \phi  (y)   \)   \partial_y^{\alpha _2} \partial_\lambda^{\beta _2}   \tilde{m}_2 (y , \lambda )  d y \\& A_2 =  \sum a_{\beta _1\beta_2}^{\alpha  _1\alpha _2}  \int_x^{x_0}  \partial _{\lambda } ^{ \beta _1 }   \(  a_3 \mu_3   e^{ ( \mu_3 -\mu_2)   (  x-y) }     \) \partial _{y } ^{ \alpha _1 }  \(  f' ( \phi  (y)   \)   \partial_y^{\alpha _2} \partial_\lambda^{\beta _2}   \tilde{m}_2 (y , \lambda )  d y
 \end{align*} like in Lemma \ref{lem:estimate_f1}.
  Since $A_1$ and $A_2$ are bounded above by the right hand side of \eqref{eq:estimate_tildef232}, inverting $1-T$ we obtain \eqref{eq:estimate_tildef232}.  Finally, the last statement follows in an elementary way from
\begin{align*}
  |Tg (x)| &\lesssim   \<\lambda \> ^{-\frac{1}{3}     }   \| g\| _{L^\infty (-\infty , x_0)} \left [ \int_{-\infty}^x    e^{(p-1)y}    d y \right .
\\&\left .  +  e^{- \Re ( \mu_3 -\mu_2)    \frac{|x|}{2} }
  \int _{-\infty}^{x_0}  e^{(p-1) y}      d y  +  e^{-\frac {(p-1)}{2}|x|}  \int _{x}^{\frac{x}{2}}  e^{- \Re ( \mu_3 -\mu_2)   |x-y| }   dy   \right ] .
\end{align*}

\end{proof}
We will need also additional solutions.
\begin{lemma}\label{lem:extrajost}
\begin{enumerate}
\item There exists $ F_1(x , \lambda )= e^{\mu_1 x} M_1(x , \lambda )  $  and   $ x_1\in \R $ such that for any  $\alpha, \beta \ge 0$  there is a constant $C_{\alpha\beta  }$   such that  for all $x\le x_1$  and   $|\Re \lambda| \le \varepsilon $
\begin{align*}
 |    \partial_\lambda^\beta \partial_x^{\alpha} \( M_1(x , \lambda )-1\)  |  \leq  C_{\alpha, \beta}  \<\lambda \> ^{-\frac{1}{3}     - \beta  }  e^{-   (p-1)  |x|}
\text{ and }
\lim_{x\to -\infty} M_1(x , \lambda ) =1.
\end{align*}

\item There exists $ F_2(x , \lambda )= e^{\mu_2 x} M_2(x , \lambda ) $   and   $ x_2\in \R $ such that for any  $\alpha, \beta \ge 0$  there is a constant $C_{\alpha\beta  }$ such that  for all $x\ge x_2$  and      $|\Re \lambda| \le \varepsilon $
\begin{align*}
 |    \partial_\lambda^\beta \partial_x^{\alpha} \( M_2(x , \lambda )-1\)  |  \leq  C_{\alpha, \beta}  \<\lambda \> ^{-\frac{1}{3}     - \beta  }  e^{-   (p-1)   |x|}
\text{ and }
\lim_{x\to +\infty} M_2(x , \lambda ) =1.
\end{align*}

\item There exists $ F_3(x , \lambda )= e^{\mu_3x } M_3(x , \lambda ) $ and   $ x_3\in \R $ such that for any  $\alpha, \beta \ge 0$  there is a constant $C_{\alpha\beta  }$ such that  for all $x\ge x_3$  and      $|\Re \lambda| \le \varepsilon  $
\begin{align*}
 |    \partial_\lambda^\beta \partial_x^{\alpha} \( M_3(x , \lambda )-1\)  |  \leq  C_{\alpha, \beta}  \<\lambda \> ^{-\frac{1}{3}     - \beta  }  e^{-  (p-1)  |x|}
\text{ and }
\lim_{x\to +\infty} M_3(x , \lambda ) =1.
\end{align*}

\end{enumerate}
\end{lemma}
 \begin{proof}
     Here $ M_1(x , \lambda )$   is obtained solving
    \begin{align*}&
          M_1(x , \lambda )  = 1 +    T     M_1(\cdot  , \lambda )     \end{align*}
     where
\begin{align*}
T g (x) &:= - a_1 \mu_1 \int_{-\infty}^x   f' ( \phi   (y))  g(y) d y
\\&
+ a_2  \mu_2 \int_{x} ^{x_1}  e^{ ( \mu_2 -\mu_1)   (  x-y) }  f' ( \phi  (y))  g(y) d y
+
a_3 \mu_3 \int_x^{x_1}   e^{ ( \mu_3 -\mu_1)   (  x-y) } f' ( \phi    (y))  g(y) d y\ .
\end{align*}
  Like in Lemma   \ref{lem:estimate_tildef2}  there exists a $ x_0\in \R $ such that   $ \| T\| _{\mathcal{L}\(L^\infty\( (  -\infty , x_1  )  \)\)}\leq 1/2 $  for $x_1\le x_0$. This gives the estimate for $(\alpha ,\beta )=(0,0)$  and dominated convergence yields  the limit.  Proceeding by induction
       \begin{align*}&
   (1-T)  \partial_x^{\alpha} \partial_\lambda^\beta  M_1(x , \lambda )   =  A_1+A_2  \text{ where } \\ & A_1 = -\sum a_{\beta _1\beta_2}^{\alpha  _1\alpha _2}      \int_{-\infty}^x  \partial _{\lambda } ^{ \beta _1 }   \(  a_1 \mu_1      \)   \partial _{y } ^{ \alpha _1 }  \(  f' ( \phi (y)   \)   \partial_y^{\alpha _2} \partial_\lambda^{\beta _2}   M_1 (y , \lambda )  d y
   \\ & A_2 = -\sum _{j=2}^{3}\sum a_{\beta _1\beta_2}^{\alpha  _1\alpha _2}      \int_{x} ^{x_1}  \partial _{\lambda } ^{ \beta _1 }   \(  a_j \mu_j   e^{ ( \mu_j -\mu_1)   (  x-y) }  \)   \partial _{y } ^{ \alpha _1 }  \(  f' ( \phi  (y)   \)   \partial_y^{\alpha _2} \partial_\lambda^{\beta _2}  M_1 (y , \lambda )  d y
 \end{align*}
    like in Lemma \ref{lem:estimate_f1}. Then    using   \eqref{eq:expdl}   by induction
 \begin{align*}
   | A_j|\lesssim  \< \lambda \> ^{-2/3   -  \beta  }  \int_{-\infty}^x e^{- |y|}   dy   \text{ for }j=1,2.
 \end{align*}
                          The  statement      for $F_3(x , \lambda )$ is obtained similarly. Finally
$ M_2(x , \lambda )$ is obtained solving
    \begin{align*}&
          M_2(x , \lambda )  = 1 +    \widetilde{T}     M_2(\cdot  , \lambda )     \end{align*}
     where
\begin{align*}
\widetilde{T} g (x) &:=   a_ 2 \mu_2 \int^{ +\infty}_x   f' ( \phi  (y))  g(y) d y
\\&
- a_1  \mu_1  \int_{x_2} ^{x }  e^{ ( \mu_1 -\mu_2)   (  x-y) }  f' ( \phi   (y))  g(y) d y
+
a_3 \mu_3 \int^{ +\infty}_x   e^{ ( \mu_3 -\mu_2)   (  x-y) } f' ( \phi   (y))  g(y) d y\ .
\end{align*}
Then again   there exists a $ \widetilde{x}_0\in \R $ such that   $ \| \widetilde{T}\| _{\mathcal{L}\(L^\infty\( (    x_2, +\infty  )  \)\)}\leq 1/2 $  for $x_2\ge \widetilde{x}_0$.  This gives the estimate for $(\alpha ,\beta )=(0,0)$  and dominated convergence yields  the limit. The other estimates are obtained by induction as above.

 \end{proof}

The following   shows that $\tilde{f}_2(x , \lambda )$  does not satisfy good bounds: we will remedy this later.

\begin{lemma}\label{lem:estm2} For any     pair $(\alpha ,\beta) $  and any $\lambda $ with $0\le\Re \lambda \le \varepsilon$ there is a  $C_{\alpha\beta}(\lambda )$ which is locally bounded in $\lambda $ such that  we have
\begin{align} \label{eq:estimate_tildef234}
   |  \partial_x^{\alpha} \partial_\lambda^\beta   \tilde{m}_2(x  , \lambda ) |  \le C_{\alpha, \beta} (\lambda )    \(1 +    e^{\Re (\mu _3-\mu _2) x} \)    \text{  for all $x\in \R$} .
\end{align}

\end{lemma}
\begin{proof}
First of all, there exists some $x_0$ such that
for $x\le x_0$ we have the estimates  in Lemma \ref{lem:estimate_tildef2}. Next we observe that
  $\tilde{f}_2(\cdot  , \lambda )\in \Span \{   {f}_1(\cdot  , \lambda ), {F}_2(\cdot  , \lambda ), F_3(\cdot  , \lambda )        \}$ since the latter three functions are, by their asymptotic behavior at $+\infty$, linearly independent and so they generate the  vector space of  all the solutions of \eqref{eq:lineq1}. Hence
  \begin{align*}
    \tilde{f}_2(x  , \lambda )=A_1(\lambda ) {f}_1(x  , \lambda )+A_2(\lambda ) {F}_2(x  , \lambda ) +A_3(\lambda ) F_3(x  , \lambda )
  \end{align*}
  for the  $ A_j(\lambda )$   three analytic functions in $\lambda $. The estimate \eqref{eq:estimate_tildef234} is then true for some $x_1$ and for  $x\ge x_1$ by Lemmas \ref{lem:estimate_f1} and \ref{lem:extrajost}. Finally for $x_0\le x\le  x_1$ the estimates hold for  \eqref{eq:estimate_tildef234} by smooth dependence  of solutions of \eqref{eq:lineq1} on $\lambda$.

\end{proof}

We set now
 \begin{align}& \label{eq:bxlambda}b_{ij}(x,\lambda ) := f_i (x,\lambda ) f_j' (x,\lambda ) - f_i'(x,\lambda ) f_j (x,\lambda ) \text{  and } \\& \label{eq:bxlambdatil}\widetilde{b}_{ j2}(x,\lambda ) := f_j (x,\lambda ) \widetilde{f}_2' (x,\lambda ) - f_j'(x,\lambda ) \widetilde{f}_2 (x,\lambda ) \text{ for $j=1,3 $}
 \end{align}
 and $\widetilde{b}_{  2j}(x,\lambda ):=-\widetilde{b}_{ j2}(x,\lambda )$.
For $Y= Y(x,\lambda ):= ( {f}_j(x, \lambda ),  {f}_j'(x, \lambda ) , {f}_j'(x, \lambda ))^\intercal $   for $j=1$ or $j=3$     then
\begin{align}
  \frac{d}{dx} Y =AY \text{   for }A = \left(
                                         \begin{array}{ccc}
                                           0 & 1 & 0 \\
                                           0 & 0 & 1 \\
                                          -\lambda -\( f'(\phi )\) '   &  1- f'(\phi  ) & 0 \\
                                         \end{array}
                                       \right) .\label{eq:systA}
\end{align}
Consider $Z=(z_1,z_2,z_3)^\intercal$  a solution of the dual system  $\frac{d}{dx}Z =-A^\intercal Z$. Then   we have   $\mathcal{L} ^*z_3=\lambda z_3$. Notice that this observation implies that
\begin{align}
  \label{eq:soldual}   \( \mathcal{L} ^* -\lambda  \)     b_{ij}(\cdot ,\lambda )  =0 \text{ for all $i,j$.}
\end{align}
Pego and Weinstein  \cite {PegoWei1} define the Evans function $D(\lambda ) =Y^\intercal (x, \lambda)  Z (x, \lambda)$ for solutions $Y$ of  \eqref{eq:systA}   and $Z$  like above chosen so that
\begin{align*}
   Y^\intercal _  + (\lambda) Z_-(\lambda) =1 \text{ where }  Y  _+(\lambda) =\lim _{x\to +\infty}  e^{-\mu_ {1}(\lambda)  x} Y(x, \lambda)  \text{ and } Z  _-(\lambda) =\lim _{x\to -\infty}  e^{ \mu_ {1} (\lambda) x} Z(x,\lambda).
\end{align*}
In Pego and  Weinstein \cite {PegoWei1,PegoWei2} it is  shown that $\lambda \in  {\C}  $  with $\Re \lambda \ge 0$    is an eigenvalue if and only if $D(\lambda )=0$.   We can apply  Pego and  Weinstein \cite [Corollary 1.18]{PegoWei1}
and conclude $D(\lambda ) \xrightarrow{\lambda \to \infty}1$.   From  Pego and  Weinstein \cite  {PegoWei1}, in particular p. 73,  we have
\begin{align}
  \label{eq:derevans0}D(0)=D'(0) =0 \text{ and }D''(0)= \(   \frac {p+1}2 \)^{-\frac 2{p-1}} \< \Lambda_p \phi  , \phi \> .
\end{align}
From Pego and Weinstein \cite[ Formula (1.28)]{PegoWei1}
\begin{align}\label{eq:limm1-infty}
   \lim _{x\to -\infty}  e^{-\mu_ {1}x} Y (x, \lambda ) =D(\lambda )  {Y}_+(\lambda).
\end{align}

\begin{lemma} \label{lem:jostlambda0}
   We have, for the $\Lambda_p$ in \eqref{eq:lambdap},
   \begin{align}
       \label{eq:jostlambda01}&  f_1(x,0 )=  f_3(-x ,0  )=   - \beta ^{-1}(p )   \phi ' (x)  \text{ where  }\beta  (p ):=\(    2 ({p+1})  \)^{\frac 1{p-1}}, \\  \label{eq:jostlambda02}&  \partial _\lambda f_1(x,0 )=       \beta ^{-1}(p ) \(  \frac{1}{p-1} \phi ' +  \Lambda_p  \phi   \) \text{  and } \\  \label{eq:jostlambda02pr}&  \partial _\lambda f_3(x,0 )=     \beta ^{-1}(p ) \(  \frac{1}{p-1} \phi ' -  \Lambda_p  \phi   \) .
   \end{align}
\end{lemma}
\begin{proof}
  We have   $  f_{1}(\cdot , 0 )  \in  \Span \{  \phi ', {F}_3(\cdot  , 0), M_2(\cdot  , 0 )        \}$.  We have  ${f}_1(x , 0 ) =e^{-x}(1+o(1))$   for $x \to +\infty $   by   Lemma \ref{lem:estimate_f1}. On the other hand
  $\phi '(x) =-\phi (x) \tanh \(  \dfrac{p-1}{2} x\) $ and for $x\to +\infty $
  \begin{align*}
   \phi '(x) &=   -\(\frac {p+1}2 \)^{\frac 1{p-1}}
     \(  \cosh\(\frac{p-1}2 x\)  \)  ^{-\frac 2{p-1}}\tanh \(  \dfrac{p-1}{2} x\)\\& \sim -\(\frac {p+1}2 \)^{\frac 1{p-1}}
      \( 2 ^{-1}  \exp \(   \frac{p-1}2 x\)  \)  ^{-\frac 2{p-1}} =-\(    2 ({p+1})  \)^{\frac 1{p-1}} e ^{-x}.
  \end{align*}
   Hence the asymptotic behaviors of $M_2(\cdot  , 0 )$,     Lemma    \ref{lem:estimate_tildef2},  and $F _3(\cdot  , 0)$, Lemma  \ref{lem:extrajost}, and the following,
  \begin{align*}
     e ^{x  }\( f_1(x,0 ) +      \beta ^{-1}(p )\phi ' (x) \) \xrightarrow{x \to +\infty} 0,
  \end{align*}
 yield  the formula for $ f_1(x,0 )$. From \eqref{eq:symmetryf13}
 \begin{align*}
   f_{3}(  x,  0 )=f_{1}(-x, 0 )=  -\beta ^{-1}(p )\phi ' (-x)=        \beta ^{-1}(p )\phi ' (x)
 \end{align*}
  we conclude \eqref{eq:jostlambda01}. Next  we have
    \begin{align*}
       \mathcal{L} \partial _\lambda f_1(\cdot ,0 )=  f_1(\cdot  ,0 ) =  -  \beta ^{-1}(p )  \phi '   .
    \end{align*}
    By \eqref{eq:M2.3} we have $ \mathcal{L}  \Lambda_p  \phi =- \phi ' $ and so
    \begin{align*}
     G(x):= \partial _\lambda f_1(\cdot ,0 ) -    \beta ^{-1}(p )\Lambda_p  \phi   \in  \Span \{  \phi ', {F}_3(\cdot  , 0), M_2(\cdot  , 0 )        \} .
    \end{align*}
     Since $G(x) \xrightarrow{x\to  +\infty}0  $   by Lemma \ref{lem:estimate_f1},  necessarily $G(x) =  \alpha    \beta ^{-1}(p )  \phi ' $ for some $\alpha \in \R$ and so
    \begin{align*}
     \partial _\lambda f_1(\cdot ,0 )&= \alpha   \beta ^{-1}(p )  \phi ' +   \beta ^{-1}(p )\Lambda_p  \phi    .
    \end{align*}
 We
     obtain $\alpha$     matching    asymptotic behaviors  in the last formula. By Lemma \ref{lem:estimate_f1}, for $x\to +\infty$
     \begin{align}\label{eq:m2prime0}
        \partial _\lambda f_1(x ,0 )&= e^{-x }\(  \partial _\lambda \mu _1 (0)  x+o(1)     \) =  e^{-x }\(   -\frac{1}{2 }  x+o(1)     \) \text{ by }\partial _\lambda \mu _1 (0)=-\frac{1}{2 },
     \end{align}
     with the last equality  obtained by   elementary implicit differentiation in \eqref{eq:mus}.
    For $x\to +\infty$
     \begin{align*}&
       \alpha    \beta ^{-1}(p )  \phi ' (x) +    \beta ^{-1}(p )\Lambda_p  \phi ( x) =   e^{-x}   \left (  -\frac{1}{2} x    -    \alpha + \frac{1}{p-1} +o(1)   \right )
     \end{align*}
     which   to match with \eqref{eq:m2prime0} requires $\alpha =  \frac{1}{p-1}$ giving   \eqref{eq:jostlambda02}.
Using \eqref{eq:symmetryf13}  and  \eqref{eq:jostlambda02} we obtain
\begin{align*}
  \partial _\lambda f_3(x ,0 ) &= - \partial _\lambda f_1(-x ,0 )= -  \beta ^{-1}(p ) \(  \frac{1}{p-1} \phi ' (-x) +  \Lambda_p  \phi (-x)  \) \\&=    \beta ^{-1}(p ) \(  \frac{1}{p-1} \phi ' ( x) -  \Lambda_p  \phi ( x)  \) .
\end{align*}

\end{proof}

\begin{corollary}\label{lem:b13at0}   The following facts are true.
\begin{description}
  \item[i]  We have $\partial ^{a}_\lambda b _{13}(x,0)=0$ for $a=0,1 $.
  \item[ii] For any      $n, m \ge 0$  there is a constant $C_{n,m}$ such that   for all   $\lambda \in \C $ such that $ 0\le  \Re \lambda  \le \varepsilon  $  and $x\in \R$
    \begin{equation}\label{eq:b13_estimate_partial}
        \left |     \partial_x^{n}    \partial_\lambda^m \(    e^{ \mu _2 (\lambda )x }   b _{13}(x , \lambda )   \)  \right |  \leq C_{n,m}  \<\lambda \> ^{ \frac{1}{3}   -m  }    .
    \end{equation}
  \item[iii]   There exists a  constant $b_0\in \C$  such that  $\partial ^{2}_\lambda b _{13}(x,0)=2D''(0) +b_0 \phi (x)$
\end{description}

\end{corollary}
\begin{proof} Starting from \textbf{i},
  case $a=0$ is obvious from \eqref{eq:jostlambda01} while   \small
\begin{align*}
  &\partial  _\lambda b _{13}(x,0)  =\partial _\lambda f_1 (x,0 ) f_3' (x,0 )  +  f_1 (x,0 ) \partial _\lambda f_3' (x,0 )  - \partial _\lambda f_1'(x,0 ) f_3 (x,0 ) -  f_1'(x,0 ) \partial _\lambda f_3 (x,0 )=\\& \frac{1}{\beta  ^2(p )}
    \left [          \(  \cancel{\frac{1}{p-1} \phi ' }+  \bcancel{\Lambda_p  \phi }  \) \phi '' -\phi '  \(  \cancel{\frac{1}{p-1} \phi '}  - \xcancel{ \Lambda_p  \phi }  \) ' - \(  \cancel{\frac{1}{p-1} \phi '} +  \xcancel{ \Lambda_p  \phi }   \) '  \phi '+ \phi ''    \(  \cancel{\frac{1}{p-1} \phi '}  -   \bcancel{\Lambda_p  \phi }   \) \right ] \\& =0.
\end{align*}
\normalsize
By elementary computation and using \eqref{eq:sum_of_mus} we obtain
\begin{align*}
 &  b _{13}(x , \lambda ) =  e^{\mu _1 (\lambda )x }    m_1 (x,\lambda )  \(   e^{\mu _3 (\lambda )x } m_3  (x,\lambda )\) '  -  \(   e^{\mu _1 (\lambda )x } m_1  (x,\lambda )\) '    e^{\mu _3 (\lambda )x } m_3  (x,\lambda ) \\& =  e^{-\mu _2 (\lambda )x }   \( \(  \mu _3 (\lambda )-\mu _1 (\lambda )    \)   m_1 (x,\lambda )  m_3 (x,\lambda )  +    m_1 (x,\lambda )  m_3 ' (x,\lambda ) - m_1 '(x,\lambda )  m_3 (x,\lambda )       \) .
\end{align*}
Hence from  {Lemmas} \ref{lem:characteristic-roots}--\ref{lem:estimate_f3}, we have
the following, proving \textbf{ii},  
\small
\begin{align*}&
  \left |      \partial_x^{n}    \partial_\lambda^m \(    e^{ \mu _2 (\lambda )x }   b _{13}(x , \lambda )   \)    \right |  =\\&  \left |     \partial_x^{n}    \partial_\lambda^m  \(  \(  \mu _3 (\lambda )-\mu _1 (\lambda )    \)   m_1 (x,\lambda )  m_3 (x,\lambda )  +    m_1 (x,\lambda )  m_3 ' (x,\lambda ) - m_1 '(x,\lambda )  m_3 (x,\lambda )   \)  \right |  \leq C_{n,m}  \<\lambda \> ^{ \frac{1}{3}   -m  }    .
\end{align*}\normalsize
 Differentiating  \eqref{eq:soldual} and using Claim \textbf{i},  it is easy to conclude that  $   \mathcal{L} ^*       \partial _\lambda ^2 b_{13}(\cdot ,0)  =0$.
Then  \begin{align*}
  \partial _\lambda ^2    {b}_{13}   (\cdot ,0 ) \in \Span \{ 1, \phi , g     \}
\end{align*}
where $g$ is an unbounded solution of $ \mathcal{L}^*u=0$, but since, by Claim \textbf{i} and \eqref{eq:b13_estimate_partial},    $ \partial _\lambda ^2 {b}_{13}   (\cdot ,0 ) \in L^\infty (\R )$
there are two constants such that
 $ \partial _\lambda ^2    {b}_{13}   (\cdot ,0 )  =a_0 +b_0\phi  $. Here we show the explicit value $a_0=  2D''(0)\neq 0$, as this is the source of some difficulty later. Obviously we have
 $\displaystyle a_0=\lim _{x\to -\infty}\partial _\lambda ^2    {b}_{13}   (x ,0 ) $. We have
\begin{align*}
   \partial _\lambda ^2    {b}_{13}   (x ,0 )  &= 2 a(x) +b(x) \text{ with } a(x) = \partial _\lambda ^2 m_1 (x,0)  m_3 (x,0 )  \text{ and }
\\  b(x)&= 4\partial _\lambda   m_1 (x,0)  \partial _\lambda m_3 (x,0 )  +  2 \ m_1 (x,0)  \partial _\lambda ^2 m_3 (x,0 )   \\& + \partial _\lambda ^2  \(  m_1 (x,\lambda )  m_3 ' (x,\lambda ) - m_1 '(x,\lambda )  m_3 (x,\lambda )   \)|  _{\lambda =0}
\end{align*}
where we used   \textbf{i} and  $\mu _3 (0 )-\mu _1 (0 )=2$. Now   $m_3 (x,0 ) \xrightarrow{x\to -\infty}1$            by \eqref{eq:m3_estimate_partial}. We have
\begin{align*}
    m_1(x , \lambda ) -1 &=     \sum_{ {j}= 2,3 } a_ {j}  \mu _ {j}     \int_ {x}  ^{+\infty }e^{ (\mu_{ {j}}   -\mu_{1} ) (x-y)}  f' ( \phi   (y)   ) m_1 (y , \lambda )  d y \\& -  a_ {1}  \mu _ {1}   \int_ {x}  ^{+\infty }   f' ( \phi   (y)   ) m_1 (y , \lambda )  d y =: A_1(x,\lambda ) +  A_2(x,\lambda ).
\end{align*}
It is easy to see that always $\partial _\lambda ^m  \partial _\lambda ^n A_1(x,\lambda )  \xrightarrow{x\to -\infty}0$ and  by \eqref{eq:limm1-infty}
\begin{align*}
 \lim _{x\to -\infty}\partial _\lambda ^n A_2(x,\lambda )  = -\partial _\lambda ^n \(    a_ {1}  \mu _ {1}   \int_ {\R }   f' ( \phi   (y)   ) m_1 (y , \lambda )  d y  \) = \partial _\lambda ^n \( D(\lambda ) -1\)
\end{align*}
So we have proved $a (x  ) \xrightarrow{x\to -\infty} 2D''(0)$. Finally we prove  $b (x  ) \xrightarrow{x\to -\infty} 0$. For most of the terms in $b(x)$  this follows from \eqref{eq:m3_estimate_partial}.
The only possibly critical term is  the following, which gives the proof,
\begin{align*}
  \partial _\lambda ^2 \partial _x m_1  (x,\lambda ) = \partial _\lambda ^2 \partial _x  A_1(x,\lambda ) +  \partial _\lambda ^2 \(  a_ {1}  \mu _ {1}     f' ( \phi   (x)   ) m_1 (x , \lambda ) \)  \xrightarrow{x\to -\infty} 0.
\end{align*}

\end{proof}

\begin{lemma} \label{lem:dual0}
   For $j=1$ and $3$ and using  the notation in  \eqref{eq:bxlambdatil},    we have
   \begin{align}
       \label{eq:dual01}&   \widetilde{b}_{ j2}(x,0 )  =   (-1) ^{\frac{j-1}{2}}   \beta ^{-1}(p ) \<\Lambda_p  \phi   ,\phi   \>    \eta _2   \text{ and }\\ \label{eq:dual02}&   \partial _\lambda    \widetilde{b}_{ j2}(x,0  ) =-  (-1) ^{\frac{j-1}{2}}  \beta ^{-1}(p ) \<\Lambda_p  \phi,\phi  \>    \eta _1 +\alpha _j  \eta _2  \\&\label{eq:dual02b}  \text{with }  \alpha _1=-\alpha_3-\frac{2}{(p-1) \beta  (p ) \theta _3}    \text{  and }\alpha _3= \frac{   1+\frac{1}{p-1} +\frac{\theta _ 2}{\theta _ 1}}{ \beta  (p ) \theta _3} .
   \end{align}
\end{lemma}
 \begin{proof}  The space of the  solutions of equation $ \mathcal{L} ^* u=0$  coincides with
$ \Span \{ 1, \eta _2 , g     \}$ for a function $g$ that explodes exponentially as $x\to \pm \infty$.
By \eqref{eq:soldual} we have $\mathcal{L} ^* \widetilde{b}_{ 32}(\cdot ,\lambda )=-\lambda \widetilde{b}_{ 32}(\cdot ,\lambda )$. By $\widetilde{b}_{ 32}(x,0 ) \xrightarrow{x\to -\infty} 0$
    we can conclude that there exists an $ \alpha \in \R $ such that
$\widetilde{b}_{ 32}(\cdot ,0 )  = \alpha   \beta ^{-1}(p )   \phi   $. We find $\alpha$ by  matching the   asymptotic behaviors: for $x\to -\infty$ we have
\begin{align*}
  & \widetilde{b}_{ 32}(x ,0 )=  - f_3'(x,0 )  \( 1+o(1) \) = -  e^{x }  \( 1+o(1) \) \text{  and }\\& \alpha   \beta ^{-1}(p )   \phi  (x) = \alpha  e^{x }  \( 1+o(1) \)
\end{align*}
  obtaining $\alpha =- 1$. Formula \eqref{eq:dual01}  for $j=3$ follows from \eqref{eq:ker2}.  Case $j=1$  follows from $f_3 (x,0 )=-f_1 (x,0 ) $.  Differentiating in $\lambda $ the equation of $\widetilde{b}_{ j2}(\cdot ,\lambda )$ and using \eqref{eq:dual01} we obtain
\begin{align*}&
    \mathcal{L} ^* \partial _\lambda \widetilde{b}_{ j2}(x ,0 )=- \widetilde{b}_{ j 2}(x ,0 ) =   - (-1) ^{\frac{j-1}{2}}  \beta ^{-1}(p ) \<\Lambda_p  \phi ,\phi   \>    \eta _2 (x ).
\end{align*}
By \eqref{eq:M2.4} we obtain
\begin{align}\label{eq:predual021}
   \mathcal{L} ^* \( \partial _\lambda \widetilde{b}_{ j2}(\cdot  ,0 ) + (-1) ^{\frac{j-1}{2}}    \beta ^{-1}(p ) \<\Lambda_p  \phi ,\phi   \>    \eta _1 \) =0.
\end{align}
Since, by  \eqref{eq:ker2},    \eqref{eq:estimate_tildef233} and \eqref{eq:bxlambdatil} and by Lemma \ref{lem:jostlambda0} the      function   in \eqref {eq:predual021}   decays to 0 for $x\to -\infty$,  there exists an $\alpha  _j\in \R$ such that
\begin{align}\label{eq:dual021}
 \partial _\lambda \widetilde{b} _{ j2}(\cdot  ,0 ) +  (-1) ^{\frac{j-1}{2}} \beta ^{-1}(p ) \<\Lambda_p  \phi ,\phi \>    \eta _1 =\alpha _j  \eta _2 .
\end{align}
 We find $\alpha _j$ by matching the asymptotic behaviors for $x\to -\infty$. By \eqref{eq:ker2} and \eqref{eq:lambdap}
 \begin{align}\nonumber
   \eta _1 (x) &= \frac{\theta _1}{2 } \int _{-\infty}^{x}y\phi ' (y)dy + \theta _1 \frac{1}{p-1}  \int _{-\infty}^{x} \phi  (y)dy + \theta _ 2 \phi   (x)\\& =  \frac{\theta _1}{2 } x \phi  (x) +\left [  \theta _1  \( \frac{1}{p-1}        - \frac{1}{2}\) + \theta _ 2 \right ]  \phi  (x) +o \(  \phi (x)  \)  \text{ for }x\to -\infty. \label{eq:aseta1-infty}
 \end{align}
By \eqref{eq:sum_of_mus}, \eqref{eq:m3_estimate_partial}  and \eqref {eq:estimate_tildef233} we have
\begin{align*}
   \widetilde{b}_{ 32}(x  ,\lambda ) = e^{-\mu _1 ( \lambda )x} \( \mu _2( \lambda ) - \mu _3( \lambda )\) +o\(     \frac{e^{-\mu _1( \lambda ) x}}{x} \)   \text{ for } x\to -\infty
\end{align*}
so that
\begin{align*}
   \partial _\lambda \widetilde{b}_{ 32}(x  ,0 )&= e^{ x}   \(   \frac{3}{2} -\frac{x}{2}   +o(1) \)  \text{ for } x\to -\infty
 \end{align*}
 Matching the asymptotic terms in \eqref{eq:dual021}, after a cancellation of the two dominating terms in the left hand side, we obtain  the following, \small
 \begin{align*} \alpha  _3\beta  (p ) \theta _3&=  \frac{3}{2 }  +    \beta ^{-1}(p ) \frac{1}{\theta _1} \left [  \theta _1  \( \frac{1}{p-1}        - \frac{1}{2}\) + \theta _ 2 \right ] \beta  (p )    =1+\frac{1}{p-1} +\frac{\theta _ 2}{\theta _ 1}.
 \end{align*}\normalsize
By \eqref{eq:estimate_tildef233} and Lemma \ref{lem:jostlambda0}     we have, using $\mu _2'(0)=1$,
\begin{align*}
   \partial _\lambda \widetilde{b}_{ 12}(x  ,0 )&=  \partial _\lambda f_1 (x,0 ) \widetilde{f}_2' (x,0 ) +   f_1 (x,0 ) \partial _\lambda \widetilde{f}_2' (x,0 ) - \partial _\lambda f_1' (x,0 ) \widetilde{f}_2  (x,0 ) -   f_1' (x,0 ) \partial _\lambda \widetilde{f}_2 (x,0 ) \\& =o(e^x)-e^{x}\mu _2'(0)  - \( \frac{2}{p-1}e^x   +\frac{x}{2}e^x \) ' +e^x \mu _2'(0) x\\ &=      \frac{1}{2 }  x  e^{x }-\(  \frac 32+\frac{2}{p-1}  \)     e^{x } +o(e^{x }) \text{ for $x\to -\infty$}.
 \end{align*}
By matching  the asymptotic behaviors for $x\to -\infty$ in \eqref{eq:dual021} we obtain \small
 \begin{align*} \alpha  _1\beta  (p ) \theta _3&= -\(  \frac 32+\frac{2}{p-1}  \)     - \frac{\theta_2}{\theta _1}- \(     \frac{1}{p-1} -\frac{1}{2}  \)   = -1-\frac{3}{p-1} -  \frac{\theta_2}{\theta _1} .
 \end{align*}
\normalsize

 \end{proof}

 In general $\tilde{f}_2$ increases exponentially as $x\to \infty$ even if $\Re \lambda=0$.
 In the following, we modify $\tilde{f}_2$ to be globally bounded.
We first examine the solutions of $\mathcal{L}u=0$.
 \begin{lemma}\label{lem:structLkernel}
     The space of solutions of $\mathcal{L}u=0$ is spanned by $\{\phi', h_2,h_3\}$, where $h_2,h_3$ are even functions satisfying $\lim_{x\to \infty }h_2(x)=1$ and $\lim_{x\to \infty}h_3(x)=\infty$.
     In particular, there exists no solution of $\mathcal{L}u=0$ satisfying $\lim_{x\to -\infty} u(x)=1$ and $\lim_{x\to +\infty} |u(x)|=\infty$.
 \end{lemma}

 \begin{proof}
     Since $\mathcal{L}u=0$ is an 3rd order ordinal differential equation, the solutions are determined by $u(0)$, $u'(0)$ and $u''(0)$.
     Since $\phi' $ solves $\mathcal{L}u=0$ and $\phi'(0)=\phi'''(0)=0$ since $\phi$ is an even function, it suffices to consider the remaining 2 dimensional space of solutions with $u'(0)=0$.
     Now, since $\mathcal{L}u=0$ is equivalent to $L_+ u=c$ for some $c\in \R$ and $L_+$ preserves the parity, we see that solutions of $\mathcal{L}u=0$ with $u'(0)=0$ are even.
     From Lemma \ref{lem:extrajost}, there exist solutions of $\mathcal{L}u=0$ satisfying $\lim_{x\to +\infty}u(x)=1$ and solutions satisfying $\lim_{x\to +\infty}u(x)=\infty$ (recall $\mu_3(0)=1$), yielding the desired $h_2$ and $h_3$.
     Finally, if $\lim_{x\to +\infty}u(x)=\infty$, then $u$ has to have $h_3$ component.
     However, this will imply $\lim_{x\to -\infty}u(x)=\infty$.
 \end{proof}

 Now, since $\widetilde{f}_2(x,\lambda), f_3(x,\lambda)\in \mathrm{Span}\{f_1(x,\lambda),F_2(x,\lambda),F_3(x,\lambda)\}$, we can write
\begin{align}
     \widetilde{f}_2(x,\lambda)&=c_{21}(\lambda)f_1(x,\lambda)+c_{22}(\lambda)F_2(x,\lambda)+c_{23}(\lambda)F_3(x,\lambda),\label{eq:connect2}\\
    f_3(x,\lambda)&=c_{31}(\lambda)f_1(x,\lambda)+c_{32}(\lambda)F_2(x,\lambda)+c_{33}(\lambda)F_3(x,\lambda).\label{eq:connect3}
\end{align}

\begin{lemma}\label{lem:c23vanish} The
$c_{jk}(\lambda)$ ($j=2,3$, $k=1,2,3$)   in \eqref{eq:connect2}--\eqref{eq:connect3} are analytic in $\lambda$.
Furthermore, we have $c_{33}(\lambda)=D(-\lambda)$ and $c_{23}(0)=0$.
\end{lemma}

\begin{proof}
    Set $c_{1j}$ by $F_1(x,\lambda)=c_{11}(\lambda)f_1(x,\lambda)+c_{12}(\lambda)F_2(x,\lambda)+c_{13}F_3(x,\lambda)$.
    Then, the $3\times 3$ matrix $C=(c_{jk})_{1\leq j,k\leq 3}$ becomes invertible and since it can be determined by the values of $F_j(0,\lambda)$ ($j=1,2,3$), $f_1(0,\lambda)$, $\widetilde{f}_2(0,\lambda)$ and $f_3(0,\lambda)$ which are all analytic, we have the analyticity of $c_{jk}$.

    Next, by the symmetry in \eqref{eq:symmetryf13} we have $m_3(x,\lambda ) =m_1(-x,-\lambda )$ for
$  \lambda \in \im \R $ so that by the   limit \eqref{eq:limm1-infty} we have
\begin{align}\label{eq:m3+infty}
  c_{33}(\lambda)=\lim  _{x\to +\infty}  m_3(x,\lambda ) &=    \lim  _{x\to -\infty}  m_1(x,-\lambda ) =D(-\lambda )  .
\end{align}
Finally, suppose $c_{23}(0)\neq 0$.
Then, $\widetilde{f}_2(x,0)=\widetilde{m}_2(x,0)$ will satisfy $\lim_{x\to -\infty}\widetilde{f}_2(x,0)=1$ and $\lim_{x\to -\infty} |\widetilde{f}_2(x,0)|=\infty$.
Since $\widetilde{f}_2(x,0)$ solves $\mathcal{L}u=0$ this contradicts with Lemma \ref{lem:structLkernel}.
Therefore, we have $c_{23}(0)=0$.
\end{proof}

Now, we set
\begin{align}\label{eq:defc0}
    c_0(\lambda):=\frac{\lambda c_{23}(\lambda)}{D(-\lambda)}.
\end{align}
Notice that from Lemma \ref{lem:c23vanish} and \eqref{eq:derevans0}, $c_0$ is analytic in the neighborhood of $0$.
Set
\begin{align}
  \label{eq:deff2}
  f_{2}(x,\lambda ) := \lambda  \widetilde{f}_{2}(x,\lambda ) -c_0(\lambda){f}_{3}(x,\lambda ).
\end{align}
Then, from the choice of $c_0(\lambda)$, we have $f_2(x,\lambda)\in \mathrm{Span}\{f_1(x,\lambda),F_2(x,\lambda)\}$.
Indeed, from \eqref{eq:connect2}, \eqref{eq:connect3}, \eqref{eq:defc0} and Lemma \ref{lem:c23vanish},
\begin{align*}
    f_2(x,\lambda)=&\left(\lambda c_{21}(\lambda)-c_0(\lambda)c_{31}(\lambda)\right) f_1(x,\lambda)+\left(\lambda c_{22}(\lambda)-c_0(\lambda)c_{32}(\lambda)\right) F_2(x,\lambda)\\&
    +\cancel{\left(\lambda c_{23}(\lambda)-c_0(\lambda)D(-\lambda)\right)} F_3(x,\lambda).
\end{align*}
The new Jost function $f_2$ satisfies better estimates than $ \tilde{f}_2(x , \lambda ) $.

\begin{lemma}
  \label{lem:bdm2}  For any preassigned $a_0>0$  there are constants     $C_{\alpha\beta}>0$ such that \begin{align} \label{eq:estm21}
   |  \partial_x^\alpha  \partial_\lambda^\beta   m_2(x  , \lambda ) |  \le C _{\alpha\beta}    \text{ for all $x\in \R$ and all $\lambda \in \(  -\im a_0, \im a_0\) $}.
\end{align}
\end{lemma}
\begin{proof}
  We know from Lemma \ref{lem:estimate_tildef2} that \eqref{eq:estm21} holds for $x\le 0$.  For $x\ge 0$ it follows from
 $    {f}_2(\cdot  , \lambda )  \in \Span \{    {f}_1(\cdot  , \lambda  ), {F}_2(\cdot  , \lambda  )        \}$ and  from Lemmas \ref{lem:estimate_f1} and  \ref{lem:extrajost}.

\end{proof}

Consider now the Wronskian, see  \eqref{eq:wronskdef},   \begin{align}&    W(\lambda ):=  W [{f}_1(y, \lambda ),{f}_2(y  , \lambda ), {f}_3(y  , \lambda )]  .\label{eq:wronsk}
 \end{align}
 We claim that
 \begin{align}\label{eq:wronsk1}
 W(\lambda )  = \lambda D(\lambda )  W_0(\lambda ).
\end{align}
 To prove \eqref{eq:wronsk1} notice that
 \begin{align*}
   W(\lambda )=\lambda W [{f}_1(y, \lambda ),\widetilde{{f}}_2(y  , \lambda ), {f}_3(y  , \lambda )].
 \end{align*}     For
 \begin{align*}
  & Y (x, \lambda )  :=( {f}_1(x, \lambda ) , {f}_1'(x, \lambda ) , {f}_1'(x, \lambda ))^\intercal \text{   and } Z(x,\lambda ):=
  ( z_1  , z_2  , z_3 )^\intercal \text{ with }
  \\&  z_1 = \widetilde{f}_2' (x,\lambda ) f_3''(x,\lambda ) -\widetilde{f}_2'' (x,\lambda ) f_3' (x,\lambda ), \\&
   z_2 =  \widetilde{f}_2'' (x,\lambda ) f_3 (x,\lambda ) -\widetilde{f}_2 (x,\lambda ) f_3''(x,\lambda )  \text{ and} \\& z_3 =   \widetilde{f}_2 (x ,\lambda ) f_3'(x,\lambda )  -   \widetilde{f}_2' (x,\lambda ) f_3 (x,\lambda ),
 \end{align*}
 then $Y$ solves \eqref{eq:systA} and $Z$ solves the dual system $\frac{d}{dx}Z =-A^\intercal Z$   and we have
 \begin{align}\label{eq:connwron}
    W [{f}_1(x, \lambda ),\widetilde{{f}}_2(x   , \lambda ), {f}_3(x  , \lambda )] =  Y^\intercal (x, \lambda)  Z (x, \lambda).
 \end{align}
 By Lemmas    \ref{lem:estimate_f1},     \ref{lem:estimate_f3}   and \ref{lem:estimate_tildef2} we have
 \begin{align*}&
   e^{-\mu_ {1}x} Y (x, \lambda ) \xrightarrow{x\to +\infty}  {Y}_+(\lambda)  =( 1 , \mu_ {1} , \mu_ {1} ^2)^\intercal \text{ and}\\&  e^{ \mu_ {1}x} Z (x, \lambda ) \xrightarrow{x\to -\infty} {Z}_-(\lambda) =( \mu_ {2}\mu_ {3} (\mu_ {3} -\mu_ {2}), \mu_ {2} ^2-\mu_ {3} ^2 ,\mu_ {3} -\mu_ {2} )^\intercal
 \end{align*}
 with, by an elementary computation which uses \eqref{eq:sum_of_mus}  and  \eqref{eq:vandermonde},
 \begin{align*}
    Y^\intercal _+ (\lambda) Z_-(\lambda)=  W _0\(\lambda\) .
 \end{align*}
 Then we obtain the following which, with    \eqref{eq:connwron}, yields  \eqref{eq:wronsk1},
 \begin{align*}
     Y^\intercal (x, \lambda)  Z (x, \lambda) =D(\lambda ) W _0\(\lambda\).
 \end{align*}

\begin{lemma}\label{lem:varpar}
   For $R(\lambda ) = R_{\mathcal{L} }(\lambda )$  for $\varepsilon \ge \Re \lambda \ge  0$   with  $\lambda \neq 0$  we have
    \small
 \begin{align} \label{eq:res0}
   R(\lambda ) g     & =   R_1(\lambda ) g + R_2(\lambda ) g  \text{  where  } \\ R_1(\lambda ) g& (x)= -
     {f}_2(x , \lambda )   \int_  x ^{+\infty}   \frac{g(y)  b_{13}(y,\lambda )}{W(\lambda )} dy  \text{   and} \label{eq:res1}
    \\ R_2(\lambda ) g&  =R _{23}(\lambda ) g+ R _{21}(\lambda ) g \text{ with}\label{eq:res2} \\
   R _{23}(\lambda ) g&(x)=
    - {f}_1(x , \lambda )  \int_  {-\infty} ^x  \frac{g(y)  b_{23}(y,\lambda )}{W(\lambda )} dy   \text{ , }   R _{21}(\lambda ) g (x) ={f}_3(x , \lambda )   \int_  x ^{+\infty} \frac{g(y)  b_{12}(y,\lambda )}{W(\lambda )} dy, \nonumber
\end{align} \normalsize
   \end{lemma}
 \begin{proof}
   First of all, by the  variation of parameters formula      the right hand side in   \eqref{eq:res0}, which is well defined in the range    $\varepsilon \ge \Re \lambda \ge  0$ for some $\varepsilon >0$,  
    solves the equation $(\mathcal{L}-\lambda )u=g $.  
     Furthermore,  by \eqref{eq:stres0} and by  Lemmas \ref{lem:estimate_f1}, \ref{lem:estimate_f3}  and \ref{lem:bdm2}  it defines a bounded operator in $L^2(\R )$ when $\varepsilon \ge\Re \lambda >0$.  By continuity in $\lambda = z$ in   \eqref{eq:forresol}, it holds also  for  $\Re \lambda \ge  0$   with  $\lambda \neq 0$.

 \end{proof}
 Notice that   \eqref{eq:res0}--\eqref{eq:res2} are equivalent to    \small
\begin{align} \label{eq:tilres0}
   R(\lambda ) g     & =   \widetilde{R}_1(\lambda ) g + \widetilde{R}_2(\lambda ) g  \text{  where  } \\ \widetilde{R}_1(\lambda ) g& (x)= -
     \widetilde{{f}}_2(x , \lambda )   \int_  x ^{+\infty}   \frac{g(y)  b_{13}(y,\lambda )}{D(\lambda ) W _0\(\lambda\)} dy  \text{   and} \label{eq:tilres1}
    \\ \widetilde{R}_2(\lambda ) g&  =\widetilde{R} _{23}(\lambda ) g+ \widetilde{R} _{21}(\lambda ) g \text{ with}\label{eq:tilres2} \\
   \widetilde{R} _{23}(\lambda ) g&(x)=
    - {f}_1(x , \lambda )  \int_  {-\infty} ^x  \frac{g(y)  \widetilde{b}_{23}(y,\lambda )}{D(\lambda ) W _0\(\lambda\)} dy    \text{ and  }       \widetilde{R} _{21}(\lambda ) g  (x) ={f}_3(x , \lambda )   \int_  x ^{+\infty} \frac{g(y)  \widetilde{b}_{12}(y,\lambda )}{D(\lambda ) W _0\(\lambda\)} dy. \nonumber
\end{align} \normalsize


\section{Proof of Proposition \ref{prop:smooth2} }\label{sec:proofsmooth2}

We start by first proving the inequality \eqref{eq:smooth21}.
We  distinguish between $\lambda$ close to 0 and away from 0.  We start with the  small $\lambda$  case.
\begin{lemma}\label{lem:smooth2small} There exists an $\varepsilon _4>0$  and a  $C(  p)>0$ such that
  \begin{align}\label{eq:smooth21small}
      \left \|       R (\im  \tau )      g     \right \|   _{   \widetilde{\Sigma}    } \le C(p)          \left \|        g     \right \| _{   L ^{1,1}_+\( \R\)    }  \text{ for all $ g\in \ker P  $ and } \tau  \in  (-\varepsilon _4, \varepsilon _4  )\backslash \{ 0  \}.
   \end{align}

\end{lemma}

\begin{proof} For $\lambda =\im \tau$, using \eqref{eq:res0} and \eqref{eq:tilres0} we set
\begin{align*}
  R  (\lambda ) g=\frac{ N  (\lambda )g }{W(\lambda )} = \frac{ \widetilde{N}  (\lambda )g }{D(\lambda ) W _0\(\lambda\)} ,
\end{align*}
where $ N  (\lambda )= \lambda  \widetilde{N}  (\lambda ) $     can be defined using \eqref{eq:res1}--\eqref {eq:res2} resp. \eqref{eq:tilres1}--\eqref {eq:tilres2}. Next we consider
\begin{align*}
  R^{(0)}(\lambda ) g     & :=  \frac{1}{D(\lambda ) W _0\(\lambda\)} \sum _{i=0,1}    \lambda ^i   \partial _z ^i  \widetilde{N}( 0) g    \\& =   \frac{1}{D(\lambda ) W _0\(\lambda\)}
  \(  - {f}_1(x , 0 )  \int_  {-\infty} ^x   g(y)  \widetilde{b}_{23}(y,0 )   dy  + {f}_3(x , 0 )   \int_  x ^{+\infty}  g(y)  \widetilde{b}_{12}(y,0 )  dy \) \\& +  \frac{\lambda}{D(\lambda ) W _0\(\lambda\)}
  \(  - \partial _\lambda {f}_1(x , 0 )  \int_  {-\infty} ^x   g(y)  \widetilde{b}_{23}(y,0 )   dy  + \partial _\lambda{f}_3(x , 0 )   \int_  x ^{+\infty}  g(y)  \widetilde{b}_{12}(y,0 )  dy \) \\& +    \frac{\lambda}{D(\lambda ) W _0\(\lambda\)}
  \(  -  {f}_1(x , 0 )  \int_  {-\infty} ^x   g(y)  \partial _\lambda\widetilde{b}_{23}(y,0 )   dy  +  {f}_3(x , 0 )   \int_  x ^{+\infty}  g(y)  \partial _\lambda\widetilde{b}_{12}(y,0 )  dy \)  \\& =:  \frac{1}{D(\lambda ) W _0\(\lambda\)}   A_0+\frac{\lambda}{D(\lambda ) W _0\(\lambda\)} A_{11}+\frac{\lambda}{D(\lambda ) W _0\(\lambda\)} A_{12}
\end{align*}
 with the contribution  of $\widetilde{R}_1(\lambda )$   null by Corollary \ref{lem:b13at0}.
  We claim that, since $g \in \ker P$, we have $ R^{(0)}(\lambda ) g =0$.
  From  \eqref{eq:jostlambda01} and \eqref{eq:dual01} we have
\begin{align*}
   A_0=   \beta ^{-1}(p ){f}_3(x , 0 ) \<\Lambda_p  \phi   ,\phi   \>  \< g ,   \eta _2 \> =0 \text{ by }g \in \ker P.
\end{align*}
We have  \small
\begin{align*}
   A_{11}&=  -    \frac{\beta ^{-2}(p )}{\theta_3} \(  \frac{1}{p-1} \phi ' +  \Lambda_p  \phi   \)  \int_  {-\infty} ^x   g(y)   \eta _2(y)   dy   +   \frac{\beta ^{-2}(p )}{\theta_3} \(  \frac{1}{p-1} \phi ' -  \Lambda_p  \phi   \)    \int_  x ^{+\infty}  g(y)       \eta _2(y)  dy \\& =\cancel{ \frac{\beta ^{-2}(p )}{\theta_3}  \Lambda_p  \phi \< g ,   \eta _2 \>} -    \frac{\beta ^{-2}(p )}{\theta_3} \frac{\phi '}{p-1}    \int_  {-\infty} ^x   g(y)   \eta _2(y)   dy  +   \frac{\beta ^{-2}(p )}{\theta_3}   \frac{\phi '}{p-1}  \int_  x ^{+\infty}  g(y)       \eta _2(y)  dy ,
\end{align*}\normalsize
where the cancelled term is null by $g \in \ker P.$   Similar cancellations yield \small
\begin{align*}
   A_{12}&=   \frac{\beta ^{-2}(p )}{\theta_3}   \phi '  \<  g,   \eta _1  \> - \beta ^{-1}(p ) \alpha_3 \phi '  \int_  {-\infty} ^x   g(y)      \eta _2(y) dy  \\&   -    \(\alpha_3+\frac{2}{(p-1) \beta  (p ) \theta _3} \)  \beta ^{-1}(p )  \phi '   \int_  x ^{+\infty}  g(y)    \eta _2(y) dy =-\frac{2}{(p-1) \beta  ^2 (p ) \theta _3}  \phi '   \int_  x ^{+\infty}  g(y)    \eta _2(y) dy.
\end{align*} \normalsize
Then we obtain the following, which proves  $ R^{(0)}(\lambda ) g =0$ for $g \in \ker P $,
\begin{align*}
   A_{11}+ A_{12}= - \frac{\beta ^{-2}(p )}{\theta_3} \frac{\phi '}{p-1}  \< g ,   \eta _2 \> =0.
\end{align*}
Since
\begin{align*}
   R^{(0)}(\lambda ) g =  \frac{1 }{\lambda D(\lambda ) W _0\(\lambda\)} \sum _{i=0}^{2}   \frac{\lambda ^i}{i!}   \partial _z ^i  {N}( 0) g=0
\end{align*}
we have for $ R ^{(1)} (\lambda ) g=  R  (\lambda ) g  -R ^{(0)} (\lambda ) g$,
by the Taylor expansion formula we have
\begin{align}
  R ^{(1)} (\lambda ) g & = R _{ 1}^{(1)}(\lambda  ) g+R _{21}^{(1)}(\lambda  ) g + R _{23}^{(1)}(\lambda  ) g     \text{   where} \label{eq:defR(1)}
\end{align}
\begin{align}\label{eq:defR211}
  R _{21}^{(1)}(\lambda  ) g &=  \frac{2^{-1}}{\lambda D(\lambda )W_0(\lambda ) }  \int _{0} ^{\lambda}  dz  (\lambda -z)^{2}   \int_  x ^{+\infty}  g(y)  \partial _z ^3 \(     {f}_3(x ,z  )   {b}_{12}(y,z )  \)      dy , \\ \label{eq:defR211b}
     R _{23}^{(1)}(\lambda  ) g &= - \frac{2^{-1}}{\lambda D(\lambda )W_0(\lambda ) }  \int _{0} ^{\lambda}  dz (\lambda -z)^{2}   \int_  {-\infty} ^x  g(y)  \partial _z ^3 \(     {f}_1(x ,z  )   {b}_{23}(y,z )  \)      dy \text{   and }
\\  \label{eq:defR2111}
  R _{ 1}^{(1)}(\lambda  ) g  & =  -\frac{1}{\lambda D(\lambda )W_0(\lambda ) }      \int_  x ^{+\infty}  g(y) \(  {f}_2(x ,\lambda   )   {b}_{13}(y,\lambda )   - \frac{\lambda ^2}{2}  {f}_2(x ,0   )   \partial _\lambda ^2 {b}_{13}(y,0 ) \)  dy.
\end{align}
We examine each of these  three terms.  We claim that
\begin{align}\label{eq:estR12}
  \left |  R _{21}^{(1)}(\lambda  ) g (x)  \right | \lesssim  \int_  x ^{+\infty}   \< x-y\> ^3      e^{-  |x-y|}  |g(y) | dy
\end{align}
which in turn, by the convolution Young's inequality,  implies
\begin{align}\label{eq:boundr121}
   \|  R _{21}^{(1)}(\lambda  ) g \| _{L^1(\R )\cap L^\infty (\R )}\lesssim  \| g\| _{L^1(\R )} .
\end{align}
To obtain \eqref{eq:estR12} notice that
\begin{align*}&
   \partial _z ^3 \(     {f}_3(x ,z  )   {b}_{12}(y,z )  \)     =    \partial _z ^3  \left [  e^{ \mu _3 (z ) (x-y)  }   {m}_3(x ,z  )
\gamma (y,z) \right ]  \text{ where} \\&  \gamma (y,z) = \(  \mu _2 (\lambda )-\mu _1 (\lambda )    \)   m_1 (y,z )  m_2 (y,z )  +    m_1 (x,z )  m_2 ' (y,z  ) - m_1 '(y,z )  m_2 (y,z)       .
\end{align*}
Since by   Lemmas \ref{lem:estimate_f1}, \ref{lem:estimate_f3} and \ref{lem:bdm2} for $ y\ge x$ for appropriate constants  $ C _{nm\ell}$  we have
\begin{align*}
\left  | \partial _x ^n \partial _y ^m \partial _z ^\ell   \left [      {m}_3(x ,z  )
\gamma (y,z) \right ]  \right |\le C _{nm\ell}
\end{align*}  if, as in our case,    $z$ is close to 0,      by    $D''(0) \neq 0$ and by $\Re \mu _3(z)\ge 1$, see \eqref{eq:detmu13},    we obtain
\begin{align*}&
  | \partial _z ^3 \(     {f}_3(x ,z  )   {b}_{12}(y,z )  \)     |\lesssim    \< x-y\> ^3     e^{-  |x-y|} ,
\end{align*}
yielding \eqref{eq:estR12}.
 Since with the same proof    we obtain also
\begin{align}\label{eq:estR22}
  \left |  R _{23}^{(1)}(\lambda  ) g (x)  \right | \lesssim  \int_  {-\infty} ^x   \< x-y\> ^3      e^{-  |x-y|}  |g(y) | dy
\end{align}
we conclude also
\begin{align}\label{eq:boundr123}
   \|  R _{23}^{(1)}(\lambda  ) g \| _{L^1(\R )\cap L^\infty (\R )}\lesssim  \| g\| _{L^1(\R )} .
\end{align}
Next we examine $R _{ 1}^{(1)}(\lambda  )  g$, which is subtler.  Since $\partial _\lambda ^2    {b}_{13}   (y,0 )= \partial _\lambda ^2  \(    e^{ \mu _2  y} {b}_{13}  \) (y,0 )$   by Corollary \ref{lem:b13at0}, we can write
\begin{align*} &
   {f}_2(x ,\lambda   )   {b}_{13}(y,\lambda )   - \frac{\lambda ^2}{2}  {f}_2(x ,0   )   \partial _\lambda ^2 {b}_{13}(y,0 )\\& =
 {f}_2(x ,\lambda   )  e^{-\mu _2(\lambda) y}  \(   e^{ \mu _2(\lambda) y} {b}_{13}(y,\lambda ) \)  - \frac{\lambda ^2}{2}  {f}_2(x ,0   )   \partial _\lambda ^2  \(    e^{ \mu _2  y} {b}_{13}  \) (y,0 ) \\& =
e^{-\mu _2(\lambda) y}   \(    {f}_2(x ,\lambda   )    e^{ \mu _2(\lambda) y} {b}_{13}(y,\lambda ) -   \frac{\lambda ^2}{2}  {f}_2(x ,0   )   \partial _\lambda ^2  \(    e^{ \mu _2  y} {b}_{13}  \) (y,0 )\)
\\& + \frac{\lambda ^2}{2}  \( e^{-\mu _2(\lambda) y} -1  \)   f _2(x ,0   )  \partial _\lambda ^2    {b}_{13}   (y,0 )  \\& =  e^{-\mu _2(\lambda) y} \int _0^\lambda \frac{(\lambda - z) ^2}{2} \partial ^3 _z  \( e^{ \mu _2(z) x}  {m}_2(x ,z   )    e^{ \mu _2(z) y} {b}_{13}(y,z )\) dz\\&  +   \frac{\lambda ^2}{2}  \( e^{-\mu _2(\lambda) y} -1  \)   f _2(x ,0   )  \partial _\lambda ^2    {b}_{13}   (y,0 )    =: A(\lambda , x,y) +B(\lambda , x,y) .
\end{align*}
The $A(\lambda , x,y)$ term   gives a contribution to \eqref{eq:defR2111} because
\begin{align*}
    \frac{1}{|\lambda D(\lambda )W_0(\lambda ) |}      \int_  x ^{+\infty}  |g(y) A(\lambda , x,y)|    dy\lesssim  (1+|x|)^3  \int_  x ^{+\infty}  |g(y)  |    dy
\end{align*}
from the fact that
\begin{align*}
  |A(\lambda , x,y)|\lesssim  (1+|x|)^3\int _{0}^{|\lambda|} (|\lambda |+\sigma)^2 d\sigma  \sim |\lambda |^3  (1+|x|)^3
\end{align*}
by   $|\partial ^n _z  \( {m}_2(x ,z   )    e^{ \mu _2(z) y} {b}_{13}(y,z )\)|\lesssim 1$ by \eqref{eq:b13_estimate_partial} and   \eqref {eq:estm21}.
Next,   Corollary \ref{lem:b13at0}  and \eqref{eq:deff2}
give
\begin{align}\nonumber
  B(\lambda , x,y) &=   - D''(0) c_0(0)   {\lambda ^2}  \( e^{-\mu _2(\lambda) y} -1  \)   f _3(x ,0   )       \\&  -b_0  c_0(0)  \frac{\lambda ^2}{2}  \( e^{-\mu _2(\lambda) y} -1  \)   f _3(x ,0   ) \phi (y)=:  B_1(\lambda , x,y)+  B_2(\lambda , x,y). \label{eq:defB2}
\end{align}
The second term is easy to treat because
\begin{align*}
    \frac{1}{|\lambda D(\lambda )W_0(\lambda ) |}      \int_  x ^{+\infty}  |g(y) B_2(\lambda , x,y)|    dy&\lesssim    \frac{|\lambda \mu _2(\lambda )| }{|  D(\lambda )W_0(\lambda ) |} |\phi ' (x)|  \int_  x ^{+\infty}  |g(y)  |  \< y \>    \phi (y) dy \\& \lesssim |\phi ' (x)|  \int_  x ^{+\infty}  |g(y)  |    dy .
\end{align*}
Let us set now
\begin{align}\label{def:tildephi}
    \widetilde{\phi} (x) :=    \frac{ \int _{-\infty} ^{x }   \phi   ( y)  dy}{ \int _{\R }     \phi   ( y)  dy}
\end{align}
and let us consider the term
\begin{align}\label{eq:defrhat1}
    \widehat{R} _{ 1}^{(1)}(\lambda  ) g  & :=  - D''(0) c_0(0) \frac{\lambda}{ 2 D(\lambda )W_0(\lambda ) }   f _3(x ,0   )     \int_  x ^{+\infty}  g(y)\( e^{-\mu _2(\lambda) y} -1  \)     dy.
\end{align}
Then we can write
\begin{align}\nonumber
    \widehat{R} _{ 1}^{(1)}(\lambda  ) g  & =  - D''(0) c_0(0) \frac{\lambda}{ 2 D(\lambda )W_0(\lambda ) }   f _3(x ,0   )     \int_  x ^{+\infty}  g(y)\( e^{-\mu _2(\lambda) y} -1  \)  (1- \widetilde{\phi} (y) )  dy   \\& \label{eq:rhats} + D''(0) c_0(0) \frac{\lambda}{ 2 D(\lambda )W_0(\lambda ) }   f _3(x ,0   )   \int_  {-\infty} ^x  g(y)\( e^{-\mu _2(\lambda) y} -1  \)   \widetilde{\phi} (y) dy   \\& -  D''(0) c_0(0) \frac{\lambda}{ 2 D(\lambda )W_0(\lambda ) }   f _3(x ,0   )   \int_  {\R }   g(y)\( e^{-\mu _2(\lambda) y} -1  \)   \widetilde{\phi} (y) dy=\sum _{j=1}^{3} \widehat{R} _{ 1}^{(1j)}(\lambda  ) g .\nonumber
\end{align}
We have, recalling $y^\pm :=\max (0,\pm y),$
\begin{align*}
  |\widehat{R} _{ 1}^{(11)}(\lambda  ) g (x)| \lesssim  \frac{|\mu _2(\lambda)|}{|\lambda |  }    |\phi ' (x)|      \int_  x ^{+\infty}  |g(y)| \< y\> O(e^{-y^+}) dy \lesssim e^{-|x|}  \int_  x ^{+\infty}  |g(y)| dy.
\end{align*}
Similarly
\begin{align*}
  |\widehat{R} _{ 1}^{(12)}(\lambda  ) g (x)| \lesssim  \frac{|\mu _2(\lambda)|}{|\lambda |  }    |\phi ' (x)|     \int_  {-\infty} ^x  |g(y)| \< y\> O(e^{-y^-} )dy \lesssim e^{-|x|} \int_  {-\infty} ^x  |g(y)| dy.
\end{align*}
So far we have proved that
\begin{align} \label{eq:smooth21smallalm}
    \left \|     \(  R  (\im  \tau )   -  \widehat{R} _{ 1}^{(13)}(\im  \tau ) \)    g     \right \|   _{   \widetilde{\Sigma}    } \le C_0          \left \|        g     \right \| _{   L^1\( \R\)    }  \text{ for all $ g\in \ker P  $ and } \tau  \in  (-\varepsilon _4, \varepsilon _4  )\backslash \{ 0  \}.
\end{align}
We finally have the following, which completes the proof of \eqref{eq:smooth21small},
\begin{align*}
   \left \|      \widehat{R} _{ 1}^{(13)}(\im  \tau )     g     \right \|   _{   \widetilde{\Sigma}    } \lesssim   \left \|      \phi '      \right \|   _{   \widetilde{\Sigma}    }   \frac{|\mu _2(\im  \tau )|}{|\tau |} \int _\R  \< y \> |g(y)  \widetilde{\phi} (y)| dy\lesssim \| g\| _{L^1(\R _-)}+ \| g\| _{L ^{1,1}(\R _+)}\sim \|g\|_{L^{1,1}_+}.
\end{align*}

\end{proof}

Having discussed what happens for  $\lambda $ close to 0,  we consider $\lambda $ away from 0.

\begin{lemma}\label{lem:smooth2large} There exists a  $C(  p, \varepsilon _4)>0$ such that
  \begin{align}\label{eq:smooth21large}
     \sup  _{ \tau \in \R  \backslash (-\varepsilon _4, \varepsilon _4  ) } \left \|     R_{\mathcal{L} } (\im  \tau  )      \right \| _{ \mathcal{L}\( L^1\( \R\),   \widetilde{\Sigma} \)   }\le   C(  p, \varepsilon _4)   .
   \end{align}

\end{lemma}
 \begin{proof}
    We have  for $R_{0}(\im \tau )= R_{-\partial _x^3 + \partial _x}(\im \tau +0 )$,
    \begin{align*}
       R  (\im  \tau  ) = \( 1-  R_{0} (\im  \tau )  \partial _x f '(\phi )   \) ^{-1} R_{0} (\im  \tau  )  .
    \end{align*}
   Here $  R_{0} (\im  \tau )  \partial _x f '(\phi ) $  is a compact operator in $L^\infty (\R)$.  Recall, see \cite[formula (A.4)]{MizuKdV2001},
that $R_{0} (\im  \tau )$ is a convolution with the function
\begin{align*}
  R_{0} (\im  \tau , x) = \left\{
                                  \begin{array}{ll}
                                    a_1(\im \tau )   e^{\mu _1(\im \tau )x}, & \hbox{ if $x\ge 0$ and  } \\
                                    -  a_2(\im \tau  )   e^{\mu _2(\im \tau )x} - a_3(\im \tau )   e^{\mu _3(\im \tau )x}, & \hbox{  if $x\le 0$.}
                                  \end{array}
                                \right.
\end{align*}
Notice that thanks to \eqref{eq:sum_aj} we have $ R_{0} (\im  \tau )  \partial _x  = \partial _x  R_{0} (\im  \tau )  $
with the latter having kernel
\begin{align*}
  \widetilde{R}_{0} (\im  \tau  , x) = \left\{
                                  \begin{array}{ll}
                                    \mu _1(\im \tau )a_1(\im \tau )   e^{\mu _1(\im \tau )x}, & \hbox{ if $x\ge 0$ and  } \\
                                    -  \mu _2(\im \tau )a_2(\im \tau  )   e^{\mu _2(\im \tau )x} -\mu _3(\im \tau ) a_3(\im \tau )   e^{\mu _3(\im \tau )x}, & \hbox{  if $x\le 0$.}
                                  \end{array}
                                \right.
\end{align*}
So if $ g=  R_{0} (\im  \tau  )  \partial _x f '(\phi )g$ we have  $g\in L^2_{a}(\R)$ for $0<a<3^{-1/2}.$ But then $g$ is an eigenfunction   in  $  L^2_{a}(\R)$ with eigenvalue $\im  \tau$
with $\tau \in \R  \backslash (-\varepsilon _4, \varepsilon _4  )$,   not possible by Propositions \ref{prop:PEGOWEI2}  and \ref{prop:noeigh}. Furthermore  from \eqref{eq:estamu}, which is relating to the smoothing properties of the group of $-\partial _x^3 + \partial _x$,     we have
\begin{align}\label{eq:rlarge}
  \|  R_{0} (\im  \tau  )  \partial _x f '(\phi ) \| _{\mathcal{L}(L^\infty (\R))} + \|  R_{0} (\im  \tau  )    \| _{\mathcal{L}(L^1 (\R) , L^\infty (\R))}  + \|  \partial _xR_{0} (\im \tau )    \| _{\mathcal{L}(L^1 (\R) , L^\infty (\R))}\xrightarrow{\tau \to \infty} 0.
\end{align}
Then  from the immersion $L^\infty (\R) \hookrightarrow \widetilde{\Sigma}$  we obtain
 \begin{align*}&
    \sup  _{ \tau \in  \R \backslash (-\varepsilon _4, \varepsilon _4  ) } \left \|     R_{\mathcal{L} } (\im  \tau )     \right \| _{ \mathcal{L}\( L^1\( \R\),   \widetilde{\Sigma} \)   } \\&  \lesssim   \sup  _{\tau  \in  \R \backslash (-\varepsilon _4, \varepsilon _4  )  } \left \|     \( 1-  R_{0} (\im  \tau )  \partial _x f '(\phi )   \) ^{-1}   \right \| _{ \mathcal{L}\( L^\infty (\R ) \)   }    \sup  _{ \tau \in \R } \|  R_{0} (\im  \tau )    \| _{\mathcal{L}(L^1 (\R) , L^\infty (\R))} <+\infty .
 \end{align*}

 \end{proof}
The above two lemmas  prove inequality \eqref{eq:smooth21}.

\bigskip

  \textit{Proof of  inequality \eqref{eq:smooth221}.} The proof
is  easier than the proof of   \eqref{eq:smooth21}. In higher energies we use
 \begin{align*}&
    \sup  _{\tau  \in  \R \backslash (-\varepsilon _4, \varepsilon _4  )  } \left \|     R_{\mathcal{L} } (\im  \tau)   \partial _x   \right \| _{ \mathcal{L}\( L^1\( \R\),   \widetilde{\Sigma} \)   }\\&   \lesssim   \sup  _{\tau \in   \R \backslash (-\varepsilon _4, \varepsilon _4  )  } \left \|     \( 1-  R_{0} (\im  \tau )  \partial _x f '(\phi )   \) ^{-1}   \right \| _{ \mathcal{L}\( L^\infty (\R ) \)   }    \sup  _{ \tau  \in \R } \|  \partial _x   R_{0} (\im  \tau)    \| _{\mathcal{L}(L^1 (\R) , L^\infty (\R))} <+\infty .
 \end{align*}
For $\tau  \in  (-\varepsilon _4, \varepsilon _4  )\backslash \{ 0  \}$ we proceed like in Lemma \ref{lem:smooth2small}. We have   for $\lambda =\im \tau $ \begin{align*}
   \(    R _{23}^{(0)}(\lambda   )   + R _{21} ^{(0)}(  \lambda )    \) \circ Q \circ   \partial _x=0 ,
\end{align*}
while we recall  $ R _{1} ^{(0)}(  \lambda )\equiv 0$.
Then, see \eqref{eq:defR(1)}, we have
\begin{align*}
  & R ^{(1)} (\lambda )   Q  g ' = R ^{(1)} (\lambda )    g '     +\sum _{j=1}^{2}\<  g, \eta _j'\>    R ^{(1)} (\lambda )  \xi _j
\end{align*}
where  obviously $|\<  g, \eta _j'\>|\lesssim \| g \|  _{L^1 \( \R  \)}$          by \eqref{eq:ker1} and by \eqref {eq:estR12} and  \eqref{eq:estR22}
\begin{align*}
   \|    R _{21}^{(1)}(\lambda  )   \xi _j \| _{L^\infty \( \R  \)}  + \|  R _{23}^{(1)}(\lambda  )    \xi _j\| _{L^\infty \( \R  \)} \lesssim    \|        \xi _j\| _{L^1 \( \R  \)}   .
\end{align*}
Using the operator  in \eqref{eq:defrhat1}, we have
\begin{align*}
  \|   R _{ 1}^{(1)}(\lambda  )   \xi _j \| _{L^\infty \( \R  \)}  &\le   \|  \( R _{ 1}^{(1)}(\lambda  ) -    \widehat{R} _{ 1}^{(1 )}(\lambda )  \)   \xi _j \| _{L^\infty \( \R  \)} + \|   \widehat{R} _{ 1}^{(1 )}(\lambda )  \xi _j \| _{L^\infty \( \R  \)}  \\& \lesssim  \|        \xi _j\| _{L^1 \( \R  \)} +\|      \widehat{R} _{ 1}^{(1 )}(\lambda )  \xi _j \| _{L^\infty \( \R  \)}    \text{  where}
\end{align*}
\begin{align*}
   \left | \widehat{R} _{ 1}^{(1)}(\lambda  )  \xi _j  \right | &  \lesssim  \frac{1 }{ |\lambda| } |f _3(x ,0   )|    \int_  x ^{+\infty}  \left |\xi _j(y)    \( e^{-\mu _2(\lambda) y} -1  \)  \right |   dy
\\& \lesssim   \frac{|\mu _2(\lambda)| }{ |\lambda| }    |\phi ' (x)|     \int_   {\R } | y \xi _j (y) |    dy\lesssim 1.
\end{align*}
Next, integrating by parts  and using the identity
\begin{align*}
   -   {f}_3(x ,z  )   {b}_{12}(x,z ) -   {f}_1(x ,z  )   {b}_{23}(x,z )  +   {f}_2(x ,z  )   {b}_{13}(x,z )\equiv 0  \end{align*}
which eliminates a boundary term, we have  \small
\begin{align*}
  R ^{(1)} (\lambda )      g ' & =   \frac{2^{-1}}{\lambda D(\lambda )W_0(\lambda ) }  \int _{0} ^{\lambda}  dz (\lambda -z)^{2}   \int_  {-\infty} ^x  g(y)  \partial _z ^3 \(     {f}_1(x ,z  )   \partial _y{b}_{23}(y,z )  \)      dy\\& +\frac{2^{-1}}{\lambda D(\lambda )W_0(\lambda ) }  \int _{0} ^{\lambda} dz   (\lambda -z)^{2}   \int_  x ^{+\infty}  g(y)  \partial _z ^3 \(     {f}_2(x ,z  )  \partial _y{b}_{13}(y,z ) \)  dy\\&
 - \frac{1}{\lambda D(\lambda )W_0(\lambda ) }     \int_  x ^{+\infty}  g(y) \widetilde{ \gamma} (x,y, \lambda )     dy  =:\sum _{j=1}^{3} \mathcal{R}_j(\lambda) g \text{ where}
 \\
 \widetilde{ \gamma} (x,y, \lambda )& = {f}_2(x ,\lambda   )   \partial _y{b}_{13}(y,\lambda )   - \frac{\lambda ^2}{2}  {f}_2(x ,0   )    \partial _y\partial _\lambda ^2 {b}_{13}(y,0 ) .
\end{align*}\normalsize
By the  arguments  applied to $  R _{21}^{(1)}(\lambda  )$ and   $  R _{23}^{(1)}(\lambda  )$ it is easy to show that
\begin{align*}
   \| \mathcal{R}_j(\lambda) g \| _{\widetilde{\Sigma}}\lesssim  \| g\| _{L^1(\R )} \text{ for $j=2,3$.}
\end{align*}
As we show now the term $\mathcal{R}_1(\lambda) g$  is simpler than $  R _{ 1}^{(1)}(\lambda  ) g$. We write
\begin{align} & \nonumber
    \widetilde{ \gamma} (x,y, \lambda ) =
-\mu _2(\lambda) e^{-\mu _2(\lambda)y}    \(  {f}_2(x ,\lambda   )   e^{ \mu _2(\lambda)y}  {b}_{13}(y,\lambda )   - \frac{\lambda ^2}{2}  {f}_2(x ,0   )   \partial _\lambda ^2 {b}_{13}(y,0 )   \) \\& \nonumber +   e^{-\mu _2(\lambda)y}    {f}_2(x ,\lambda   ) \partial _y \(  e^{ \mu _2(\lambda)y}  {b}_{13}(y,\lambda )\)  - \frac{\lambda ^2}{2}  {f}_2(x ,0   )    \partial _y\partial _z ^2   \( e^{ \mu _2(z)y}  {b}_{13}(y,z )\) |_{z=0} \\& -  \mu _2(\lambda) e^{-\mu _2(\lambda)y}  \frac{\lambda ^2}{2}    f _2(x ,0   )   \partial _\lambda ^2 {b}_{13}(y,0 ) =:\sum _{j=1}^{3}\widetilde{ \gamma}_j (x,y, \lambda ) \label{eq:gammatilde3}
\end{align}
where we use $\partial ^n _y\partial _z ^2   \( e^{ \mu _2(z)y}  {b}_{13}(y,z )\) |_{z=0}=  \partial ^n _y \partial _\lambda ^2 {b}_{13}(y,0 )$ for any $n$ by Corollary \eqref{lem:b13at0}.
Then  \small
\begin{align*} &
 \left | \frac{1}{\lambda D(\lambda )W_0(\lambda ) }  \int_  x ^{+\infty}  g(y)  \widetilde{ \gamma}_1 (x,y, \lambda ) dy \right | \\ &\lesssim \frac{1}{|\lambda |^2}  \int _{0}^{|\lambda|}  |\lambda |^2  \int_  x ^{+\infty} dy |g(y)| \left | \partial _z ^3 \(   e^{\mu (z) 2 x} {m}_2(x ,z   )   e^{ \mu _2(\lambda)y}  {b}_{13}(y,z )   \)    \right | dz   \lesssim  |\lambda | \< x \> ^3  \int_  x ^{+\infty}    |g(y)| dy
\end{align*} \normalsize
by  \eqref{eq:m1_estimate_partial}  and \eqref{eq:b13_estimate_partial}. For the same reasons \small
\begin{align*} &
 \left | \frac{1}{\lambda D(\lambda )W_0(\lambda ) }  \int_  x ^{+\infty}  g(y)  \widetilde{ \gamma}_2 (x,y, \lambda ) dy \right | \\ &\lesssim \frac{1}{|\lambda |^3}  \int _{0}^{|\lambda|}  |\lambda |^2  \int_  x ^{+\infty} dy |g(y)| \left | \partial _z ^3 \(   e^{\mu (z) 2 x} {m}_2(x ,z   ) \partial _y \(  e^{ \mu _2(\lambda)y}  {b}_{13}(y,z )   \)  \)    \right | dz  \\&  \lesssim    \< x \> ^3  \int_  x ^{+\infty}    |g(y)| dy .
\end{align*} \normalsize
Finally, thanks to the $\mu _2(\lambda )$ factor appearing in  \eqref{eq:gammatilde3} and conspicuously absent  in \eqref{eq:defB2}, we have
\begin{align*} &
 \left | \frac{1}{\lambda D(\lambda )W_0(\lambda ) }  \int_  x ^{+\infty}  g(y)  \widetilde{ \gamma}_3 (x,y, \lambda ) dy \right | \\ &\lesssim \frac{  | \mu _2(\lambda ) \lambda ^2| }{|\lambda |^3}    |\phi ' (x)|     \int_  x ^{+\infty} dy |g(y) (2D''(0) +b_0 \phi (y))        |         \lesssim   |\phi ' (x)|    \int_  x ^{+\infty}    |g(y)| dy.
\end{align*}
  The above guarantee the following thus   completing the proof of  \eqref{eq:smooth221}  and thus of   Proposition  \ref{prop:smooth2},
\begin{align*}
  \|   R ^{(1)} (\lambda )      g '     \| _{\widetilde{\Sigma}} \le C(p) \| g\| _{L^1(\R )}  \text{ for all $g\in L^1(\R )$ and all $\lambda   \in  \im (-\varepsilon _4, \varepsilon _4  ) \backslash \{ 0  \}$.}
\end{align*}

\qed

\section{Proof of Proposition \ref {prop:noeigh} }\label{sec:noeigh}

In this section, we prove what is left of    Proposition \ref {prop:noeigh}, that is    that $\mathcal{L}_{[c]}$ has no nonzero purely imaginary eigenvalue.
First of all, it is elementary to see by scaling that it suffices to consider case $c=1$. Next,
we  observe the following.
\begin{lemma}\label{lem:L_cadj}
    Let $\lambda>0$.
    Then $\im \lambda$ is an eigenvalue of $\mathcal{L} $ if and only if it is an eigenvalue of $L_ +\partial_x$.
\end{lemma}

\begin{remark}
    Since if $\im \lambda$ ($\lambda\in \R$) is an eigenvalue of $\mathcal{L} $ with an eigenfunction $\psi$, then $-\im \lambda$ is also an eigenvalue with the eigenfunction $\overline{\psi}$.
    Thus, it suffices to show the nonexistence of an embedded eigenvalue $\im \lambda$ with $\lambda>0$.
\end{remark}

\begin{proof}
    Let $\im \lambda$ be an eigenvalue with associated eigenfunction $\varphi$.
    From the results in Section \ref{sec:jost}, $\varphi$ and its derivatives  decay exponentially.
    Now, since $\varphi$ satisfies $\partial_x L_+\varphi=\im \lambda\varphi$, setting $\psi=L_+\varphi$, we have
    \begin{align*}
        L_+\partial_x \psi=L_+ (\partial_x L_+ \varphi)=\im \lambda L_+ \varphi=\im \lambda \psi.
    \end{align*}
    Thus, $\psi$ is an eigenvalue of $L_+\partial_x$ associated to $\im \lambda$.
    Next, if $\psi$ satisfies $L_+ \partial_x \psi=\im \lambda \psi$, then $\varphi:=\partial_x \psi$ satisfies
    \begin{align*}
        \partial_x L_+ \varphi=\partial_x L_+ \partial_x \psi=\im \lambda \partial_x \psi=\im \lambda \varphi.
    \end{align*}
    Therefore, we have the conclusion.
\end{proof}

We prove the non-existence of non-zero purely imaginary eigenvalue of $L_+\partial_x$ by using the virial functional   in \cite{MaMegeneral2008}.
\begin{proposition}
    $\mathcal{L} $ has no eigenvalue in $\im \R\setminus\{0\}$.
\end{proposition}

\begin{proof}
    From Lemma \ref{lem:L_cadj}, it suffices to show that $L_+\partial_x$ does not have an eigenvalue in $\im \R\setminus\{0\}$.
     Assume $\im \lambda$ is an eigenvalue of $L_+\partial_x$ and let $\psi$ be the corresponding eigenfunction.
    We set $$\mu(x)=-\phi'(x)/\phi(x)=\frac{p-1}{2}\tanh\(\frac{p-1}{2}x\).$$
    Then, we have
    \begin{align*}
        \Re \int \overline{\psi} (L_+\psi')\mu \,dx=\Re \im \lambda \int |\psi|^2\mu\,dx=0.
    \end{align*}
    Now, setting $w=\psi/\phi$, we have
    \begin{align*}
    \Re \int \overline{\psi} (L_+\psi')\mu\,dx=&
        \Re\int \psi'\overline{L_+(\psi\mu)}dx\\&
        =\Re\int (w'\phi +w\phi ')\overline{(2w'\phi ''+w''\phi ')}dx\\&
          =\Re\int (2|w'|^2\phi \phi ''+w'\overline{w''}\phi \phi '+2w \overline{w'}\phi '\phi ''+w\overline{w''}(\phi ')^2)dx\\&
         =\Re\int (2|w'|^2\phi \phi ''-\frac{1}{2}|w'|^2(\phi \phi ')'-(w\overline{w'})'(\phi ')^2+w\overline{w''}(\phi ')^2)dx\\&
        =\Re\int (2|w'|^2\phi \phi ''-\frac{1}{2}|w'|^2(\phi ')^2-\frac{1}{2}|w'|^2\phi \phi ''-|w'|^2(\phi ')^2)dx\\&
        =-\frac{3}{2}\int |w'|^2\(-\phi \phi ''+(\phi ')^2\)dx.
    \end{align*}
    From $\phi''=\phi-\phi^p$ and $(\phi')^2=\phi^2-\frac{2}{p+1}\phi^{p+1}$, we have
    \begin{align*}
        -\phi \phi ''+(\phi ')^2=\frac{p-1}{p+1}\phi^{p+1}.
    \end{align*}
    Therefore, we have
    \begin{align}\label{eq:virialMM08}
        0=-\frac{3(p-1)}{2(p+1)}\int_{\R} \left|\(\frac{\psi(x)}{\phi(x)}\)'\right|^2\phi(x)^{p+1}\,dx.
    \end{align}
    This shows that there exists $C$ s.t. $\psi(x)=C\phi(x)$.
    However, this will imply that $\psi$ is an eigenfunction of $0$ and contradicts our assumption.
\end{proof}

\appendix
\section{Appendix: conservation of $\mathbf{Q}$}\label{app:A}
In this section, we prove the conservation of $\mathbf{Q}$, which is well-known when $f\in C^3$ but seems to need verification for small $p$.
For the proof, we follow Ozawa's argument \cite{Ozawa06CV}.

    \begin{proposition}\label{prop:consQ}
    Let $T^*>0$ and  $u\in C([0,T^*),H^1(\R))$ be a solution of \eqref{gKdVintegral}.
    Then, we have $\mathbf{Q}(u(t))=\mathbf{Q}(u_0)$ for all $t\geq 0$.
\end{proposition}

\begin{proof}
    We have
    \begin{align}
        \|u(t)\|_{L^2}^2&=\|e^{t\partial_x^3}u(t)\|_{L^2}^2\nonumber\\
        &=\|u_0 - \int_0^t e^{s\partial_x^3}\partial_x(f(u(s)))\,ds \|_{L^2}^2\nonumber\\
        &=\|u_0\|_{L^2}^2-2\int_0^t\<u_0,e^{s\partial_x^3}\partial_x(f(u(s)))\>\,ds\nonumber\\&\quad+2\int_0^t\<e^{s_1\partial_x^3}\partial_x(f(u(s_1))),\int_0^{s_1}e^{s_2\partial_x^3}\partial_x(f(u(s_2)))\,ds_2\>ds_1,\label{eq:Ozarg1}
    \end{align}
    where for the last line, we have used
    \begin{align*}
        &\int_{[0,t]^2}\<e^{s_1\partial_x^3}\partial_x(f(u(s_1))),e^{s_2\partial_x^3}\partial_x(f(u(s_2)))\>\,ds_1ds_2\\&
          =\(\int_{(s_1,s_2)\in[0,t],s_1\geq s_2}
           +\int_{(s_1,s_2)\in[0,t],s_2\geq s_1}\)
         \<e^{s_1\partial_x^3}\partial_x(f(u(s_1))),e^{s_2\partial_x^3}\partial_x(f(u(s_2)))\>\,ds_1ds_2\\&
         =2\int_{(s_1,s_2)\in[0,t],\ s_1\geq s_2}\<\partial_x(f(u(s_1))),e^{s_2\partial_x^3}\partial_x(f(u(s_2)))\>\,ds_1ds_2.
    \end{align*}
    Now, from \eqref{gKdVintegral}, the last term in \eqref{eq:Ozarg1} can be rewritten as
    \begin{align}
        &2\int_0^t\<e^{s_1\partial_x^3}\partial_x(f(u(s_1))),\int_0^{s_1}e^{s_2\partial_x^3}\partial_x(f(u(s_2)))\,ds_2\>ds_1\nonumber\\&=2\int_0^t\<e^{s_1\partial_x^3}\partial_x(f(u(s_1))),u_0-e^{s_1\partial_x^3}u(s_1)\>\,ds_1\nonumber\\&
        =2\int_0^t\<u_0,e^{s_1\partial_x^3}\partial_x(f(u(s_1)))\>\,ds_1-2\int_0^t\<\partial_x(f(u(s_1))),u(s_1)\>\,ds_1.\label{eq:Ozarg2}
    \end{align}
    Finally, since $\partial_x (f(u)) u= \partial_x g$, where $g=\frac{p}{p+1}|u|^{p+1}$ in the case $f(u)=|u|^{p-1}u$ and $g=\frac{p}{p+1}|u|^p u$ in the case $f(u)=|u|^p$.
    Thus, we see that the last term in \eqref{eq:Ozarg2} vanishes.
    Therefore, we have the conclusion.
\end{proof}

\section*{Acknowledgments}
C. was supported   by the Prin 2020 project \textit{Hamiltonian and Dispersive PDEs} N. 2020XB3EFL.
M.  was supported by the JSPS KAKENHI Grant Number 23H01079 and 24K06792.

	\begin{footnotesize}
		\bibliography{references}
		\bibliographystyle{plain}
	\end{footnotesize}

Department of Mathematics, Informatics and Geosciences,  University
of Trieste, via Valerio  12/1  Trieste, 34127  Italy.
{\it E-mail Address}: {\tt scuccagna@units.it}

Department of Mathematics and Informatics,
Graduate School of Science,
Chiba University,
Chiba 263-8522, Japan.
{\it E-mail Address}: {\tt maeda@math.s.chiba-u.ac.jp}

\end{document}